\documentclass{CUP-JNL-FMP}

\usepackage{graphicx}
\usepackage{multicol,multirow}
\usepackage{amsmath,amssymb,amsfonts}
\usepackage{mathrsfs}
\usepackage{amsthm}
\usepackage{rotating}
\usepackage{appendix}
\usepackage[numbers]{natbib}
\usepackage{ifpdf}
\usepackage[T1]{fontenc}
\usepackage{newtxtext}
\usepackage{tikz}
\usepackage{newtxmath}
\usepackage{textcomp}
\usepackage[colorlinks,allcolors=weblmcolor]{hyperref}

\articletype{RESEARCH ARTICLE}

\usepackage{amsmath,amssymb,latexsym,amsthm,enumerate,color}
\numberwithin{equation}{section}
\usepackage{etoolbox}
\usepackage{dynkin-diagrams}
\def\row#1/#2!{#1_{\IfStrEq{#2}{}{n}{#2}} & \dynkin{#1}{#2}\\}

\newtheorem{theorem}{Theorem}[section]
\newtheorem{lemma}[theorem]{Lemma}
\newtheorem{proposition}[theorem]{Proposition}
\newtheorem{corollary}[theorem]{Corollary}

\newtheorem{notation}[theorem]{Notation}
\newtheorem{hypothesis}[theorem]{Hypothesis}
\newtheorem{remark}[theorem]{Remark}

\newtheorem{question}{Question}

\theoremstyle{definition}

\newtheorem{Strategy}[theorem]{Strategy}

\def\Om{\Omega}
\def\Ga{\Gamma}

\def\Si{\Sigma}
\def\si{\sigma}
\def\depth{depth\ }

\def\Alt{{\rm Alt}}
\def\Hom{{\rm Hom}}
\def\Aut{{\rm Aut}}
\def\Sym{{\rm Sym}}
\def\soc{{\rm soc}}

\def\Inv{{\rm Inv}}
\def\Core{{\rm Core}}
\def\core{{\rm Core}}
\def\Out{{\rm Out}}
\def\Ker{{\rm Ker}}
\def\Otwo{{(2),\Om}}
\def\Stwo{{(2),\Si}}
\def\hT{\widehat{T}}
\def\hS{\widehat{S}}
\def\ol{\overline}
\def\wt{\widetilde}

\begin{document}
	
	\begin{Frontmatter}
		\title[Totally $2$-closed groups]{Totally $2$-closed finite groups with trivial Fitting subgroup}
		\author[1]{Majid Arezoomand}
		\author[2]{Mohammad A. Iranmanesh}
		\author[3]{Cheryl E Praeger}
		\author[4]{Gareth Tracey}
		
		\authormark{}
		
		\address[1]{\orgname{University of Larestan}, \orgaddress{\city{Larestan}, \postcode{74317-16137}, \country{Iran}};\email{arezoomand@lar.ac.ir}}
		
		\address[2]{\orgname{Yazd University}, \orgaddress{\city{Yazd}, \postcode{89195-741}, \country{Iran}};\email{iranmanesh@yazd.ac.ir}}
		
		\address[3]{\orgname{University of Western Australia}, \orgaddress{\street{35 Stirling Highway}, \city{Crawley}, \postcode{WA 6009}, \country{Australia}};\email{cheryl.praeger@uwa.edu.au}}
		
		\address[4]{\orgname{School of Mathematics, University of Birmingham}, \orgaddress{\street{Edgbaston}, \city{Birmingham}, \postcode{B15 2TT}, \country{United Kingdom}};\email{g.tracey@bham.ac.uk}}

		
		
		\abstract{
			A group $G$ is said to be totally $2$-closed if in each of its faithful permutation representations, say on a set $\Om$, $G$ is the largest subgroup of $\Sym(\Om)$ which leaves invariant each of the $G$-orbits for the induced action on $\Om\times \Om$.
			We prove that there are precisely $47$ finite totally $2$-closed groups with trivial Fitting subgroup. Each of these groups is a direct product of pairwise non-isomorphic sporadic simple groups, with the direct factors coming from the Janko groups $\mathrm{J}_1, \mathrm{J}_3$ and  $\mathrm{J}_4$, together with $\mathrm{Ly}, \mathrm{Th}$ and the Monster $\mathbb{M}$. These are the first known examples of insoluble totally $2$-closed groups.  As a by-product of our methods, we develop several tools for studying $2$-closures of transitive permutation groups -- a vital tool in the study of representations of finite groups as automorphism groups of digraphs. We also prove a dual to a 1939 theorem of Frucht from Algebraic Graph Theory.}

	\end{Frontmatter}
	
	\maketitle

	
	\section{Introduction}\label{intro}
	
	In 1969, Wielandt~\cite[Definition 5.3]{Wielandt} introduced the concept of $2$-closure for a permutation group $G$ on a set $\Om$. The \emph{$2$-closure} $G^{(2)}$ of $G$ is the set of all  $g\in\Sym(\Om)$ such that $g$ leaves invariant each $G$-orbit in the induced $G$-action on ordered pairs from $\Om$. The $2$-closure $G^{(2)}$ is a subgroup of $\Sym(\Om)$ containing $G$ \cite[Theorem 5.4]{Wielandt}, and $G$ is said to be \emph{$2$-closed} if $G^{(2)}=G$. 


	It is important to note that the $2$-closure of a permutation group $G$ on $\Om$ is not intrinsic to the group $G$. It can depend on the nature of the action on $\Om$. For example, the symmetric group $\Sym(3)$ acts faithfully and intransitively on
	$\{1,2,3,4,5\}$ with orbits $\{1,2,3\}$ and $\{4,5\}$, and in this action its $2$-closure is $\Sym(3)\times C_2$; while $\Sym(3)$ is $2$-closed in its natural action on $\{1,2,3\}$. In 2016,
	D.~F.~Holt\footnote{\texttt{mathoverflow.net/questions/235114/2-closure-of-a-permutation-group}}
	suggested a stronger concept independent of the permutation representation.  
	This was adopted and explored by the first author with Abdollahi  in \cite{AA} 
	where it was called  \emph{total $2$-closure}. To explain the concept we first enrich the notation for $2$-closure by writing $G^\Otwo$ for the $2$-closure of a permutation group $G\leq \Sym(\Om)$. We then define a group $G$ to be \emph{totally $2$-closed} if  $G^\Otwo=G$ whenever $G$ is faithfully represented as a permutation group on $\Om$. It was proved in \cite[Theorem 1]{AA} that a finite totally $2$-closed group has cyclic centre, and moreover~\cite[Theorem 2]{AA} that a finite nilpotent group is totally $2$-closed if and only if it is either
	cyclic or a direct product of a generalised quaternion group with a cyclic group of odd order.

	The first and fourth authors, with Abdollahi undertook a general study of total $2$-closure in \cite{AAT}.   They proved \cite[Theorem B]{AAT} that a finite soluble group is totally $2$-closed if and only if it is nilpotent, and hence is one of the groups identified in \cite[Theorem 2]{AA}.  Further they showed in \cite[Theorem A]{AAT} that the Fitting subgroup of a totally $2$-closed finite group is itself totally $2$-closed  (and hence lies in the family just described). At the time of their writing of the paper \cite{AAT}, no examples of insoluble totally $2$-closed groups were known. They studied the structure of an insoluble example $G$ of smallest order, showing that $G$ modulo its cyclic centre $Z(G)$ has a unique minimal normal subgroup which is nonabelian, \cite[Proposition 4.1]{AAT}. 
	
	The purpose of this paper is to seek out finite insoluble totally $2$-closed groups $G$. We focus on the case where the Fitting subgroup $F(G)$ is trivial, or equivalently, the case where $G$ has no nontrivial abelian normal subgroup. We show first that such a group $G$ is a direct product of  nonabelian simple groups, each of which is itself totally $2$-closed (see Theorem \ref{thm:issimple}). 
	In light of this result, we decided that our next challenge should be to find all totally $2$-closed finite simple groups. The $2$-closure condition seemed so strong that we initially believed that we would find no examples. Surprisingly (to us) such groups do indeed exist, but they are rare, and the only examples are six of the sporadic simple groups.
	
	\begin{theorem}\label{thm:J}
		Let $T$ be a nonabelian finite simple group. Then $T$ is totally $2$-closed if and only if $T$ is one of the groups $\mathrm{J}_1$, $\mathrm{J}_3$, $\mathrm{J}_4,\mathrm{Ly}, \mathrm{Th}$ or $\mathbb{M}$.
	\end{theorem}
	
	Theorem \ref{thm:J} gives us the building blocks of finite totally $2$-closed groups with trivial Fitting subgroup. We were able to use this result as the basis of an inductive approach yielding the following complete classification of all such groups.
	
	\begin{theorem}\label{thm:classification}
		Let $G$ be a non-trivial finite group with trivial Fitting subgroup. Then $G$ is totally $2$-closed if and only if each of the following holds:
		\begin{enumerate}[\upshape(1)]
			\item $G=T_1\times\hdots\times T_r$, where the $T_i$ are nonabelian finite simple groups and $r\le 5$; 
			\item $T_i\not\cong T_j$ for each $i\neq j$; and 
			\item One of the following holds:
			\begin{enumerate}[\upshape(i)]
				\item $T_i\in\{\mathrm{J}_1, \mathrm{J}_3, \mathrm{J}_4, \mathrm{Th}, \mathrm{Ly}\}$ for each $i\le r$; or
				\item $T_i\in \{\mathrm{J}_1, \mathrm{J}_3, \mathrm{J}_4,\mathrm{Ly}, \mathbb{M}\}$ for each $i\le r$.
			\end{enumerate}
		\end{enumerate}
	\end{theorem}

\begin{corollary}\label{cor1}
There are precisely $47$ finite totally $2$-closed groups with trivial Fitting subgroup.	
\end{corollary}

	Proving these results involved a wide variety of methods, and we needed to develop several tools which we hope will assist future analyses of finite totally $2$-closed groups. More generally, we also hope that our results will add significantly to the understanding of $2$-closures of transitive permutation groups. In particular, in Section \ref{ImprimitiveSection}, we prove that if $G$ is an imprimitive transitive permutation group, and its action on a set of non-trivial blocks is faithful and $2$-closed, then either $G$ is $2$-closed, or the $2$-closure of $G$ has a very particular structure (see Lemma \ref{Crucial} and Proposition \ref{RedProp}). 
	
The study of  $2$-closures of transitive permutation groups has important applications
in graph theory, the basis of which we discuss in Subsection~\ref{s:graph}. Perhaps because of these applications,  such as to the \emph{Polycirculant conjecture} (that each vertex-transitive graph has an automorphism of prime order with no fixed vertices), the theory of $2$-closures of transitive permutation groups has become a well-developed branch of both group theory and graph theory over the past fifty years (see, for example, \cite[Section 1.3.4]{BG}  and the recent survey \cite{AAS}). We hope that the general theory we develop through this paper, as well as our main results, will play an important role in making further progress in this area.

	
\subsection{Link between $2$-closure and graphs, and an application}\label{s:graph}	
	The concept of  $2$-closure is a very combinatorial notion, and there is indeed a close relationship between $2$-closed permutation groups and digraphs: each subset $A\subseteq \Om\times\Om$ may be viewed as the arc-set of a digraph $\Ga=(\Om,A)$  with vertex-set $\Om$
	and with a directed edge (an arc) from $\alpha$ to $\beta$ for each $(\alpha,\beta)\in A$. The automorphism group $\Aut(\Ga)$ of $\Ga$  is  defined as the subgroup of all permutations of $\Om$ which leave the arc-set $A$ invariant. In particular, each element of $(\Aut(\Ga))^{(2),\Omega}$ leaves $A$ invariant and hence is an automorphism of $\Ga$. Thus  $(\Aut(\Ga))^{(2),\Omega}\subseteq \Aut(\Ga)$, and the reverse inclusion holds by the definition of $2$-closure. Therefore $(\Aut(\Ga))^{(2),\Omega}= \Aut(\Ga)$. The same argument also shows that $(\Aut(\Ga))^{(2),\Omega}= \Aut(\Ga)$ when $\Gamma$ is an undirected graph. Thus, to summarise, \emph{for each (di)graph $\Ga$ with vertex set $\Om$, the automorphism group $\Aut(\Ga)$ is a $2$-closed subgroup of $\Sym(\Om)$.} In  fact, for each permutation group $G\leq\Sym(\Om)$, the $2$-closure $G^{(2),\Omega}$ is equal to the intersection of the groups $\Aut(\Ga)$ over all digraphs $\Ga=(\Om,A)$ for which the arc-set $A$ is a $G$-orbit in $\Om\times\Om$.

	In particular, if $G\leq\Sym(\Om)$ is transitive on $\Om$ and if $A$ is a $G$-orbit in $\Om\times\Om$ not equal to the diagonal $\{(\alpha,\alpha)\mid \alpha\in\Om\}$, then the digraph $\Ga=(\Om, A)$ is called a \emph{$G$-orbital digraph}, and the group $G$ is admitted as a vertex-transitive and  arc-transitive subgroup of $\Aut(\Ga)$. Thus, in this case, the $2$-closure $G^{(2),\Omega}$ is the intersection of the automorphism groups of all the $G$-orbital digraphs. The study of $G$-orbital digraphs goes back to the  work of D. G. Higman~\cite{Higman} at approximately the same time as Wielandt's work on closures in the late 1960s.
	

	

	Before describing our approach to proving Theorems \ref{thm:J} and \ref{thm:classification}, we would like to draw attention to a graph theoretic  consequence of our results, which may be of independent interest. In 1939, Frucht \cite{Frucht} confirmed a conjecture of K\"{o}nig that every finite group $G$ can be written as the automorphism group of a finite simple undirected graph $\Gamma$. 
	We will say that a faithful permutation representation $G\hookrightarrow G^{\Omega}\le \Sym(\Omega)$ of $G$ on a finite set $\Omega$ is a \emph{Frucht representation} of $G$ if there exists a set $E$ of unordered pairs from $\Omega$ such that $G^{\Omega}$ is the automorphism group of the finite simple undirected graph $\Gamma=\Gamma(\Omega,E)$. Since $\Aut(\Gamma)$ is $2$-closed, Theorem \ref{thm:classification} implies the following, which can be viewed as a dual to Frucht's theorem in the case where $F(G)=1$.

	\begin{corollary}\label{cor:Frucht}
		Let $G$ be a non-trivial finite group with trivial Fitting subgroup, and assume that every  faithful permutation representation of $G$ is a Frucht representation. Then $G=T_1\times\hdots\times T_r$, where the $T_i$ are pairwise non-isomorphic finite simple groups; $r\le 5$; 
		and either
		\begin{enumerate}[\upshape(i)]
			\item for each $i\le r$, $T_i\in\{\mathrm{J}_1, \mathrm{J}_3, \mathrm{J}_4, \mathrm{Th}, \mathrm{Ly}\}$; or
			\item for each $i\le r$, $T_i\in \{\mathrm{J}_1, \mathrm{J}_3, \mathrm{J}_4,\mathrm{Ly}, \mathbb{M}\}$.
		\end{enumerate}
	\end{corollary}

We mention a few open questions about total closure. Firstly, since we have considered only groups with trivial Fitting subgroup in this paper, we ask:

\begin{question}\label{Q1}
Are there any finite insoluble totally $2$-closed groups with nontrivial Fitting subgroup? If so, can we classify them?
\end{question}

\noindent
One of the ultimate goals of this programme of work is to classify all finite totally $2$-closed groups. Given the results in the current paper, this goal would be realised with a complete answer to Question~\ref{Q1}. See Section~\ref{sec:concluding} for some remarks on the obstacles involved in answering this question. 

The analogous concept of total $k$-closure may be defined for arbitrary positive integers $k$, namely $G$ is \emph{totally $k$-closed} if, for every faithful representation $G\leq \Sym(\Omega)$,
$G$ is equal to the set of all elements $g\in\Sym(\Om)$ such that $g$ leaves invariant each $G$-orbit in the induced $G$-action on ordered $k$-tuples from $\Om$. This concept  has been explored by Churikov and the third author in \cite{CP}. For each $k\geq 2$,
it follows from \cite[Theorem 5.8]{Wielandt} that, if $G$ is totally $(k-1)$-closed then $G$ is also totally $k$-closed. On the other hand,  by \cite[Theorem 1.2]{CP}, there are infinitely many finite totally $k$-closed groups which are not totally $(k-1)$-closed, and this is true even for abelian $p$-groups for each prime $p$. The paper \cite{CP} gives a discussion of the topic and poses three open problems about  total $k$-closure (where $k\geq3$) for finite nilpotent groups, soluble groups, and sporadic simple groups  -- all problems which have been resolved  for $k=2$, in \cite{AA, AAT} and the current paper. It follows from \cite[Theorem 5.11]{Wielandt} that an alternating group $A_n$ (with $n\geq5$) is not totally $k$-closed for any $k\leq n-2$, so we further ask:

\begin{question}
Are there any finite simple groups of Lie type which are totally $3$-closed? If so, find them all.
\end{question}

Moreover, for any faithful representation of a finite simple group $G$, there is a $G$-orbit on which $G$ acts nontrivially and hence faithfully; applying \cite[Theorem 5.12]{Wielandt} to the $G$-action on this orbit we deduce that $G$ is totally $n$-closed, where $n=|G|$. Thus we ask: 

\begin{question}
Is there a fixed integer $k$ such that all finite simple groups are totally $k$-closed? 
Indeed, for a given finite simple group $G$, what is the least integer $k$ such that $G$ is totally $k$-closed? 
\end{question}

\subsection{Structure of the paper}\label{sub:structure}

In the remained of this section, we outline our approach to the proofs of Theorems~\ref{thm:J} and \ref{thm:classification}, and detail the structure of the paper. The fundamental starting point is Wielandt's Dissection Theorem \cite[Dissection Theorem 6.5]{Wielandt} which for convenience we state as Theorem~\ref{thm:W2}. It is an invaluable tool, and in particular it underpins our proof of Theorem~\ref{thm:issimple}, which we present in Section~\ref{sec:WDT}.
	
	As mentioned above, Theorem~\ref{thm:issimple} reduces our classification of finite totally $2$-closed groups $G$ with trivial Fitting subgroup to the case where $G$ is a direct product of nonabelian simple groups,
	  which are `independent' of each other in the sense that each of the simple direct factors is not isomorphic to a section of the direct product of the others. (We comment on this condition in relation to the sporadic simple groups in Remark~\ref{rem:section}.)  Our next step is to reduce further: in Proposition \ref{PreTransRed}, we show that such a group $G$ is totally $2$-closed if and only if no simple direct factor of $G$ admits a non-trivial factorisation, and each (not necessarily faithful) \emph{transitive} permutation representation of $G$ is $2$-closed. 
	
	This suggests two natural avenues to pursue. The first is to use
	the classification of the maximal factorisations of the nonabelian simple groups in~\cite{LPS2} to eliminate many families of simple groups from being direct factors of $G$. We do this in Section \ref{sub:factns}, where we show, in particular, that the simple direct factors of $G$ are either sporadic groups or exceptional groups of Lie type.
	
	Our second natural avenue to pursue concerns transitive actions of such groups $G$. Strong restrictions on the $2$-closures of primitive almost simple groups were obtained in \cite{LPS}. With these restrictions in mind for the case where $G$ is simple, we develop powerful reduction theorems in Section \ref{ImprimitiveSection} for general transitive imprimitive permutation groups $G$ which induce a faithful $2$-closed action on a set of non-trivial blocks of imprimitivity: we show that either $G$ itself is $2$-closed, or the $2$-closure of $G$ has a very particular structure. This provides further criteria to determine whether or not such permutation groups are $2$-closed (see, in particular, Proposition \ref{MainCor1}).

	The tools developed in Sections \ref{sec:WDT} and \ref{ImprimitiveSection} are sufficient  to begin our analysis of the $2$-closures of permutation representations of the finite simple groups. In Sections \ref{SporadicSection},  and \ref{sec:ExGroups}, we deal with the sporadic groups other than $\mathrm{J}_4$, and the exceptional groups of Lie type, respectively. The case of $G=\mathrm{J}_4$ is delicate, and requires a nuanced approach, see Section \ref{J4Section}. These three sections complete the proof of Theorem \ref{thm:J}.
	
	Section \ref{DirectProductsSection} is reserved for the study of $2$-closures of direct products of finite simple groups, and the proof of Theorem \ref{thm:classification}.
	
	Finally, in Section \ref{sec:concluding}, we make some concluding remarks on the question of classifying the finite insoluble totally $2$-closed groups with non-trivial Fitting subgroup. We also prove a useful reduction theorem for this problem.
	
	\vspace{0.5cm}
	\noindent\textbf{Notation}: Most of our notation is standard. For a finite group $G$, we write $Z(G)$, $F(G)$, $\Phi(G)$ and $[G,G]$ for the centre, Fitting subgroup, 
	Frattini subgroup, and derived subgroup of $G$, respectively.
	
	For a set $\Omega$, we write $\Sym(\Omega)$ and $\Alt(\Omega)$ for the symmetric and alternating group on $\Omega$ respectively. If a finite group $G$ acts on $\Omega$, we write $G^{\Omega}$ for the image of the induced action of $G$ on $\Omega$. 
	
	For group names, we use conventions from the Atlas \cite[especially page xx]{Atlas}, except that we write $\Alt(n)$ and $\Sym(n)$ for the alternating and symmetric groups of degree $n$.

	\vspace{0.5cm}
	\noindent\textbf{Acknowledgements}:
	The authors thank Rob Wilson for useful discussions.
	The third author thanks Yazd University for hospitality during a visit in January 2019 when this project began, and the second author thanks the Research Deputy of Yazd University for some financial support.
	The third and fourth authors thank the Isaac Newton Institute for Mathematical Sciences, Cambridge, for support and hospitality during the programme \emph{Groups, representations and applications: new perspectives}  (supported by EPSRC grant number EP/R014604/1), when significant work on this paper was undertaken. In particular the third author acknowledges a Kirk Distinguished Fellowship during this program.  
	The fourth author would also like to express his thanks to EPSRC for their support via the grant EP/T017619/1.

	\section{Applying Wielandt's Dissection Theorem}\label{sec:WDT}
	
	As we mentioned in Section~\ref{intro}, Wielandt's Dissection Theorem~\cite[Dissection Theorem 6.5]{Wielandt} is a fundamental tool for studying $2$-closure. For convenience we state it here. Note that,
	for $G\leq\Sym(\Om)$ and $\delta\in\Om$, we denote by $G_\delta$  the stabiliser of $\delta$ in $G$, and by $\delta^G$ the $G$-orbit in $\Om$ containing $\delta$. Also if $\Ga\subseteq\Om$ and $\Ga$ is $G$-invariant, then $G^\Ga$ denotes the permutation group on $\Ga$ induced by $G$.
	
	\begin{theorem}\label{thm:W2} {\rm [Wielandt's Dissection Theorem]}\quad
		Let $G\leq\Sym(\Omega)$, and suppose that $\Om$ is the disjoint union $\Omega=\Gamma\cup\Delta$ such that both $\Gamma$ and $\Delta$ are $G$-invariant.
		Then the following are equivalent.
		\begin{enumerate}[\upshape(a)]
			\item $G^\Gamma\times G^\Delta\leq G^{\Otwo}$;
			\item $G=G_\gamma G_\delta$ for all $\gamma\in\Gamma, \delta\in\Delta$;
			\item $G_\delta$ is transitive on $\gamma^G$ for all $\gamma\in\Gamma, \delta\in\Delta$.
		\end{enumerate}
	\end{theorem}
	
	In this section we discuss some ways in which this result can be applied. To begin, we state the following helpful application given in \cite{AAT}. Here, for a subgroup $H$ of a group $G$, we denote by
	$\Core_G(H)$ the \emph{core} of $H$ in $G$, that is, $\Core_G(H)=\cap_{g\in G} H^g$, the largest normal subgroup of $G$ contained in $H$. By a \emph{nontrivial factorisation} of $G$ we mean a factorisation $G=HK$ with both $H$ and $K$ proper subgroups of $G$.
	
	\begin{lemma}\label{lem:AAT31} {\rm \cite[Lemma 3.1]{AAT}} \quad
		Let $G$ be a finite totally $2$-closed group such that $G$ admits a nontrivial factorisation $G=HK$
		with
		$$
		\Core_G(H)\cap \Core_G(K)=1.
		$$
		Then $G = \Core_G(H)\times \Core_G(K)$ and both $\Core_G(H)$ and $\Core_G(K)$ are totally $2$-closed.
		In particular, if $H \cap K = 1$, then $G = H \times K$ and both $H$ and $K$ are totally $2$-closed.
	\end{lemma}
	
	Lemma~\ref{lem:AAT31} has the following important consequence for simple groups.
	
	\begin{corollary}\label{lem:W2}
		Let $T$  be a finite nonabelian simple group. If $T$ is totally $2$-closed, then $T$ admits no nontrivial factorisation.
	\end{corollary}
	
	\begin{proof}
		Suppose that $T$  is totally $2$-closed. If $T=HK$ is a nontrivial factorisation, then the cores $\Core_T(H)=\Core_T(K)=1$ since $T$ is simple, and this contradicts Lemma~\ref{lem:AAT31}. Hence $T$ has no nontrivial factorisations.
	\end{proof}

	The following result from Wielandt~\cite{Wielandt} is also useful when making a detailed
	analysis of $2$-closures.
	
	\begin{theorem}\label{thm:W}\cite[Theorem 5.6]{Wielandt}\quad Let $G\leq \Sym(\Omega)$ and $x\in\Sym(\Omega)$. Then $x\in G^\Otwo$ if and only if, for all $\alpha, \beta\in\Omega$, there exists $g\in G$ such that $\alpha^x = \alpha^g$ and $\beta^x=\beta^g$.
	\end{theorem}

	\subsection{Reduction to direct products of nonabelian finite simple groups}
	
	In this section, we reduce our proof of Theorem \ref{thm:classification} to particular direct products of finite simple groups. Before we state our main result, we remind the reader of some standard terminology: a nontrivial group $X$ is said to be a \emph{section} of a group $G$ if there are subgroups $L \unlhd H \leq G$ such that $H/L\cong X$. Our reduction theorem can now be stated as follows.
	\begin{theorem}\label{thm:issimple}
		Let $G$ be a non-trivial finite totally $2$-closed group such that $F(G)=1$. Then $G=T_1\times\dots\times T_r$ for some $r\geq1$, where for each $i\leq r$,  $T_i$ is a finite nonabelian simple group  which is totally $2$-closed, and $T_i$ is not a section of $\prod_{j\ne i} T_j$.
	\end{theorem}
	\begin{proof}
		%
		Since $G$ is finite, nontrivial, and $F(G)=1$, a minimal normal subgroup $M$ of $G$ is a direct product of nonabelian finite simple groups. If $M=G$ then, as $M$ is minimal normal, the group $G$ would itself be a nonabelian simple group which is totally $2$-closed, and the conclusion holds with $r=1$. Thus we may assume that $M\ne G$, and $M=S_1\times\dots\times S_s$ with each $S_i$ a finite nonabelian simple group.
		We note that the minimal normal subgroups of $M$ are precisely the subgroups
		\[
		\hS_i=\{(x_1,\dots,x_s)\mid x_j=1\ \mbox{for $j\ne i$, and } x_i\in S_i\}\cong S_i,
		\]
		for $i=1,\dots,s$ (see, for example, \cite[Theorem4.16(iv)]{PS}).
		By the `Odd Order Theorem', since each of the $\hS_i$ is a finite nonabelian simple group, each $\hS_i$ has a nontrivial Sylow $2$-subgroup, $P_i$ say, and hence $P=P_1\times\dots\times P_s$ is a (nontrivial)  Sylow $2$-subgroup of $M$. By the Frattini argument (see, for example, \cite[5.2.14]{R}), $G=N_G(P)M$. Now $N_M(P)= N_{\hS_1}(P_1)\times\dots\times N_{\hS_s}(P_s)$, and $N_M(P)$ projects to a proper subgroup $N_{\hS_i}(P_i)$ of $\hS_i$, for each $i$. In particular $M\not\leq N_G(P)$. Hence the core $X=\Core_G(N_G(P))$ does not contain $M$, and so $X\cap M=1$ since $M$ is a minimal normal subgroup of $G$. Also since $M$ is normal in $G$, $\Core_G(M) = \cap_{g\in G} M^g = M$. Then, since $G$ is totally $2$-closed, it follows from Lemma~\ref{lem:AAT31} that $G=X\times M$ and both $X$ and $M$ are totally $2$-closed. Since $M\ne G$ it follows that $X\ne 1$ and since $M$ is a minimal normal subgroup of $G$ we conclude that $M=\hS_1$ is a finite nonabelian simple group.
		
		Thus each minimal normal subgroup of $G$ is a totally $2$-closed, finite nonabelian simple group and is a direct factor of $G$.
		Applying this argument recursively we conclude that $G=T_1\times T_2\times\dots\times T_r$ for some $r\geq2$, where each $T_i$ is a totally $2$-closed, finite nonabelian simple group. For $1\leq j\leq r$ let $\pi_j$ denote the projection  $\pi_j:G\to T_j$ given by $(t_1,\dots, t_r)\to t_j$, and let $\hT_j$ denote the minimal normal subgroup of $G$ defined similarly to the $\hS_i$ above.

		Let $i,j$ be distinct integers between $1$ and $r$. We  show next that $T_i$ is not a section of $T_j$. Suppose that this is false, so there exist subgroups $L\unlhd H\leq T_j$ such that $H/L\cong T_i$, or equivalently, there exists an epimorphism $\varphi:H\to T_i$ with kernel $L$.
		Without loss of generality assume that $i=1$ and $j=2$, and let $J=\{ (\varphi(h),h,1,\dots,1) \mid h\in H\}$ and $K=1\times T_2\times\dots\times T_r$. Clearly both $J$ and $K$ are proper subgroups of $G$. Moreover,  since $\varphi$ is an epimorphism it follows that $G=JK$ is a nontrivial factorisation, and since  $K$ is normal in $G$ we have $\Core_G(K)=K$. We claim that $\Core_G(J) = 1$. Now $\Core_G(J)\leq J < \hT_1\times\hT_2\times 1\dots\times 1$, and if $\Core_G(J)\ne 1$ then $\Core_G(J)$ would contain at least one of the minimal normal subgroups of $G$, namely one of $\hT_1,\dots,\hT_r$, so $\Core_G(J)$ would contain $\hT_1$ or $\hT_2$.
		Let $x=(\varphi(h),h,1,\dots,1)\in\Core_G(J)$. If $x_1=\varphi(h)\ne 1$ then $h\ne 1$, so $\Core_G(J)\cap\hT_1=1$. Also if $x\in\hT_2$, then $x_1=\varphi(h)=1$ so $h\in \Ker(\varphi)=L$.  Since $L$ is a proper subgroup of $T_2$ it follows that $\Core_G(J)$ does not contain $\hT_2$ either. Hence $\Core_G(J)=1$ as claimed.
		It now follows from Lemma~\ref{lem:AAT31} that $G=1\times K$, which is a contradiction. Thus none of the $T_i$ is a section of any of the other $T_j$ (and in particular the $T_i$ are pairwise non-isomorphic).
		
		Finally, suppose that, for some $i\leq r$, $T_i$ is a section of $K:=\prod_{j\ne i}T_j$. Then there are subgroups $L\unlhd H\leq K$ such that $H/L\cong T_i$. Let $H, L$ be  subgroups of $K$ with this property  such that $|H|$ is as small as possible. Identify $K$ with the normal subgroup $\prod_{j\ne i}\hT_j$ of $G$. Let $1\leq j\leq r$ with $j\ne i$, and consider  $H_j :=H\cap (\prod_{k\ne j}\hT_k)$, the kernel of the restriction  $\pi_j|_{H}$  to $H$ of the projection
		$\pi_j$. We will identify $\pi_j(H)$ with the quotient $H/H_j$ so that, for each subgroup $J\leq H$, the image $\pi_j(J)$ is identified with $JH_j/H_j$. In particular $\pi_j(L)= LH_j/H_j$. Then
		\[
		Q:=\pi_j(H)/\pi_j(L)=(H/H_j)/(LH_j/H_j)\cong H/(LH_j)\cong (H/L)/(LH_j/L),
		\]
		a quotient of the simple group $H/L\cong T_i$. Thus $Q\cong T_i$ or $1$. Now $\pi_j(L)\unlhd \pi_j(H)\leq T_j$, and by the previous paragraph $T_i$ is not a section of $T_j$, and hence the quotient $\pi_j(H)/\pi_j(L)\not\cong T_i$.
		Thus $Q$ must be trivial, and hence $H=LH_j$. This, however, implies that $T_i\cong H/L=(LH_j)/L\cong
		H_j/(L\cap H_j)$, and so, by the minimality of $|H|$, we conclude that $H=H_j$, that is to say, $\pi_j(H)=1$.  Since this holds for each $j\ne i$, and since $H\leq \prod_{j\ne i}\hT_j=K$, it follows that $H=1$, which is a contradiction. Thus $T_i$ is not a section of $K$, and the proof is complete.
	\end{proof}
	
	\begin{remark}\label{rem:section}{\rm
	We make some comments about the condition in Theorem~\ref{thm:issimple} that each $T_i$ is not isomorphic to a section of the product $\prod_{j\ne i}T_j$. Since the $T_i$ are simple groups, this condition is equivalent to requiring that, for all distinct $i, j$, the group $T_i$ is not isomorphic to a section of $T_j$.  Deciding precisely when this condition holds was not easy, but in the case where both $T_i$ and $T_j$ are sporadic simple groups (which is the case with the examples in Theorem~\ref{thm:classification}), an almost complete answer is given in \cite[p.\,238]{Atlas}, which provides a beautiful diagram showing exactly which sporadic simple groups are involved in this way in other sporadic simple groups. The only unresolved case in this diagram is whether or not the small Janko group $J_1$ occurs as a section of the Monster $\mathbb{M}$, but it is now known that $J_1$ is not a section of the Monster, see \cite[Section 5.8]{W} for a discussion. 
	Considering this diagram (and the situation for $J_1, \mathbb{M}$) we see that, for $T_i, T_j$ any distinct pair of sporadic groups occurring in the statement of Theorem~\ref{thm:classification}, if $T_i$ is a section of $T_j$ then $T_i=Th$ and 
	$T_j=\mathbb{M}$ (which explains the need for two parts in Theorem~\ref{thm:classification}(3)).
}	\end{remark}

	\subsection{Totally $2$-closed direct products of simple groups}\label{sub:directprods}
	
	By Theorem~\ref{thm:issimple}, we need to analyse direct products of totally $2$-closed simple groups. With this in mind we introduce the following hypothesis and notation.

	\begin{hypothesis}\label{not1}
		{\rm
			Let $G=T_1\times T_2\times\dots\times T_r$ for some $r\geq1$, where each $T_i$ is a  finite nonabelian simple group, and suppose that, for  $1\leq i\leq r$, $T_i$ is not a section of $\prod_{j\ne i} T_j$.  For $1\leq i\leq r$, let $\pi_i:G\to T_i$ denote the projection map $\pi_i:(t_1,\dots, t_r)\to t_i$, and define
			\[
			\hT_i=\{g\in G\mid \pi_j(g)=1\ \mbox{for all $j\ne i$}\} \cong T_i.
			\]
			As mentioned earlier, (see, for example, \cite[Theorem 4.16(iv)]{PS}),  $\{\hT_1, \dots, \hT_r\}$ is the complete set of minimal normal subgroups of $G$.
		}
	\end{hypothesis}
	
	This notation helps avoid possible confusion by distinguishing between the abstract group $T_i$ and the simple normal subgroup $\hT_i$ of $G$. Note that Hypothesis~$\ref{not1}$ holds, by Theorem~\ref{thm:issimple}, if $G$ is totally $2$-closed with $F(G)=1$.  We extend the argument in the last paragraph of the proof of Theorem~\ref{thm:issimple} to obtain the following characterisation of certain subgroups of $G$.
	
	\begin{lemma}\label{lem:dp1}
		Let $G, r$ be as in Hypothesis~$\ref{not1}$, and let $i\leq r$. If $H\leq G$ has $T_i$ as a section, then $H$ contains $\hT_i$.
	\end{lemma}
	
	\begin{proof}
		Suppose that $H\leq G$ and $H$ has a section $T_i$, that is to say, there exist $L\unlhd K\leq H$ such that $K/L\cong T_i$. We need to show that $\hT_i\leq H$. Without loss of generality we may assume that no proper subgroup of $K$ has $T_i$ as a section.
		Let $1\leq j\leq r$ with $j\ne i$, and let $J=K\cap (\prod_{k\ne j}\hT_k)$, the kernel of the restriction $\pi_j|_{K}$ of $\pi_j$ to $K$. Then $\pi_j(L)\unlhd \pi_j(K)\leq T_j$ and $\pi_j(K)/\pi_j(L)\cong(K/J)/(LJ/J)\cong K/(LJ)\cong (K/L)/(LJ/L)$, a quotient of the simple group $K/L\cong T_i$. By Hypothesis~$\ref{not1}$, $T_i$ is not a section of $\prod_{k\ne i}T_k$, and in particular,
		$T_i$ is not a section of $T_j$. Hence $\pi_j(K)/\pi_j(L)\not\cong T_i$. This implies that
		$\pi_j(K)=\pi_j(L)$, or equivalently, $K=LJ$. Thus $T_i\cong K/L=(LJ)/L\cong
		J/(L\cap J)$, and since $T_i$ is not a section of any proper subgroup of $K$,
		we conclude that $K=J$, that is, $\pi_j(K)=1$. Since this holds for each $j\ne i$, it follows that $K\leq \hT_i$, and since $T_i$ is a section of $K$ this means that $K=\hT_i$. Hence $\hT_i\leq H$.
	\end{proof}

	Lemma~\ref{lem:dp1} has an important consequence for factorisations of groups satisfying Hypothesis~\ref{not1}. The proof uses the notion of a strip. For $G=T_1\times T_2\times\dots\times T_r$ as in Hypothesis~\ref{not1}, a subgroup $D$ is called a \emph{strip} if $D\ne 1$, and for each $i\leq r$, either $\pi_i(D)=1$ or $D\cap\Ker(\pi_i)=1$. Note that  $D\cap\Ker(\pi_i)=1$ if and only if the restriction $\pi_i|_D$ of $\pi_i$ to $D$ is an isomorphism from $D$ to $\pi_i(D)$. The \emph{support} of a strip is the set $\{i\mid \pi_i(D)\ne 1\}$, and two strips are said to be \emph{disjoint} if their supports are disjoint. A strip $D$ is \emph{full} if $\pi_i(D)=T_i$ for all $i$ in the support of $D$.
	A subgroup $H$ of a group $G$ with a given direct decomposition $G=T_1\times T_2\times\dots\times T_r$ is called a \emph{subdirect subgroup} if $\pi_i(H)=T_i$ for each $i=1,\dots,r$.

	\begin{lemma}\label{lem:dp2}
		Let $G$ be as in Hypothesis~$\ref{not1}$, and such that each $T_i$ admits no nontrivial factorisation. If $G=HK$ is a nontrivial factorisation, then $\Core_G(H)\ne 1$ and $\Core_G(K)\ne 1$.
	\end{lemma}
	
	\begin{proof}
		Suppose that $G=HK$ with proper subgroups $H, K$, and suppose that one of $H, K$ is core-free, say $\Core_G(H)=1$. By the assumption on the $T_i$, we must have $r\geq2$. Let $i\leq r$.
		If $\pi_i(H)=T_i$, then  $T_i$ is a quotient of $H$, and hence by Lemma~\ref{lem:dp1}, $H$ contains $\hT_i$, contradicting the fact that $\Core_G(H)=1$. Thus $\pi_i(H)< T_i$. Now the factorisation $G=HK$ implies that $T_i=\pi_i(H)\pi_i(K)$, and since $\pi_i(H)<T_i$, and by assumption the group $T_i$ has no nontrivial factorisation, it follows that $\pi_i(K)=T_i$. Thus $K$ is a subdirect subgroup of $G=T_1\times T_2\times\dots\times T_r$. By `Scott's Lemma' \cite[Theorem 4.16(iii)]{PS}, $K$ is a direct product of pairwise disjoint full strips of $G$. Since $K$ is a proper subgroup of $G$, at least one of these strips, say $D$, has support of size greater than $1$. Let $i,j$ be distinct elements of the support of $D$. By the definition of a strip, $D\cap\Ker(\pi_i)=1$ and  $D\cap\Ker(\pi_j)=1$, and as we noted above, this implies that $D\cong \pi_i(D)=T_i$ and $D\cong \pi_j(D)=T_j$. This implies that $T_i$ is a section of $T_j$, and hence also of $\prod_{k\ne i}T_k$, contradicting Hypothesis~\ref{not1}.
	\end{proof}

	The next proposition assists us to decide which groups satisfying Hypothesis~$\ref{not1}$ are in fact totally $2$-closed by reducing consideration to transitive (not necessarily faithful) permutation representations.

	\begin{proposition}\label{PreTransRed}
		Let $G$ be as in Hypothesis~$\ref{not1}$. Then $G$ is totally $2$-closed if and only if, the following two conditions hold.
		\begin{enumerate}[\upshape(a)]
			\item For $i=1,\dots,r$, the group $T_i$ admits no nontrivial factorisation; and
			
			\item for each permutation representation $\varphi:G\to\Sym(\Om)$ with $|\Om|>1$ and $\varphi(G)$ transitive,  we have $\varphi(G)^\Otwo=\varphi(G)$.
			
		\end{enumerate}
	\end{proposition}
	
	\begin{proof}
		Suppose first that $G$ is totally $2$-closed. Then condition (a) holds  by Theorem~\ref{thm:issimple}.
		Let $\varphi:G\to\Sym(\Om)$ with $|\Om|>1$ and $\varphi(G)$ transitive. By \cite[Theorem 4.16(iv)]{PS}, there is a proper subset $I\subseteq\{1,\dots,r\}$ such that $\Ker(\varphi)=\prod_{i\in I}\hT_i$, and so $\varphi(G)\cong G_0$ where $G_0:= \prod_{j\in J}\hT_j$ with $J=\{1,\dots,r\}\setminus I$. Hence $G=\Ker(\varphi)\times G_0$ and by Lemma~\ref{lem:AAT31}, $G_0$ is totally $2$-closed, whence $\varphi(G)^\Otwo=\varphi(G)$ holds. Thus condition (b) holds also.
		
		Conversely suppose that  conditions (a) and (b) hold for $G$. In particular each nontrivial transitive permutation representation of $G$ is $2$-closed. In order to prove that $G$ is totally $2$-closed we must therefore show that each faithful intransitive permutation representation of $G$ is also $2$-closed. Suppose that this is not true. Then we may suppose that $G\leq \Sym(\Om)$ with orbits $\Om_1,\dots, \Om_m$ and $m\geq2$, such that $X:=G^\Otwo$ contains $G$ as a proper subgroup. Suppose moreover that $m$ is minimal, over all faithful non-$2$-closed intransitive representations of all groups $G$ satisfying Hypothesis~\ref{not1} and satisfying conditions (a) and (b).
		
		For each $i\leq m$, let $G_i=G^{\Om_i}$, the transitive permutation group on $\Om_i$ induced by $G$.
		It follows from the minimality of $m$ that each of the $|\Om_i|\geq2$ and hence each of the $G_i$ is a nontrivial transitive permutation representation of $G$. Thus by condition (b), each $G_i$ is $2$-closed, that is $G_i^{(2),\Om_i}=G_i$.
		Let $\Delta=\cup_{i=2}^m\Om_i$, and let $Y=G^\Delta$, the permutation group on $\Delta$ induced by $G$. Then, arguing as in the first paragraph, $G_1\cong \prod_{i\in I}T_i$ and $Y\cong \prod_{j\in J}T_j$ for some non-empty subsets $I, J\subseteq\{1,\dots,r\}$, and since $G$ is faithful on $\Om$, $I\cup J=\{1,\dots,r\}$. Further, Hypothesis~\ref{not1} holds for $Y$, as does condition (a) since it holds for $G$. Also,  condition (b) on nontrivial transitive permutation representations for $Y$ holds since each such representation for $Y$ is also a nontrivial transitive permutation representation of $G$. Since $Y$ has only $m-1$ orbits in $\Delta$, it follows that
		$Y^{(2),\Delta}=Y$ (either from the minimality of $m$ if $m\geq3$, or from the condition on nontrivial transitive permutation representations of $G$ if $m=2$).  Now $\Om=\Om_1\cup\Delta$, and by Theorem~\ref{thm:W},  $X=G^\Otwo\leq G_1^{(2),\Om_1}\times Y^{(2),\Delta} = G_1\times Y\cong (\prod_{i\in I}T_i)\times (\prod_{j\in J}T_j)$. If $I\cap J=\emptyset$, then since $I\cup J=\{1,\dots,r\}$ it would follow that $X\cong G$, which is a contradiction. Hence $I\cap J\ne\emptyset$, $X$ is a subdirect subgroup of $G_1\times Y$, and
		\[
		G_1\times Y\cong \prod_{i=1}^r T_i^{\delta_i},\ \mbox{where $\delta_i\in\{1,2\}$ and
			$\delta_i=2$ if and only if $i\in I\cap J$}.
		\]
		Hence $X\cong \prod_{k=1}^rT_k^{\alpha_k}$, where for each $k$, $1\leq \alpha_k\leq \delta_k$, and since $G<X$ there exists $i$ such that $\alpha_i=\delta_i=2$. Such an $i$ must lie in $I\cap J$, and without loss of generality, we may assume that $i=1$ so $1\in I\cap J$. Hence $\hT_1$ acts nontrivially on both $\Om_1$ and $\Delta$, and $X$ contains $\hT_1^{\Om_1}\times \hT_1^\Delta\cong T_1^2$. Again, without loss of generality, $\hT_1$ acts nontrivially on $\Om_2$, and $\hT_1^{\Om_2}\cong \hT_1^\Delta$.
		
		Let $K=\prod_{j\ne 1}\hT_j$, so that $G=\hT_1\times K$. If $K$ is transitive on $\Om_1$ then $G=KG_\omega$, where $\omega\in\Om_1$, and $G_\omega/(K\cap G_\omega)\cong G_\omega K/K=G/K\cong T_1$. Then, by Lemma~\ref{lem:dp1},  $\hT_1\leq G_\omega$, which implies that $\hT_1$ fixes $\Om_1$ pointwise, a contradiction.
		Hence $K$ is intransitive on $\Om_1$, and the same argument shows that $K$ is also intransitive on $\Om_2$. Let $\Sigma_i$ denote the set of $K$-orbits in $\Om_i$, for $i=1, 2$. Since $G=\hT_1\times K$,
		the group $\hT_1$ is transitive on both $\Si_1$ and $\Si_2$, and since $X$ contains  $\hT_1^{\Om_1}\times \hT_1^\Delta\cong T_1^2$, it follows that $X$ is transitive on $\Si_1\times\Si_2$ in product action, that is, $x:(\si_1,\si_2)\to (\si_1^x,\si_2^x)$ for $\si_1\in\Si_1, \si_2\in\Si_2, x\in X$.
		
		We claim that the subgroup $\hT_1$ is also transitive on $\Si_1\times\Si_2$. Choose $(\si_1,\si_2)\in\Si_1\times\Si_2$, and let $\beta_1\in\si_1, \beta_2\in\si_2$. Since $X=G^{\Otwo}$, the $G$-orbit $(\beta_1,\beta_2)^G$ is equal to the $X$-orbit $(\beta_1,\beta_2)^X$. First we note that, since   $X$ is transitive on $\Si_1\times\Si_2$, the stabiliser $X_{\si_1}$ is transitive on $\Si_2$. Now $\si_1$ is a $K$-orbit in $\Om_1$ and so $X_{\si_1}=KX_{\beta_1}$. It follows that the  $X_{\beta_1}$-orbit containing $\beta_2$ must contain points from each element of $\Si_2$. On the other hand, $(\beta_1,\beta_2)^G=(\beta_1,\beta_2)^X$ implies that the $X_{\beta_1}$-orbit containing $\beta_2$ is the same as the $G_{\beta_1}$-orbit containing $\beta_2$, and therefore this $G_{\beta_1}$-orbit contains points from each element of $\Si_2$, and hence $G_{\beta_1}$ is transitive on $\Sigma_2$. Since $G_{\beta_1}<G_{\si_1}$, also  $G_{\si_1}$ is transitive on $\Si_2$ and hence $G$ is transitive on $\Si_1\times\Si_2$. Since
		$G=\hT_1\times K$ and $K$ acts trivially on  $\Si_1\times\Si_2$, we conclude that $\hT_1$ is transitive on  $\Si_1\times\Si_2$, proving the claim.
		
		Finally, for $i=1, 2$, let $S_i=(\hT_1)_{\si_i}$. Then since  $\hT_1$ is transitive on  $\Si_1\times\Si_2$, $S_1$ is transitive on $\Si_2$ and hence $\hT_1=S_1S_2$, a nontrivial factorisation, which is a  contradiction to condition (a).
	\end{proof}
	
	For proving that a simple group is totally $2$-closed, the following immediate corollary of Proposition~\ref{PreTransRed} gives a reduction to consideration of transitive actions.
	
	\begin{corollary}\label{TransRed}
		Let $T$ be a finite nonabelian simple group.  Then $T$ is totally $2$-closed if and only if $T$ admits no nontrivial factorisations and, for each permutation representation $\varphi:G\to\Sym(\Om)$ with $|\Om|>1$ and $\varphi(G)$ transitive,  we have $\varphi(G)^\Otwo=\varphi(G)$.
	\end{corollary}
	
	\subsection{Simple groups, factorisations, and rank}\label{sub:factns}
	
	In Proposition~\ref{lem:simple1}, we consider the totally $2$-closed simple groups  $G$ in the light of Corollary~\ref{TransRed}, applying the classification of their maximal factorisations \cite{LPS2}, and considering their low rank permutation representations. For example, if a simple group $G$ is $2$-transitive on $\Om$ then $G^\Otwo = \Sym(\Om)$, by \cite[Theorem 5.11]{Wielandt}, and hence $G$ is not totally $2$-closed.

	We denote the \emph{socle} of $G$ by $\soc(G)$, that is, the product of the minimal normal
	subgroups of $G$.  We denote the set of right cosets of a subgroup $H$ in $G$ by $[G:H]=\{ Hg\mid g\in G\}$, and note that $G$ induces by right multiplication a transitive
	action on $[G:H]$ with kernel $\Core_G(K)$, and such that $H$ is the stabiliser of the `point' $H$.
	
	\begin{proposition}\label{lem:simple1}
		Let $G$ be a  finite non\-abelian simple group that is totally $2$-closed. Then one of the following is true.
		
		\begin{enumerate}[\upshape(a)]
			\item $G$ lies in one of the following infinite families: ${}^3D_4(q)$, $E_6(q), {}^2E_6(q)$,
			$E_7(q), E_8(q)$, $F_4(q)$ (with $q$ odd), ${}^2F_4(q)'$ (with $q=2^{2c+1}$); or
			\item $G$ is one of $\mathrm{J}_1, \mathrm{J}_3, \mathrm{J}_4,\mathrm{Ly},\mathrm{ON}, \mathrm{Co}_2, \mathrm{Fi}_{23}, \mathrm{HN}, \mathrm{Th}, \mathbb{B}, \mathbb{M}$.
		\end{enumerate}
	\end{proposition}
	
	\begin{proof}
		By Corollary~\ref{TransRed}, $G$ admits no nontrivial factorisation.
		We treat the simple groups by families. First, if $G=A_n$, then $G$ is $2$-transitive of 
		degree $n$ and so, as we observed above, $G$ is not totally $2$-closed, a contradiction. Thus $G\not\cong A_n$.
		
		If $G$ is a totally $2$-closed sporadic simple group, then this implies, 
		by \cite[Theorem C]{LPS2}, that either $G$ is as in part (b), or $G\in\{ \mathrm{Co}_3, \mathrm{McL}, \mathrm{Fi}_{24}'\}$. 
		Now $G\ne \mathrm{Co}_3$ since $\mathrm{Co}_3$ has a $2$-transitive representation of degree $276$ (see, for example \cite[Section 7.7]{DM}). If  $G$ is $\mathrm{McL}$ or $\mathrm{Fi}_{24}'$, then
		$A=\Aut(G)= G.2$, and the coset action  of $A$ on the set $\Omega=[A:H]$ of right cosets of
		$H=PSU(4,3):2_3$ or  $\mathrm{Fi}_{23}\times 2$, respectively, is such that both $A$ and $G$ have rank 3 (see
		\cite[Pages 100, 207]{Atlas}).  It follows that
		$A \leq G^\Otwo$ by Theorem~\ref{thm:W}, so that $G\ne G^\Otwo$, which is a contradiction.
		
		Thus we may assume that $G$ is a totally $2$-closed simple group of Lie type. Suppose first that $G$ is a classical simple group. Then $G\ne {\rm PSL}_n(q)$ or ${\rm PSU}_3(q)$ since these groups have a $2$-transitive representation, (see, for example \cite[Section 7.7]{DM}). Among the other simple classical groups, the only ones which have no nontrivial factorisations are, by \cite[Theorem A]{LPS2},
		\begin{itemize}
			\item $G={\rm PSU}_n(q)$ with $n$ odd, $n\geq 5$, and $(n,q)\ne (9,2)$; and
			\item $G={\rm P}\Omega^-_{2m}(q)$ with $m$ even, $m\geq4$, $q\ne 2, 4$.
		\end{itemize}
		For $G$ one of these groups, it follows from \cite[Theorem 1.1 and Table IV on page 45]{KanLieb}, that $G$ has a rank $3$ action on the set $\Omega$ of totally singular (or isotropic) $1$-spaces of the natural module, of degree
		\[
		\mbox{$\frac{(q^n+1)(q^{n-1}-1)}{q^2-1}$\quad or\quad $\frac{(q^m+1)(q^{m-1}-1)}{q-1}$,\quad respectively}.
		\]
		However the group $A= {\rm P\Gamma U}_n(q)$ or ${\rm P}\Gamma{\rm O}^-_{2m}(q)$, respectively, contains $G$ as a proper subgroup  (see \cite[Table 2.1]{LPS2}), and also acts on $\Omega$ with rank $3$. Hence
		$A \leq G^\Otwo$ by Theorem~\ref{thm:W}, and so $G\ne G^{(2),\Omega}$, which is a contradiction.
		
		Finally suppose that $G$ is an exceptional simple group of Lie type. Then $G\ne {\rm Sz}(q)$ or ${\rm Ree}(q)$ since each of these groups has a $2$-transitive representation, (see, for example \cite[Section 7.7]{DM}). Among the other exceptional simple groups of Lie type, the only ones which do not appear in part (a), and which have no nontrivial factorisations are, by \cite[Theorem B]{LPS2},  the groups  $G=G_2(q)$
		(with $q>2, q\ne 4, 3^c$). However it follows from \cite[Theorem 1]{LPS}, that such a group $G$
		has a primitive permutation representation on a set $\Omega$ of size $q^3(q^3-1)/2$ such that $G^\Otwo$ has socle $\Omega_7(q)$, so $G\ne G^\Otwo$, which is a contradiction.
	\end{proof}
	
	To summarise this section, we have proved that to determine the finite totally $2$-closed groups $G$ with $F(G)=1$, we need only consider the groups $G$ which are a direct product of a subset of the nonabelian finite simple groups listed in Proposition \ref{lem:simple1}. Furthermore, the direct factors must satisfy the conditions listed in Theorem \ref{thm:issimple}. Finally, by Corollary \ref{TransRed}, such a group is totally $2$-closed if and only if it is $2$-closed in each of its (not necessarily faithful) transitive permutation representations. We investigate the transitive representations of these, and other more general, groups in Section \ref{ImprimitiveSection}.

	\section{Tools for analysing imprimitive actions}\label{ImprimitiveSection}
	
	In this section we introduce some group theoretic and combinatorial tools for studying $2$-closures of transitive group actions.
	
	Let $\Omega$ be a finite set, and let $G$ be a subgroup of $\Sym(\Omega)$. A partition $\Sigma$ of $\Omega$ is \emph{$G$-invariant} if each element of $G$ permutes the parts of $\Sigma$ among themselves. A partition is said to be \emph{trivial} either if it consists of a single part, namely $\Omega$, or if each of its parts consists of a single point. All other partitions are nontrivial. Each transitive permutation group $G$ leaves invariant these two trivial partitions, and $G$ is called \emph{primitive} if the only $G$-invariant partitions are the two trivial partitions; otherwise $G$ is said to be \emph{imprimitive}.
	
	Higman's landmark result  \cite[(1.12)]{Higman} in this area (which we will use) is that a finite transitive permutation group $G$ is primitive if and only if all the nondiagonal $G$-orbital digraphs are connected.

	\begin{lemma}\label{lem:imprim}
		Let $G\leq \Sym(\Omega)$ be transitive on $\Omega$,  and let $X=G^{(2),\Omega}$. Suppose that $\Sigma$ is a non-trivial $G$-invariant partition of $\Omega$, and let $\overline{G}=G^\Sigma$ denote the subgroup of $\Sym(\Sigma)$ induced by $G$.

		\begin{enumerate}[\upshape(a)]
			\item Then $\Sigma$ is also $X$-invariant, and the induced group $\overline{X}=X^\Sigma$ satisfies
			$\overline{X}\leq  \overline{G}^{(2),\Sigma}$.
			\item  Also, for $B\in\Sigma$, $X_B^B\leq (G_B^B)^{(2),B}$.
		\end{enumerate}
	\end{lemma}
	
	\begin{proof}
		(a) For $\alpha\in\Omega$, let $B_\alpha$ denote the (unique) part of $\Sigma$ that
		contains $\alpha$, and let $\varphi:\Omega\times\Omega\to\Sigma\times\Sigma$ be the map
		$(\alpha,\beta)\to (B_\alpha, B_\beta)$.
		By \cite[Theorem 4.11]{Wielandt}, $G$ and $X$ have the same invariant partitions in $\Omega$, so $\Sigma$ is  $X$-invariant. This implies that $\varphi$ is $X$-invariant in the sense that, for all $\alpha,\beta\in\Omega$ and $x\in X$, $\varphi(\alpha^x,\beta^x)= (B_\alpha^x, B_\beta^x)$.  We show that $\overline{X}$ leaves invariant each $\overline{G}$-invariant subset of $\Sigma\times \Sigma$:
		let $\Delta\subseteq \Sigma\times \Sigma$ be $\overline{G}$-invariant, and let $\Delta'\subseteq \Omega\times\Omega$ be the set of all ordered pairs $(\alpha, \beta)$ such that $(B_\alpha, B_\beta)\in\Delta$.
		Then $\Delta'$ is the full pre-image of $\Delta$ under $\varphi$, that is, $\varphi(\Delta')=\Delta$, and
		it is straightforward to check that $\Delta'$ is a $G$-invariant subset of $\Omega\times\Omega$.
		Hence $\Delta'$ is also preserved by $X=G^{(2),\Omega}$, and it follows that $\Delta=\varphi(\Delta')$ is preserved by $\overline{X}$.  Since this holds for each $\overline{G}$-invariant $\Delta$, it follows that $\overline{X}\leq \overline{G}^{(2),\Sigma}$.


		(b) Let $\Delta$ be a subset of $B\times B$ preserved by $G_B^B$, and let $\Delta'$ be the set of all pairs $(\alpha,\beta)\in\Omega\times\Omega$ such that $(\alpha^g,\beta^g)\in\Delta$ for some $g\in G$. Then $\Delta'$ is a  $G$-invariant subset of $\Omega\times\Omega$. Moreover,  from the definition of $\Delta'$ the set $\Delta\subseteq \Delta'\cap(B\times B)$, and  to see  that equality holds let $(\alpha,\beta)\in\Delta'\cap(B\times B)$,
		so $\alpha, \beta\in B$ and $(\alpha^g,\beta^g)\in\Delta$ for some $g\in G$. Then, since  $\Delta\subseteq B\times B$, also $\alpha^g\in B$ so $\alpha^g\in B\cap B^g$ and hence $B^g=B$, since $\Sigma$ is a $G$-invariant partition. Thus $g\in G_B$, so also $g^{-1}\in G_B$. Since $\Delta$ is $G_B$-invariant it follows that $(\alpha,\beta)=(\alpha^g,\beta^g)^{g^{-1}}\in\Delta^{g^{-1}}=\Delta$.
		So we have proved that
		$\Delta'\cap(B\times B)=\Delta$.

		Since $X=G^{(2),\Omega}$, $\Delta'$ is also $X$-invariant, and hence $\Delta'\cap(B\times B)=\Delta$ is $X_B$-invariant. Thus $X_B^B\leq (G_B^B)^{(2),B}$, as required.
	\end{proof}
	
	\begin{remark}\label{•rem:imprim}
		{\rm
			We note that the inequalities in Lemma~\ref{lem:imprim} may be strict. For example, let $G=\Gamma{\rm L}_1(16)$ of degree $15$ on $\Omega = \mathbb{F}_{16}^\times$. By \cite[Corollary 4.1]{PX}, the $2$-closure $X=G^{(2),\Omega}$ is equal to $G$. Moreover, $G$ preserves two non-trivial partitions of $\Omega$, one of them $\Sigma$ with $5$ blocks of size $3$, and the other  $\Sigma'$ with $3$ blocks of size $5$. (The partitions $\Sigma, \Sigma'$ are the orbit sets of cyclic  normal subgroups of $G$ of orders $3, 5$, respectively.) Now $G^\Sigma\cong F_{20}$ is a sharply $2$-transitive Frobenius group of order $20$. Hence  $\overline{G}^{(2),\Sigma}=\Sym(5)$, while $\overline{X}=\overline{G}=F_{20}$. Also, for $B\in\Sigma'$, $G^B_B\cong F_{20}$ is sharply $2$-transitive, and so $(G_B^B)^{(2),B}=\Sym(5)$, while $X_B^B=G_B^B=F_{20}$.
			
			For a more general example of inequality in part (a), suppose that $G$ is a finite transitive permutation group on a set $\Sigma$ such that $G$ is not $2$-closed, that is, $G^{\Stwo}\ne G$. Let $B\in\Si$ and let $H=G_B$, so that we may identify $\Si$ with $[G:H]=\{Hg\mid g\in G\}$ and $G$ acts by right multiplication. Consider the regular action of $G$ on $\Om=G$ by right multiplication. Then $X=G^{\Otwo}=G$ by \cite[Theorem 5.12]{Wielandt}, and $\Si=[G:H]$ is a $G$-invariant partition of $\Om$, so $\overline{G}=G^\Si\cong G$.  Thus $\overline{X}=\overline{G} <\overline{G}^{\Stwo}$.
		}
	\end{remark}
	
	By Theorem~\ref{thm:issimple} and Proposition \ref{PreTransRed}, the classification of finite totally $2$-closed groups with trivial Fitting subgroup is reduced to the study of nontrivial transitive permutation representations of groups $G$ satisfying Hypothesis~\ref{not1} such that the simple direct factors $T_i$ of $G$ admit no nontrivial factorisations.
	Although it would be desirable to reduce further to primitive actions, this seems difficult (and perhaps not possible). Loosely speaking, our goal in this section is to prove that if such  groups $G$ are $2$-closed in all of their primitive permutation representations, then it is `likely' that they are totally $2$-closed.
	
	More precisely, we will explicitly construct the $2$-closure of a finite group $G$ acting imprimitively on a set $\Omega$, with the property that its induced action on a nontrivial $G$-invariant partition $\Sigma$ of $\Omega$ is faithful and $2$-closed.
	To state our results, we require some notation.

	\begin{notation}\label{not2}
		{\rm
			Assume that $G$ is a finite group, and that:
			\begin{enumerate}
				\item $G$ acts faithfully, transitively, and imprimitively on a finite set $\Omega$.
				\item $\Sigma=\{\Delta_1,\Delta_2,\hdots,\Delta_s\}$ is a nontrivial $G$-invariant partition of $\Omega$, and $L=G^\Sigma$, the  (transitive) subgroup of $\Sym(\Sigma)$ induced by $G$, where $1<s<|\Omega|$.
				\item $\Delta:=\Delta_1\in\Sigma$ and $\omega\in\Delta$.
				\item  $M=G_\Delta$ (setwise stabiliser, so that $|G:M|=s$), $H=G_\omega$ (so $H<M$ with $|M:H|=|\Delta|$),  and $R=M^\Delta$, the (transitive) subgroup of $\Sym(\Delta)$  induced by $M$ (so $R_\omega= H^\Delta$).
				\item By the Imprimitive Wreath Embedding Theorem~\cite[Theorem 5.5]{PS}, we may assume  that $G\leq W=R\wr L$ acting on $\Omega=\Delta\times\{ 1, 2,\dots, s\}$; here we identify $\Sigma$ with $\{1, 2,\dots,s\}$, and $\Delta_i$ with $\Delta\times\{ i\}$ for each $i\leq s$;
				
				\item $W=B_R\rtimes L$ is a semidirect product, where $B_R=\{(x_1,\hdots,x_s)\text{ : }x_i\in R\}\cong R^s$ is called \emph{the base group} of $W$, and the subgroup $L$ is called the \emph{top group} of $W$.  Typical elements $(x_1, x_2,\dots,x_s)\in B_R$ and $\sigma\in L$ act on a point $(\delta,i)\in\Omega$ by
				\begin{center}
					$(\delta,i)^{(x_1, x_2, \dots, x_s)}=
					(\delta^{x_i},i)$ and $(\delta,i)^\sigma=(\delta,i^\sigma)$.
				\end{center}
				
				\item  We write $B_R= R_1\times\dots\times R_s$, where
				$$
				R_i=\{(x_1,\hdots,x_s)\text{ : $x_j=1$ if $j\ne i$ } \} \cong R,
				$$
				so $R_i$ induces the group $R$ on $\Delta_i$ and fixes $\Delta_j$ pointwise for all $j\ne i$. We extend this notation so that, for any subgroup $A\leq \Sym(\Delta)$, we  write $B_A=A_1,\times\dots\times A_s \cong A^s$ for the naturally embedded subgroup of $\Sym(\Delta)\wr \Sym(s)$. We refer to such subgroups $B_A$ as `base groups'.
				
				\item For $A\le \Sym(\Delta)$, and for each $i\leq s$, the (setwise) stabiliser of $\Delta_i$ in
				$A\wr L = B_A\rtimes L$ is the subgroup $W_i:=B_A\rtimes L_i$ (where $L_i$ is the stabiliser of $i$  in $L$) and the map
				\begin{center}
					$\rho_i:W_i\to A$ given by $\rho_i: (x_1, x_2,\dots,x_s)\sigma\to x_i$
				\end{center}
				is a group epimorphism; similarly, for distinct $i, j$ the map
				\begin{center}
					$\rho_{i,j}:W_i\cap W_j\to A\times A$ given by $\rho_{i,j}: (x_1, x_2,\dots,x_s)\sigma\to
					(x_i, x_j)$
				\end{center}	
				is a group epimorphism, and its image $\rho_{i,j}(W_i\cap W_j)=\rho_{i,j}(A_i\times A_j)$ is the group induced by $W_i\cap W_j$ (and by $A_i\times A_j$) on $\Delta_i\cup\Delta_j$.
				
			\end{enumerate}
			
			The group $A_i\times A_j$ fixes $\Delta_\ell$ pointwise, for all $\ell\not\in\{i,j\}$, and so acts faithfully on $\Delta_i\cup\Delta_j$; it also has a faithful \emph{product action} on $\Delta_i\times\Delta_j$ given by $(\delta_i,\delta_j)^{(x_i,x_j)}=(\delta_i^{x_i}, \delta_j^{x_j})$. In our proofs we will meet the \emph{$1$-closure} of a subgroup $S\leq A_i\times A_j$ in this product action: the $1$-closure of $S$ is denoted $S^{(1), \Delta_i\times\Delta_j}$, and is the largest subgroup of $\Sym(\Delta_i\times\Delta_j)$ with the same orbits as $S$.  To work with the group $G$ we consider restrictions of the maps in Notation~\ref{not2}.8 to the corresponding subgroups of $G$.    We also impose an additional condition on $M$.
			
			\begin{itemize}
				
				\item[9.] For $i=1,\dots,s$, let $M_i=W_i\cap G$, the setwise stabiliser in $G$ of $\Delta_i$. Note that the subgroups $M_i, 1\leq i\leq s$ form the complete $G$-conjugacy class of $M=M_1$. It follows from parts 4 and 5 that $\rho_i(M_i)= M_i^{\Delta_i} =
				R_i\cong R$, for each $i$, while $\rho_{i,j}(M_i\cap M_j)\leq R_i\times R_j$ is the subgroup
				induced by $G$ on $\Delta_i\cup\Delta_j$, for distinct $i,j$.  We will meet the $1$-closure
				$(\rho_{i,j}(M_i\cap M_j))^{(1), \Delta_i\times\Delta_j}$ of the product action of this subgroup (as described immediately above).
				
				\item[10.] $M$ is assumed to be core-free in $G$, (equivalently $G\cong G^\Sigma = L$).
			\end{itemize}
		}
	\end{notation}
	
	Condition 10 is vacuously true if $G$ is simple. But it will also suffice (for us) to study $2$-closures of groups acting faithfully on a nontrivial invariant partition, as we will see in Section \ref{DirectProductsSection}.
	Note also that if $\Delta$ is chosen to be a maximal block of imprimitivity, then $M$ is a maximal subgroup of $G$ (we will often be interested in this case in our analysis below). Sometimes we will identify $\Omega$ with $[G:H]$; $\Sigma$ with $[G:M]$; and $\Delta=\Delta_1$ with $[M:H]$.
	
	
	
	In the following series of results, we will study the cases where $G^{\Sigma}$ is $2$-closed, but $G^{\Omega}$ is not $2$-closed. We will see that this seems like quite an unusual phenomenon.
	We begin with a straightforward lemma from representation theory.

	\begin{lemma}\label{InducedModules}
		Let $G, \Sigma, M, s$ be as in Notation~$\ref{not2}$, let $p$ be a prime, and let
		$$
		V=\{(y_1,\hdots,y_s)\text{ : }y_i\in\mathbb{F}_p\},
		$$
		be the permutation module for $G$ on $\Sigma$ over the field $\mathbb{F}_p$.
		Suppose that $M$ is not contained in a proper normal subgroup of $G$.
		Then the only $1$-dimensional $\mathbb{F}_p[G]$-submodule of $V$ is the trivial submodule $\{(y,\hdots,y)\text{ : }y\in\mathbb{F}_p\}$.
	\end{lemma}
	
	\begin{proof}
		Let $U$ be a $1$-dimensional $\mathbb{F}_p[G]$-submodule of $V$. Then there is an injective homomorphism from $U$ to $V$, that is, the $\mathbb{F}_p$-space $\Hom_G(U,V)$  of $\mathbb{F}_p[G]$-homomorphisms from $U$ to $V$ is non-zero. Since $V$ is the induced module $(1_M)\!\uparrow^G$, where $1_M$ denotes the trivial $\mathbb{F}_p[M]$-module, it follows from Frobenius reciprocity \cite[Proposition 3.3.1]{Benson} that
		$\Hom_M(U\!\downarrow_M,1_M)$ is also non-zero. Thus $U$ has a non-zero quotient $\mathbb{F}_p[M]$-module acted upon trivially by $M$, and since $\dim(U)=1$, it follows that $M$ acts trivially on $U$. Thus $M$ is contained in the kernel of the action of $G$ on $U$. Since $M$ is not contained in a proper normal subgroup of $G$, we deduce that $G$ acts trivially on $U$. It is then an elementary exercise to prove that $\{(y,\hdots,y)\text{ : }y\in\mathbb{F}_p\}$ is the only non-zero trivial $\mathbb{F}_p[G]$-submodule of $V$. This completes the proof.
	\end{proof}

	The next lemma is  a crucial tool in our treatment of imprimitive groups $G$ as in Notation~\ref{not2}. It makes use of Lemma~\ref{lem:imprim}. Recall in particular the definition of base groups in Notation~\ref{not2}.7.
	
	\begin{lemma}\label{Crucial} Let $G, \Omega, \Sigma, \Delta, L, R, M,   M_i, s, \rho_{i,j}$ be as in Notation~$\ref{not2}$, let $X:=G^{(2),\Omega}$,  and $Y:=X^{\Delta}_{\Delta}$. Suppose that $G\ne X$ but $L=L^{(2),\Sigma}$. Then the following hold.
		\begin{enumerate}
			\item $X\leq Y\wr L = B_Y\rtimes L$, and $R\leq Y\leq  R^{(2),\Delta}$;
			
			\item $X=N\rtimes G$, where $N = B_Y\cap X>1$, and $N$ consists precisely of all the elements $(y_1,\hdots,y_s)\in B_Y$ such that
			\begin{align}\label{Cruc}
				(y_i,y_j)\in (\rho_{i,j}(M_i\cap M_j))^{(1), \Delta_i\times\Delta_j}\text{ for all }1\le i<j\le s.
			\end{align}
			
			\item If $p:=|R|$ is prime then $Y=R^{(2),\Delta}=R=C_p$ and, for $2\le j\le s$,
			\begin{align}\label{Claim}
				(\rho_{1,j}(M\cap M_j))^{\Delta_1\times \Delta_j}=(Y\times Y)\cap (\rho_{1,j}(M\cap M_j))^{(1),\Delta_1\times \Delta_j}
			\end{align}
			is equal to either $Y\times Y$ or a diagonal subgroup of $Y\times Y$. In particular,  there exists
			$(y_1,\dots,y_s)\in N$ with both $y_1$ and $y_j$ nontrivial.
		\end{enumerate}
	\end{lemma}
	
	\begin{proof}
		(a) By Lemma~\ref{lem:imprim}(a), $L\leq X^\Sigma\leq L^{(2), \Sigma}=L$, and hence $X^\Sigma=L$.
		By Lemma~\ref{lem:imprim}(b), for each $\Delta_i\in\Sigma$, $R=G_{\Delta_i}^{\Delta_i}\leq X_{\Delta_i}^{\Delta_i}\leq (G_{\Delta_i}^{\Delta_i})^{(2),\Delta_i}=R_i^{(2),\Delta_i}$, and the last subgroup is $Y$: here we identify the set $\Delta_i=\Delta\times\{i\}$ with $\Delta$ as in Notation~\ref{not2}.5. Thus it follows from  Lemma~\ref{lem:imprim} that $X\leq Y\wr L  = B_Y\rtimes L$ and $R\leq Y\leq R^{(2),\Delta}$.
		
		(b)    Let $N:=B_Y\cap X$. Then $N\unlhd X$, and it follows from Notation~\ref{not2}.10 that $G\cong G^\Sigma$, equivalently, $N\cap G=1$. Thus  $|X/N|=|X^{\Sigma}|= |L|=|G^{\Sigma}|=|G|$, and therefore $X=N\rtimes G$. Since $X> G$ we have $N>1$.
		
		Let $y\in B_Y$, so $y=(y_1,\hdots,y_s)$, where each $y_i\in Y$.  Then $y\in N$ if and only if
		$y\in X=G^{(2),\Omega}$ and, by Theorem~\ref{thm:W}, this holds if and only if, for all $\alpha,\beta\in\Omega$, there exists $g\in G$ such that $(\alpha,\beta)^y= (\alpha,\beta)^g$. Suppose first that $\alpha, \beta$ lie in the same part $\Delta_i$ of $\Sigma$. Then $(\alpha,\beta)^y= (\alpha,\beta)^{y_i}$, and since $y_i\in Y=R^{(2),\Delta}$ it follows that there exists $r\in R$ with $(\alpha,\beta)^{y_i}=(\alpha,\beta)^{r}$,
		and as $R=G_{\Delta_i}^{\Delta_i}$ there exists $g\in G_{\Delta_i}$ such that $g^{\Delta_i}=r$. Hence
		$(\alpha,\beta)^y= (\alpha,\beta)^g$. Thus we may assume that $\alpha\in\Delta_i, \beta\in\Delta_j$ with $1\le i< j\le s$. Then $(\alpha,\beta)^y= (\alpha^{y_i},\beta^{y_j})\in\Delta_i\times\Delta_j$, and a suitable element $g$ would need to lie in $M_i\cap M_j$ (the setwise stabiliser in $G$ of $\Delta_i$ and $\Delta_j$), and the image  $(\alpha,\beta)^g$ would therefore be equal to $(\alpha,\beta)^a$ for some $a\in \rho_{i,j}(M_i\cap M_j)$. By the definition of the $1$-closure of $\rho_{i,j}(M_i\cap M_j)$ in its action on $\Delta_i\times\Delta_j$ (see Notation~\ref{not2}.9), such an $a=a(\alpha,\beta)$ can be found, for all $(\alpha, \beta)\in\Delta_i\times\Delta_j$, if and only if $(y_i,y_j)\in (\rho_{i,j}(M_i\cap M_j))^{(1), \Delta_i\times\Delta_j}$.
		
		(c) Suppose now that $|R|=p$  is a prime, so $|\Delta|=p$ and $R=C_p$ acting regularly. Then, by \cite[Theorem 5.12]{Wielandt}, $R^{(2),\Delta}=R$, and hence by part (a), also $Y=C_p$.
		Fix an integer $j$ with $2\leq j\leq s$, and recall that $M=M_1$. By part (b), $N\ne 1$, and hence, since $N$ is normal in $X$, there exists
		$y=(y_1,\dots,y_s)\in N$ such that  $y^{\Delta_1\cup\Delta_j}=(y_1,y_j)\ne 1$ with both $y_1$ and $y_j$ nontrivial.  In particular, $1\ne (\rho_{1,j}(y))^{\Delta_1\times \Delta_j}\in (\rho_{1,j}(M\cap M_j))^{\Delta_1\times \Delta_j}$. On the other hand $\rho_{1,j}(M\cap M_j)\leq R_1\times R_j$ by Notation~\ref{not2}.9, and in this case $R_1\times R_j=C_p\times C_p$. Thus $(\rho_{1,j}(M\cap M_j))^{\Delta_1\times \Delta_j}$ is either $C_p\times C_p$ or a diagonal subgroup $C_p$. In the first case $\rho_{1,j}(M\cap M_j)$ is transitive on   $\Delta_1\times \Delta_j$, and hence its $1$-closure is
		$(\rho_{1,j}(M\cap M_j))^{(1),\Delta_1\times \Delta_j}=\Sym(p^2)$, and \eqref{Claim} holds with both sides equal to $C_p\times C_p$. Suppose then that  $(\rho_{1,j}(M\cap M_j))^{\Delta_1\times \Delta_j}$ is a diagonal subgroup of $C_p\times C_p$, hence of order $p$ and intransitive on  $\Delta_1\times \Delta_j$.  It follows that $(Y\times Y)\cap (\rho_{1,j}(M\cap M_j))^{(1),\Delta_1\times \Delta_j}$ is also intransitive, and hence is a proper subgroup of  $Y\times Y= C_p\times C_p$. Since this group contains $(\rho_{1,j}(y))^{\Delta_1\times \Delta_j}\ne 1$, it follows that \eqref{Claim} holds with both sides equal to $C_p$. \end{proof}

	First we derive a useful consequence of Lemma \ref{Crucial}.
	\begin{lemma}\label{PrimeCor} Let $G, \Omega, \Sigma, M, R, L, \Delta$, etc. be as in
		Notation~$\ref{not2}$, and let $X:=G^{(2),\Omega}$ and $Y=X^{\Delta}_{\Delta}$. Suppose also that $G\ne X$, that $L=L^{(2),\Sigma}$, that $p:=|R|$ is prime, and that $M$ is not contained in a proper normal subgroup of $G$. Let $J$ consist of representative elements $j\in\{2, 3,\dots,s\}$, where $s=|\Sigma|$, such that
		\[
		(\rho_{1,j}(M\cap M_j))^{\Delta_1\times\Delta_j}\cong C_p,
		\]
		and such that, for distinct $j, j'\in J$, the parts $\Delta_j, \Delta_{j'}\in \Sigma$ lie in distinct $M$-orbits in $\Sigma\setminus\{\Delta\}$.
		If $\sum_{j\in J} |M:M\cap M_j|\geq s/2$, then
		the subgroup $N$ from Lemma~$\ref{Crucial}$ is the diagonal subgroup $\{(x,\dots,x) \text{ : } x\in R \}$ of $B_R=R^s\cong C_p^s$.
	\end{lemma}
	
	\begin{proof}
		Write the cyclic group $R\cong C_p$ additively, and let $N$ be the normal subgroup of $X$ from Lemma \ref{Crucial}. By Lemma~\ref{Crucial} parts~(b) and (c), $X=N\rtimes G$, $N$ is an elementary abelian $p$-subgroup of $B_Y=B_R=R^s=C_p^s$, and by Notation~\ref{not2}.10, $G\cong G^\Sigma=L$. Thus $B_R$  may be viewed as the permutation module $V$ for the $G$-action on $\Sigma$  over the field $\mathbb{F}_p$, and $N$ as a non-zero $G$-submodule. Note that $V$ is the induced module $1\!\uparrow^G_M$, and by assumption
		$M$ is not contained in a proper normal subgroup of $G$. Hence, by Lemma \ref{InducedModules}, it suffices to prove that $|N|=p$.
		
		Assume to the contrary that $|N|\geq p^2$. Then there exists a nontrivial element $y=(y_1,\hdots,y_s)\in N$ with the property that
		$y_e=0$ for some $e\le s$. Since $G^\Sigma$ is transitive, we may assume, without loss of generality, that $y_1=1$.
		Let $\mathcal{O}$ be the union of those $M$-orbits in ${\Sigma}$ containing some $\Delta_j$ with $j\in J$. Since $M\cap M_j$ is the setwise stabiliser in $M$ of $\Delta_j$, the set $\mathcal{O}$ has length $\sum_{j\in J}|M:M\cap M_j|\geq s/2$.
		By Lemma \ref{Crucial}(c) and the definition of $J$, $\rho_{1,i}(N)$ is a full diagonal subgroup of $R\times R\cong C_p^2$, for all $i\in \mathcal{O}$. Thus, there exist automorphisms $\theta_i\in\Aut(R)$ for $i\in \mathcal{O}$ such that, if $z=(z_1,\hdots,z_s)\in N$, then $z_i=z_1^{\theta_i}$ for all $i\in \mathcal{O}$.
		
		Since $y_1\neq 0$, the number of nonzero entries in $y$ is therefore at least $1+|\mathcal{O}|>s/2$, that is, strictly more than half the entries of $y$ are non-zero. On the other hand, since $y$ also contains zero entries, and $G^\Sigma$ is transitive, it follows that, for some $g\in G$, the conjugate $y^g$ has first entry $(y^g)_1=0$. Then, since $y^g\in N$, it follows from the previous paragraph that $(y^g)_i=0$ for all $i\in \mathcal{O}$, that is, strictly more than half the entries of $y^g$ are zero. However, since $y^g$ is conjugate to $y$, this implies that strictly more than half the entries of $y$ are zero. This is a contradiction. Hence $|N|=p$ and the result follows.
	\end{proof}
	
	We next note a particular example for which the condition of Lemma \ref{PrimeCor} on $J$ holds. This result will be crucial in our proof that the Thompson group $Th$ is totally $2$-closed (see Propositions~\ref{ThPrelim} and~\ref{ThProp}).

	\begin{proposition}\label{ThPropCount} Suppose that $G=\mathrm{Th}$, $M={}^3D_4(2).3$ and $H=[M,M]$, and note that all the conditions of Notation~$\ref{not2}$ hold with $\Omega=[G:H]$ and $\Sigma=[G:M]$.  Let $J$ be a
		subset consisting of representative elements $j\in\{2, 3,\dots,s\}$, where $s=|\Sigma|$, such that
		\[
		(\rho_{1,j}(M\cap M_j))^{\Delta_1\times\Delta_j}\cong C_3,
		\]
		and such that, for distinct $j, j'\in J$, the parts $\Delta_j, \Delta_{j'}\in \Sigma$ lie in distinct $M$-orbits in $\Sigma\setminus\{\Delta\}$.
		Then $\sum_{j\in J} |M:M\cap M_j|> s/2$.
	\end{proposition}
	
	\begin{proof}
		The group $G^{\Sigma}$ has rank $11$. We may assume that, for $i=2,\dots,11$, the subgroups $M\cap M_i$ are the stabilisers in $M$ of parts $\Delta_i\in\Sigma$ from pairwise distinct $M$-orbits in $\Sigma\setminus\{\Delta\}$.  These point stabilisers  $(M\cap M_i)^{\Sigma}$ are given in \cite{Mazurov}, and  are (in some order) $M\cap M_2:=2^2.[2^9].(\Sym(3)\times
		3)$, $M\cap M_3:=2^{1+8}.9.3$, $M\cap M_4:=7^2.(3\times \mathrm{SL}_2(3))$, $M\cap M_5:=[ 3^5].2$, $M\cap M_6:=2 \times \Alt(4)\times \Alt(4)$, $M\cap M_7:=(7.6)\times 3$, $M\cap M_8:= 13.6$, $M\cap M_9:=3\times \mathrm{SL}_2(3)$, $M\cap M_{10}:=3^2.6$, and $M\cap M_{11}:=\Sym(3)$.
		The groups $M\cap M_i$, for $i\in\{2,8,11\}$ clearly have no elementary abelian quotient of order $9$.
		Furthermore, it is proved in \cite{Mazurov} that the $3^2$ in $M\cap M_{10}:=3^2.6$ is self-centralising, and the `$6$ on top' acts as the unique  subgroup of $GL_2(3)$ of order $6$ (up to conjugacy). In particular $M\cap M_{10}$ has  no elementary abelian quotient of order $9$.
		Since $\sum_{i\in\{2,8,10,11\} }|M:M\cap M_i|>|\Sigma|/2$, the result is proved.
	\end{proof}

	Our next result is useful for reducing consideration of all transitive actions of a given group to the consideration of  primitive ones. We summarise some concepts used in the statement and proof:
	For a finite group $E$ and a positive integer $f$, a subgroup $D\leq E^f$ is called a \emph{diagonal subgroup} if $D$ is of the form $D=\{(e,e^{\theta_2},\hdots,e^{\theta_f})\text{ : }e\in F\}$ for some subgroup $F$ of $E$ and automorphisms $\theta_i$ of $F$; $D$ is called a \emph{full diagonal subgroup} if $F=E$. If $f=2$ and $D$ is a subdirect subgroup of $E_1\times E_2$, for finite groups $E_1, E_2$, then by  \cite[Theorem 4.8]{PS} (sometimes called `Goursat's Lemma'), $K_i := D\cap \Ker(\pi_{3-i})\unlhd E_i$ for $i=1, 2$, and
	there is a well defined homomorphism
	$\varphi_D: E_1\to E_2/K_2$ with kernel $K_1$ such that
	$D=\{(e_1, e_2) \text{ : } e_1\in E_1, e_2\in \varphi_D(e_1)\}$.
	In particular $D$ contains $K_1\times K_2$ as a normal subgroup and
	$D/(K_1\times K_2)\cong E_1/K_1\cong E_2/K_2$; and moreover $D$ is a full diagonal subgroup if and only if $K_1\times K_2=1$.

	Finally, we will write $\Omega^{(2)}$ for the set of ordered pairs of distinct points of a set $\Omega$, and, for a finite group $Y$, we will write $\mathcal{S}(Y)$ for the intersection of all nontrivial subnormal subgroups of $Y$. Note that $\mathcal{S}(Y)$ is nontrivial only if $Y$ has a unique minimal normal subgroup, which is simple.
	
	\begin{proposition}\label{RedProp}
		Let $G, \Omega, \Sigma, M, R, L, \Delta, s, \rho_{i,j}$, etc. be as in
		Notation~$\ref{not2}$, and let $X:=G^{(2),\Omega}$ and $Y=X^{\Delta}_{\Delta}$.
		Suppose that  $G\ne X$ but $L=L^{(2),\Sigma}$.
		
		\begin{enumerate}[\upshape(a)]
			\item Then $X=N\rtimes G$ with $N=X\cap B_Y\ne 1$ as in Lemma~$\ref{Crucial}$, and  $N$ is a subdirect subgroup of a base group $B_A=A^s$ for some $A\unlhd Y$ (see Notation~$\ref{not2}.7$).
			
			\item If $M$ is maximal in $G$, then either
			\begin{enumerate}[\upshape(i)]
				\item $N$ is a full diagonal subgroup of $B_A$; or
				\item for all distinct $i, j\leq s$, $\rho_{i,j}(N)$ is not a full diagonal subgroup of $A\times A$; moreover, either $N$ contains the base group $B_{\mathcal{S}(Y)}$, or  $Y$ has a unique minimal normal subgroup which is isomorphic to $C_p$, for some prime $p$.
			\end{enumerate}
			\item Furthermore if $M$ is maximal but not normal in $G$, and if in addition $p=|R|$ is prime and (b)(i) holds, then $Y=R$ and $N$ is the diagonal subgroup $\{(x,\dots,x)  \text{ : } x\in R \}$ of $B_R=R^s\cong C_p^s$.
		\end{enumerate}
	\end{proposition}
	
	\begin{proof}
		(a) Since $N$ is a subgroup of the base group $B_Y$ of $Y\wr L$ and $N$ is normalised by $X$, it
		follows that $N$ is a subdirect subgroup of a base subgroup $B_A= A^s$ for some subgroup $A\unlhd Y$, as in part (a).
		Note that each element $y\in Y$ acting by conjugation induces an automorphism of $A$, and we denote this automorphism by $\iota(y)$.
		For the rest of the proof assume that $M$ is maximal in $G$, or equivalently, $L=G^\Sigma$ is primitive. This means, by Higman's theorem \cite[(1.12)]{Higman} discussed above, that every nondiagonal $L$-orbital digraph is connected.
		
		(b)  Since $N$ is a subdirect subgroup of $B_A=A^s$, for any distinct $i, i'$, the image $\rho_{i,i'}(N)$
		is a subdirect subgroup of $A\times A$. Suppose first that there exist distinct $i,i'$ such that
		$\rho_{i,i'}(N)$ is a full diagonal subgroup of $A\times A$, that is to say, $\rho_{i, i'}(N)
		=\{(x,x^\theta) \text{ : } x\in A\}$ for some $\theta\in\Aut(A)$. If $g=y\sigma\in G<B_Y\rtimes L$
		with $y\in B_Y, \sigma\in L$ and $(j,j')=(i,i')^\sigma$, and if $z\in N$ such that $\rho_{i,i'}(z)=(x,x^\theta)$,  then $z^g\in N$ (since $G$ normalises $N$) and
		\[
		\rho_{j,j'}(z^g) = (x^{y_i}, (x^\theta)^{y_{i'}})= (u,u^{\theta'})\ \text{ where $u=x^{y_i}$ and $\theta'=\iota(y_i^{-1})\,\theta\, \iota(y_{i'})$}.
		\]
		Since $A\unlhd Y$, we have $u\in A$ and $\theta'\in\Aut(A)$, and since this holds for each $z\in N$, and hence for each $(x, x^\theta)\in\rho_{i,i'}(N)$, it follows that $\rho_{j,j'}(N)$ is a full diagonal subgroup of $A\times A$.
		
		Let $\Gamma$ be the (nondiagonal) $L$-orbital digraph containing $(i,i')$ as an edge. As mentioned above $\Gamma$ is a connected digraph, and so, since $L=G^\Sigma$ is transitive on $\Sigma$, there exists an edge of the form $(1,k)$ for some $k>1$.  We claim that $\rho_{1,j}(N)$ is a full diagonal subgroup of $A\times A$, for all $j>1$. We prove this claim by induction on the distance $d_\Gamma(1,j)$, the length of the shortest path from $1$ to $j$ in $\Gamma$.  By the argument in the previous paragraph the claim holds for all points $j$ with $d_\Gamma(1,j)=1$. Suppose that $d>1$, that the claim holds for all  $j$ with $d_\Gamma(1,j)\leq d-1$, and that $d_\Gamma(1,j)=d$. There is a path $i_0,i_1,\dots, i_d$ in $\Gamma$ with $i_0=1, i_d=j$, and set $k=i_{d-1}$. Then by the previous paragraph $\rho_{k, j}(N)
		=\{(x,x^\theta) \text{ : } x\in A\}$ for some $\theta\in\Aut(A)$, and by induction,  $\rho_{1, k}(N)
		=\{(x,x^{\theta'}) \text{ : } x\in A\}$ for some $\theta'\in\Aut(A)$, and it follows that
		$\rho_{1, j}(N)=\{(x,x^{\theta\theta'}) \text{ : } x\in A\}$. Thus the claim is proved by induction.
		It now follows that there exist automorphisms $\theta_2,\dots,\theta_s\in\Aut(A)$ such that $N=\{ (x, x^{\theta_2},\dots,x^{\theta  _s}) \text{ : } x\in A\}$ and (b)(i) holds.
		
		Thus we may assume that, for all distinct $i, j$, the subdirect subgroup $\rho_{i,j}(N)$  of $A\times A$ is not a full diagonal subgroup. This is the first assertion of part (b)(ii). As mentioned before the statement of the proposition, $\mathcal{S}(Y)$ is trivial unless $Y$ has a unique minimal normal
		subgroup, which is simple. Thus, to complete the proof of (b)(ii), we may assume that $Y$ is almost simple with simple socle $\mathcal{S}(Y)$. Hence also $A$ is almost simple with socle $\mathcal{S}(Y)$.
		If $\mathcal{S}(Y)=C_p$, for some prime $p$, then part (b)(ii) holds, so we may assume further that $\mathcal{S}(Y)$ is a nonabelian simple group.  Let $N_1:=X\cap B_{\mathcal{S}(Y)}= N\cap B_{\mathcal{S}(Y)}$. Then $N_1$ is normalised by $X$, and as in the proof of part (a),  $N_1$ is a subdirect subgroup of $B_{A_1}=A_1^s$ for some $A_1\unlhd Y$. By the definition of $N_1$, $A_1\subseteq \mathcal{S}(Y)$, so either $A_1=1$ or $A_1=\mathcal{S}(Y)$.  If $A_1=1$ then $N_1=1$, and so the insoluble group $N\cong N/N_1$ is isomorphic to a subgroup of $B_A/B_{\mathcal{S}(Y)} \leq {\rm Out}(\mathcal{S}(Y))^s$, which by the `Schreier Conjecture' is soluble, a contradiction. Thus  $A_1=\mathcal{S}(Y)$.
		%
		Then for distinct $i, j$, the image $\rho_{i,j}(N_1)= A_1\times A_1$ (since $\rho_{i,j}(N)$ is not a diagonal subgroup of $A\times A$). Further, the subdirect subgroup $N_1$ of $B_{A_1}$ is, by `Scott's Lemma' \cite[Theorem 4.16(iii)]{PS},  a direct product of pairwise disjoint full strips of $B_{A_1}$; and by our observation about the $\rho_{i,j}(N_1)$ it follows that each of the strips has support of size $1$, that is, it is one of the simple direct factors of $B_{A_1}$. Thus in this case $N_1=B_{A_1}$, and part (b)(ii) holds.
		
		(c)  Finally suppose that $M$ is maximal but not normal in $G$, that $|R|=p$ is prime, and (a)(i) holds.
		Then $Y=R$ since $R\leq Y\leq R^{(2),\Delta}=R$. Also $M$ is not contained in a proper normal subgroup of $G$, and so  the stated assertion follows from Lemma \ref{InducedModules}.
	\end{proof}

	The next result is needed to prove some important consequences of Proposition \ref{RedProp}.
	
	\begin{lemma}\label{1CLemma} For $i=1, 2$, let $\Delta_i$ be a finite set and $E_i$ a transitive subgroup of $\Sym(\Delta_i)$. Let $\Omega=\Delta_1\times\Delta_2$, and consider the product action of  $E_1\times E_2$ on $\Omega$. Suppose that $K$, a finite group, acts on $\Delta_i$ with $K^{\Delta_i}\leq E_i$, for each $i$, such that the induced $K$-action on $\Omega$ is faithful.  Let $\Gamma_i$ be the set of $K$-orbits in $\Delta_i$, and let $F_i$ be the (largest) subgroup of $E_i$  fixing setwise each $K$-orbit in $\Gamma_i$.
		Then
		\[
		(E_1\times E_2) \cap K^{(1),\Omega}\leq F_1\times F_2,
		\]
		and moreover, if $K^{\Delta_i}=F_i$, for $i=1$, $2$, then $K^{(1), \Omega}$ is a subdirect subgroup of $F_i\times F_2$.
	\end{lemma}
	
	\begin{proof} Write $\Gamma_1:=\{\Lambda_1,\hdots,\Lambda_r\}$ and $\Gamma_2:=\{\Pi_1,\hdots,\Pi_s\}$. Then the orbits of $K^{\Delta_1}\times K^{\Delta_2}$ in its action on $\Omega$ are the sets $\Lambda_i\times \Pi_j$, for $1\le i\le r$, $1\le j\le s$. Since the $1$-closure of a permutation group is the stabiliser of each of its orbits in the associated symmetric group, it follows that $(E_1\times E_2)\cap (K^{\Delta_1}\times K^{\Delta_2})^{(1),\Omega}=F_1\times F_2$. Then since $K^{\Omega}\le (K^{\Delta_1}\times K^{\Delta_2})^{\Omega}$, we deduce that $(E_1\times E_2) \cap  K^{(1),\Omega}\leq F_1\times F_2$. Finally, if $K^{\Delta_i}=F_i$, for $i=1$, $2$, then clearly $(E_1\times E_2) \cap  K^{(1),\Omega}$ is a subdirect subgroup of $F_1\times F_2$.
	\end{proof}

	We now derive a series of consequences of Proposition \ref{RedProp}. They show that the question of whether or not $G^{\Omega}$ is $2$-closed often reduces to the same question for the induced action
	of $G$ on a nontrivial block system. We use the notation and assumptions in Notation~\ref{not2}, and recall in particular that this includes the assumption that $G\cong G^\Sigma = L$.
	The \emph{base size} of a fnite transitive permutation group $L$ is the minimum size of a subset $S$ of points such that $\cap_{i\in S} L_i = 1$.

	\begin{proposition}\label{MainCor1}
		Let $G, \Omega, \Sigma, M, R, L, \Delta, s, \rho_{i,j}$, etc. be as in
		Notation~$\ref{not2}$, let $X:=G^{(2), \Omega}$ and $Y=X^{\Delta}_{\Delta}$, and
		suppose that  $L=L^{(2),\Sigma}$.
		
		\begin{enumerate}[\upshape(1)]
			\item If $|\Delta|=2$, so $R=M^\Delta = \mathbb{Z}/2\mathbb{Z}$ (written additively), then  the following are equivalent:
			\begin{enumerate}[\upshape(i)]
				\item $X\ne G$ (that is, $G^\Omega$ is not $2$-closed);
				\item  for all $i, j\leq s$, $(M_i\cap M_j)G_{\alpha_i}=M_i$, where $\alpha_i\in\Delta_i$;
				\item the element $z= (1,1,\dots,1)$ in the base group $B_R=R^s$ of $R\wr L=B_R\rtimes L$ lies in $X$.
			\end{enumerate}
			\item If $X\ne G$, then $X=N\rtimes G$ with $N$ a subdirect subgroup of a base group $B_A$ for some nontrivial $A\unlhd Y$, and all $A$-orbits in $\Delta$ have the same length, say $a$, where $a$ divides $|L_{ij}|$ for all distinct $i, j\leq s$ (where $L_{ij}$ is the stabiliser in $L$ of $\Delta_i$ and $\Delta_j$).
			
			\item If $L$ has base size $2$, then $X=G$.
			\item Suppose that $|\Delta|=p$, where $p$ is prime, so that $R=M^\Delta = \mathbb{Z}/p\mathbb{Z}$. Suppose also that $X\ne G$ (that is, $G^\Omega$ is not $2$-closed). Then for all $i, j\leq s$, and all $\alpha_i\in\Delta_i$, we have $(M_i\cap M_j)G_{\alpha_i}=M_i$.
		\end{enumerate}
	\end{proposition}
	
	\begin{proof}
		(1) (i) $\Rightarrow$ (ii):\quad Suppose that $X\ne G$ so that, by  Lemma~\ref{Crucial} (all parts of it), $Y=R$ and
		$X=N\rtimes G\leq R\wr L$, where $N$ consists of all elements $y=(y_1,\dots,y_s)\in R^s=(\mathbb{Z}/2\mathbb{Z})^s$ satisfying  \eqref{Cruc}. Suppose also that the condition in part (a)(ii) above fails. (We will derive a contradiction.) Then there exist  $i, j\leq s$ such that $(M_i\cap M_j)G_{\alpha_i}\ne M_i$, where $\alpha_i\in\Delta_i$. Since $G_{\alpha_i}$ is a subgroup of $M_i$ of index $2$, this means that $M_i\cap M_j\leq G_{\alpha_i}$. In particular $i\ne j$, and $M_i\cap M_j$ acts trivially on $\Delta_i$. Moreover, since $G$ is transitive on $\Sigma$, we may assume that $i=1$.
		By Lemma~\ref{Crucial}(c), there exists $y=(y_1,\dots,y_s)\in N$ such that both $y_1$ and $y_j$ are nontrivial. On the other hand, by Lemma~\ref{Crucial}(b), $(y_1, y_j)\in (R\times R)\cap (\rho_{1,j}(M_1\cap M_j))^{(1),\Delta_1\times\Delta_j}$, and hence $(y_1, y_j)$ fixes setwise each orbit of $M_1\cap M_j$ in its induced action on $\Delta_1\times\Delta_j$. Since $y_1\ne0$, this contradicts the fact that $M_1\cap M_j$ acts trivially on $\Delta_1$.  Thus the condition in part (a)(ii) holds.
		
		(ii) $\Rightarrow$ (iii):\quad Suppose that the condition in part (a)(ii) holds, and note that $z$ is central in $R\wr L$. Let  $U:=\langle z\rangle G$, so that $G\leq U\leq R\wr L$ and $|U:G|\leq 2$.
		We will prove that $G$ and $U$ have the same orbits in $\Omega\times\Omega$, that is $(\alpha, \beta)^G=(\alpha, \beta)^X$ for all $\alpha,\beta\in\Omega$.  If $\alpha=\beta$ then $(\alpha, \beta)^G=(\alpha, \beta)^X$ since both $X$ and $G$ are transitive on $\Omega$, so suppose that $\alpha\ne\beta$. If $\{\alpha, \beta\} = \Delta_i$ for some $i$, then, since $M_i$ interchanges $\alpha$ and $\beta$, the orbit $(\alpha,\beta)^G$ consists of all ordered pairs of points in the same block $\Delta_j$ of $\Sigma$. Hence $(\alpha, \beta)^G=(\alpha, \beta)^X$ in this case since $X$ preserves $\Sigma$, by Lemma~\ref{lem:imprim}.   Thus we may assume that  $\alpha\in\Delta_i$ and $\beta\in\Delta_j$ for distinct $i, j$.   The condition $(M_i\cap M_j)G_{\alpha_i}=M_i$ implies that
		$(M_i\cap M_j)^{\Delta_i}=M_i^{\Delta_i}=R_i$, and similarly (interchanging $i$ and $j$), $(M_i\cap M_j)^{\Delta_j}=M_j^{\Delta_j}=R_j$. Hence $\rho_{i,j}(M_i\cap M_j)$ is a subdirect subgroup of $R_i\times R_j\cong C_2\times C_2$, and in particular contains $(1,1)$. Thus there exists $x=(z_1,\hdots,z_s)\sigma\in M_i\cap M_j\le G$ such that $z_i=z_j=1$. Now
		\begin{align*}
			(R\wr L)_{\alpha \beta}= \left\{(z_1,\hdots,z_s)\sigma\text{ : }z_i=z_j=0\text{, }\sigma \in L_{ij}\right\},
		\end{align*}
		and
		\begin{align}\label{Ulab}
			U_{\alpha \beta}= G_{\alpha \beta} \cup z\{(z_1,\hdots,z_s)\sigma \in G\text{ : }z_i=z_j=1\text{, }\sigma \in L_{ij}\}.
		\end{align}
		Since $x=(z_1,\hdots,z_s)\sigma\in M_i\cap M_j\le G$ with $z_i=z_j=1$, it follows from \eqref{Ulab} that
		$zx\in U_{\alpha \beta}\setminus G_{\alpha \beta}$, and hence $|U_{\alpha \beta}|\geq 2|G_{\alpha \beta}|$. Since $[U:G]\leq 2$, we have
		\[
		|(\alpha,\beta)^U| = |U:U_{\alpha\beta}|=\frac{|U:G|. |G:G_{\alpha\beta}|}{|U_{\alpha\beta}:G_{\alpha\beta}|}\leq |G:G_{\alpha\beta}|= |(\alpha,\beta)^G|\leq |(\alpha,\beta)^U|,
		\]
		and hence equality holds. Thus we have shown that  $U\subseteq G^{(2),\Omega}=X$, and hence that $z\in X$.
		
		(iii) $\Rightarrow$ (i):\quad Finally suppose that $z\in X$. By Notation~\ref{not2}.10, $G$ is faithful on $\Sigma$, so $G\cap R^s=1$. Hence $z\in X\setminus G$, whence $X\ne G$.
		
		\medskip
		(2) Since $X\ne G$, it follows from Proposition \ref{RedProp}(a) that $X=N\rtimes G$ with $N\ne 1$, and $N$ a
		subdirect subgroup of a base group $B_A$, where $1\ne A\unlhd Y$. The $A$-orbits in $\Delta$ all have
		the same length, say $a$, since $Y=X_\Delta^\Delta$ is transitive. Choose distinct $i, j\leq s$, We apply
		Lemma \ref{1CLemma} with $K, \Delta_1,\Delta_2, E_1, E_2$ being  $\rho_{i,j}(M_i \cap M_j), \Delta_i, \Delta_j, Y, Y$, respectively. Note that the hypotheses of Lemma \ref{1CLemma} hold since $\rho_{i,j}(M_i \cap M_j)$
		acts faithfully on $\Delta_i\times \Delta_j$, and the groups it induces on $\Delta_i$ and $\Delta_j$
		are both contained in $Y$. Thus Lemma~\ref{1CLemma} implies that $Z:=(Y\times Y)\cap (\rho_{i,j}(M_i \cap M_j))^{(1),\Delta_i\times\Delta_j}$ fixes setwise each orbit of $\rho_{i,j}(M_i \cap M_j)$ in $\Delta_i\cup\Delta_j$. Further, it follows from Lemma~\ref{Crucial}(b) that $\rho_{i,j}(N)\leq Z$, and since
		$\rho_{i,j}(N)$ is a subdirect subgroup of $A\times A$, this implies that each orbit of $\rho_{i,j}(M_i \cap M_j)$ in $\Delta_i\cup\Delta_j$ is a union of $A$-orbits, and hence has length a multiple of $a$. In particular $a$ divides $|M_i\cap M_j|$. Finally, since $G\cong G^\Sigma=L$ we have $M_i\cap M_j\cong L_{ij}$. This completes the proof.
		
		\medskip
		(3) Suppose that $L=G^{\Sigma}$ has base size $2$ and let $\{i,j\}$ be a base, that is, $L_{ij}=1$. This implies, since $G\cong L$, that $M_i\cap M_j\cong L_{ij}=1$. It now follows from part (2) that $X=G$.
		
		\medskip
		(4) To prove part (4), the argument is identical to the (i) $\Rightarrow$ (ii) argument in (1) above, with $2$ replaced by $p$.
	\end{proof}

	\section{Total $2$-closure of sporadic simple groups}\label{SporadicSection}
	
	A major objective for the remainder of the paper is to determine precisely which sporadic simple groups 
	are totally 2-closed. In the next few sections we will complete the proof that  Theorem \ref{thm:J} holds if $T$ is a sporadic simple group. In Proposition \ref{lem:simple1}(b), we 
	reduced the list of possible candidates to just eleven sporadic groups. In this section we will eliminate 
	the groups $ON, \mathrm{Co}_2, \mathrm{Fi}_{23}, \mathrm{HN}, \mathbb{B}$, and we will show that five of the remaining six groups,
	namely $\mathrm{J}_1$, $\mathrm{J}_3$, $\mathrm{Ly}$, $\mathrm{Th}$, $\mathbb{M}$, are totally $2$-closed. The remaining group, namely 
	the fourth Janko group $\mathrm{J}_4$, requires additional effort and will be treated in Section \ref{J4Section}.
	
	Our broad strategy is to apply Proposition \ref{MainCor1} to the eleven sporadic groups listed in 
	Proposition \ref{lem:simple1}(b). To do this, we require detailed information, firstly, on the 
	greatest common divisors of the orders of the two-point stabilisers in transitive actions of 
	these groups; and secondly, on the base sizes of these transitive actions. Our main technical result is the following.
	
	\begin{proposition}\label{Base2Lemma}
		Let $T$ be one of the following simple groups: $\mathrm{J}_1, \mathrm{J}_3, \mathrm{J}_4, \mathrm{Ly}$, $\mathrm{Th}$, or $\mathbb{M}$. Further, 
		let $M$ be a proper subgroup of $T$, so that $T$ acts faithfully and transitively by right multiplication on the right coset space $[T:M]$. Then
		\begin{enumerate}[\upshape(a)]
			\item either $T^{[T:M]}$ has base size at most $2$, or $(T,M)$ is as in Table $\ref{tab:basesize}$, (where the last two subgroups in the $\mathrm{J}_4$-line may possibly have base size  $2$).
			\item Further, for each  $(T, M)$  in Table $\ref{tab:basesize}$, let $d(T,M)$ be
			the  greatest common divisor of the orders of the stabilisers in $M^{[T:M]}$ of points of $[T:M]\setminus\{M\}$. Then $d(T,M)$ divides $g(T,M)$, where $g(T,M)$ is as in the third column of Table $\ref{tab:basesize}$. Moreover, if $T=\mathrm{J}_4$ and $M=2^{10}:\mathrm{L}_5(2)$, then $d(T,M)\le 30$.
		\end{enumerate}
	\end{proposition}

	\begin{table}[]
		\centering
		\begin{tabular}{c|c|c}
			$T$  & $M$ &  $g(T,M)$\\
			\hline
			$\mathrm{J}_1$    &   $L_2(11)$    &     $1$          \\
			$\mathrm{J}_3$    &   $L_2(16)$,\ \ $L_2(16).2$    &  $1$,\ $2$             \\
			$\mathrm{J}_4$    &   $2^{11}:\mathrm{M}_{24}$,\ \ $2^{1+12}.3.\mathrm{M}_{22}.2$,\ \ $2^{10}:\mathrm{L}_5(2)$,    & $24$,\ $2$,\  $2^4.3^2.5.7$, \\
			& $2^{1+12}.3.\mathrm{M}_{22}$,\ \ $2^{1+12}.3.\mathrm{L}_3(4)$,\ \ $2^{1+12}.3.\mathrm{L}_3(4).2$ & $2$,\ $2$,\ $2$  \\
			$\mathrm{Ly}$    &   $G_2(5)$,\ \ $3.\mathrm{McL}$,\ \ $3.\mathrm{McL}.2$    & $48$,\ $30$,\ $30$               \\
			$\mathrm{Th}$    &   $2^5.\mathrm{L}_5(2)$,\ \ ${}^3D_4(2)$,\ \ ${}^3D_4(2).3$    &  $1$,\ $6$,\ $6$             \\
			$\mathbb{M}$    &   $2.\mathbb{B}$    &     $2090188800$         \\
			\hline
		\end{tabular}
		\caption{Information for non-base $2$ actions: the greatest common divisor of the stabiliser orders in $M^{[T:M]}$ of points in $[T:M]\setminus\{M\}$ divides $g(T,M)$.}
		\label{tab:basesize}
	\end{table}

	The proof of Proposition~\ref{Base2Lemma} is spread over six propositions: one for each of the groups $T=\mathrm{J}_1, \mathrm{J}_3, \mathrm{J}_4,\mathrm{Ly}, \mathrm{Th}$, and $\mathbb{M}$. Most of these propositions will prove that the relevant group $T$ is totally $2$-closed, and also that Proposition~\ref{Base2Lemma} holds for $T$, namely Propositions~\ref{J1Prop}, \ref{J3Prop}, \ref{LyProp}, and \ref{ThProp} for $\mathrm{J}_1, \mathrm{J}_3, \mathrm{Ly}$ and $\mathrm{Th}$, respectively. For the monster however, we give the proof of Proposition~\ref{Base2Lemma}  in Proposition~\ref{MonsterLemma}, and prove total $2$-closure in Proposition~\ref{MProp}.
	For the fourth Janko group $\mathrm{J}_4$, we give the proof of Proposition~\ref{Base2Lemma}  in Proposition~\ref{J4LemmaCor2}, and prove total $2$-closure in Proposition~\ref{J4prop}.
	Proposition~\ref{Base2Lemma} will be applied in Section \ref{DirectProductsSection} where we study the total $2$-closure of direct products of sporadic simple groups.
	
	
	The group $T^{[T:M]}$ has base size $1$ if and only if $M=1$, so we will always assume that $1<M<T$. 
	For the base sizes of the sporadic groups in their primitive actions, our main reference is \cite[Tables 1 and 2]{BOBW}. However, the subgroup $M$ occurring in  Notation \ref{not2}.4 is not always maximal in the finite group under investigation, and so we need also to study base sizes of imprimitive actions of sporadic simple groups.
	The following preliminary result from \cite{BOBW} allows us to do this.
	
	
	\begin{proposition}\textup{\cite[Corollary 2.4]{BOBW}}\label{BOBWProp}
		Let $G$ be a transitive permutation group with point stabiliser $H$, and for a positive integer $c$ let
		$$
		\widehat{Q}(G, H, c):=\sum_{i=1}^m\frac{|x_i^G\cap H|^c}{|x_i^G|^{c-1}}
		$$
		where $x_1$, $\hdots$, $x_m$ are representatives for the conjugacy classes of elements of $G$ of prime order. If $\widehat{Q}(G, H, c)<1$, then $G$ has base size at most $c$.
	\end{proposition}
	
	Lemma \ref{Base2Lemma} and Proposition \ref{BOBWProp} allow us to apply Proposition \ref{MainCor1} parts (2) and (3) to the relevant sporadic simple groups, to determine whether or not certain of their imprimitive actions are $2$-closed. Studying their primitive actions needs a different approach, and for this we use the following lemma. Recall that, for a group $G$ acting on a set $\Omega$, we denote by $G^\Omega$ the subgroup of $\Sym(\Omega)$ induced by $G$. Also, we say that permutation groups $G^{\Omega_1}$ and $H^{\Omega_2}$ on finite sets $\Omega_1$ and $\Omega_2$ are \emph{permutationally isomorphic} if there exists a bijection $f:\Omega_1\rightarrow\Omega_2$ and a group isomorphism $\varphi:G^{\Omega_1}\to H^{\Omega_2}$  such that $f(\alpha^g)=f(\alpha)^{g\varphi}$ for all $\alpha\in\Omega_1$, $g\in G^{\Omega_1}$. 
	
	\begin{lemma}\label{BaseLemma}
		Let $T$ be a finite simple group with proper subgroup $M$, and let $\Omega=[T:M]$ so that $T\leq \Sym(\Omega)$, acting by right multiplication, and $M=T_\alpha$ for the point $\alpha=M\in\Omega$.
		\begin{enumerate}[\upshape(a)]
			\item If also $T< G\leq\Aut(T)$, and $M=T\cap K$ for some $K\leq G$ such that  $TK=G$ and $G^{[G:K]}$ has base size $2$, then the actions of $T$ on $\Omega$ and on $[G:K]$ are permutationally isomorphic, and  $G^{[G:K]}\not\leq T^{(2),[G:K]}$.
			
			\item Suppose now that $M$ is maximal in $T$,  that $T^{[T:M]}$ is not $2$-transitive, and that $(T,|T:M|)$ is {\bf not} one of the following:
			\begin{itemize}
				\item $T=\mathrm{M}_{11}$, $\mathrm{M}_{12}$, $\mathrm{M}_{23}$, $\mathrm{M}_{24}$, or $\Alt(9)$, with $|T:M|=55$, $66$, $253$, $276$, or $120$, respectively; or
				\item $T=G_2(q)$ with $q\geq 3$ a prime power and $|T:M|=\frac{1}{2}q^3(q^3-1)$; or
				\item $T=\Omega_7(q)$ with $q$ a prime power and $|T:M|=\frac{1}{(2,q-1)}q^3(q^4-1)$.
			\end{itemize}
			Then the following hold:
			\begin{enumerate}[\upshape(i)]
				\item if $T=\Aut(T)$, then $T=T^{(2),\Omega}$;
				\item if $M\neq T\cap K$ for any $K\leq \Aut(T)$ such that  $K\ne M$, then $T=T^{(2),\Omega}$;
				\item if, for all  $K\leq \Aut(T)$ such that $M=T\cap K\ne K$, the group $(TK)^{[TK:K]}$ has base size $2$, then $T=T^{(2),\Omega}$.
			\end{enumerate}
		\end{enumerate}
	\end{lemma}
	
	\begin{proof}
		(a)  Let $M, T, K, G$ be as in part (a).
		Since $G^{[G:K]}$ is transitive with base size $2$, there exists a base
		$\{K, Kx\}$ containing the `point' $K$, for some $x\in G$, and hence the stabiliser in $G$ of these two points, namely $K\cap K^x$, is trivial. Further, since $G=TK=KT$, we may assume that $x\in T$.
		The equality $G=TK$ also implies that $T$ is transitive on $[G:K]$ and the stabiliser in $T$ of the `point' $K\in [G:K]$ is $T\cap K=M$. Thus the actions of $T$ on $\Omega$ and on $[G:K]$ are permutationally isomorphic. Since $M=T\cap K$, we have also $M\cap M^x = T\cap K\cap (T\cap K)^x= T\cap K\cap K^x=1$, that is, the stabiliser in $M$ of the point $Kx\in [G:K]$ is trivial, so the $M$-orbit containing $Kx$ has size
		$|M:M\cap M^x|=|M|$, while the $K$-orbit containing $Kx$ has size $|K:K\cap K^x|=|K|$. Since $|K:M|=|TK:T|=|G:T|>1$, it follows that $K$ does not leave invariant the $M$-orbit containing $Kx$, and hence
		$G$ does not leave invariant the $2$-relation $(K, Kx)^T$ of $T$. Hence $G^{[G:K]}\not\leq T^{(2),[G:K]}$.
		
		(b) From now on assume that $M$ is maximal in $T$, and that $T$ is as in part (b), so  $T^\Omega$ is primitive but not $2$-transitive.   Then
		by \cite[Theorem 1]{LPS}, the group $T^{(2),\Omega}$ is contained in $N:=N_{\Sym(\Omega)}(T^\Omega)$. Since $T^\Omega$ is primitive, $N$ is almost simple with socle $T^\Omega$. In particular $T^\Omega\leq N\leq \Aut( T^\Omega)$, and part (b)(i) follows. If $T^{(2),\Omega}$ properly contains $T^\Omega$ then, taking $K=(T^{(2),\Omega})_\alpha$, we have $T^{(2),\Omega}=T^\Omega K$ since $T^\Omega$ is transitive,  $K\leq \Aut(T^\Omega)$ since $N$ has socle $T^\Omega$, and $M\cong M^\Omega=T^\Omega\cap K=T^\Omega_\alpha \ne K$.
		Thus in part (b)(ii) we conclude that $T^{(2),\Omega}=T$.

		Finally, suppose that the assumptions of part (b)(iii) hold, and that $T^{(2),\Omega}\ne T^\Omega$. The previous paragraph shows that $K=(T^{(2),\Omega})_\alpha$ satisfies $K\leq \Aut(T^\Omega)$ and
		$M\cong M^\Omega=T^\Omega\cap K\ne K$. Further, since $T^\Omega$ is transitive we have $T^{(2),\Omega}
		= T^\Omega K$, and so, by the assumptions in (b)(iii), $T^{(2),\Omega}$ has base size $2$. However part (a), applied with $G=T^{(2),\Omega}$, then gives a contradiction.
	\end{proof}
	
	Before proceeding further with the general analysis, we illustrate how to use Lemma~\ref{BaseLemma} by proving the following useful corollary.
	
	\begin{corollary}\label{BaseCor}
		Each primitive permutation group isomorphic to $\mathrm{J}_3$ or $\mathrm{ON}$ is $2$-closed.
	\end{corollary}
	
	\begin{proof} Let $T$ be $\mathrm{J}_3$ or $\mathrm{ON}$, let $M$ be a maximal subgroup of $T$, and consider $T$ as a primitive permutation group on $\Omega=[T:M]$ acting by right multiplication. Up to permutational isomorphism, each primitive representation of $T$ arises in this way.  We need to prove that $T$ is equal to $X:=T^{(2),\Omega}$. By \cite[Theorem 1]{LPS}, $X$ is almost simple with socle $T$. In each case $G:=\Aut(T)\cong T.2$, so $X$ is equal to either $T$ or $G$. Note that $T$ has no $2$-transitive representations.
		If $M$ cannot be written as $M=K\cap T$ for any $K\leq G$ with $K\ne M$, then $X=T$ by Lemma~\ref{BaseLemma}(b)(ii). So assume that such a $K$ exists, and note that $TK=G$. By Lemma \ref{BaseLemma}(b)(iii), we may assume further that $G$ has base size larger than $2$ in its action on $\Sigma:=[G:K]$.
		Using \cite[Tables 1 and 2]{BOBW}, the possibilities are as follows.
		
		If $G=\mathrm{J}_3:2$, then $K\in\{L_2(16):4,(3\times M_{10}):2\}$. In these cases, we use GAP \cite{GAP}. If $K=L_2(16)$, then $G^\Sigma$ has rank $7$, while $T^\Sigma$ has rank $8$. If $K=(3\times M_{10}):2$, then $G^\Sigma$ has rank $18$, while $T^\Sigma$ has rank $28$. Thus, $T^{\Sigma}=T^{(2),\Sigma}$  in each case.
		%
		
		If $G=ON:2$, then $K=4.\mathrm{L}_3(4).2.2$. Here, we check using GAP \cite{GAP} that $G^\Sigma$ has rank $34$, while $T^\Sigma$ has rank $48$. So $T^{\Sigma}=T^{(2),\Sigma}$.
	\end{proof}
	
	We are now ready to outline our strategy to prove Theorem \ref{thm:J}. Our analysis uses the notion of the \depth $\ell_G(H)$ of a proper subgroup $H$ of a finite group $G$, namely the least positive integer $\ell$ such that there exists a subgroup chain $H=G_\ell < \dots < G_1<G_0=G$ with $G_i$ maximal in $G_{i-1}$ for $i=1,\dots, \ell$. The maximum value of $\ell_G(H)$ over all proper subgroups $H$ of $G$ is the \depth of the trivial subgroup, and this is at most the length of the longest subgroup chain in $G$ but may be smaller. For example, $1<C_2<C_2^2<\Alt(4)<\Sym(4)$ is a subgroup chain for $\Sym(4)$ of maximum length four, while $\ell_{\Sym(4)}(1)=3$ in view of the chain $1<C_3<\Sym(3)<\Sym(4)$.
	
	\begin{Strategy}\label{Strategy}
		Assume that $T$ is one of the sporadic simple groups in Proposition~\ref{lem:simple1}(b) so, in particular, 
		$T$ admits no non-trivial factorizations. By Corollary \ref{TransRed}, $T$ is totally $2$-closed if 
		and only if $T^{[T:H]}=T^{(2),[T:H]}$ for all proper subgroups $H$ of $G$. It follows from 
		Theorem~\ref{thm:W} that this holds for $H=1$, so we need to consider proper nontrivial 
		subgroups $H$, and in these cases we note that $T^{[T:H]}$ has base size at least $2$.
		
		We  proceed in a series of steps. At Step $\ell$ we will have remaining a non-empty set 
		$\mathcal{S}_\ell$ of proper nontrivial subgroups of $T$ to consider, where  $\mathcal{S}_1$ is
		the set of all proper nontrivial subgroups of $T$.
		Either (a) we complete the analysis at Step $\ell$ by concluding with certainty that $T$ is, or is not, 
		totally $2$-closed, or (b) we remove some subgroups from $\mathcal{S}_\ell$ leaving a set 
		$\mathcal{S}_{\ell+1}$ which, if non-empty, will be  considered in Step $\ell+1$. In particular, 
		for each $\ell$, $\mathcal{S}_\ell$ will have the following property:
		
		\begin{center}
			($\ast$)\quad
			If $H\in \mathcal{S}_\ell$, then $\ell_T(H)\geq \ell$, and if $\ell_T(H)= \ell$ and $H<M<T$, then
			$T^{[T:M]}=T^{(2),[T:M]}$ and $T^{[T:M]}$ has base size at least $3$.
		\end{center}
		
		\noindent
		Note that $\mathcal{S}_1$ has property ($\ast$) since  if $\ell_T(H)= 1$ then $H$ is maximal and the last condition holds vacuously. Also the process will complete in a finite number of steps since $T$ is finite. Finally, we describe what we do in Step $\ell$ (for an arbitrary value of $\ell$).
		
		\medskip\noindent
		\emph{Step $\ell$:}\quad For each $H\in \mathcal{S}_\ell$ with $\ell_T(H)= \ell$, we determine $T^{(2),[T:H]}$. To do this when $\ell=1$ (that is, when $H$ is maximal in $T$), we use Lemma \ref{BaseLemma} together with some known results about the subdegrees of the primitive almost simple groups.
		If $\ell\geq2$ then  $T^{[T:H]}$ is imprimitive, and by  property ($\ast$), $T^{[T:M]}=T^{(2),[T:M]}$ for all $M$ such that $H<M<T$. We can therefore use Proposition \ref{MainCor1} in our analysis. (This of course is the most difficult part of the procedure.)
		
		\begin{enumerate}[\upshape(1)]
			\item Thus for each  $H\in \mathcal{S}_\ell$ with $\ell_T(H)= \ell$, we determine $T^{(2),[T:H]}$.
			\begin{enumerate}[\upshape(a)]
				\item If $T^{[T:H]}<T^{(2),[T:H]}$, for some such $H$, then \emph{we conclude that $T$ is not totally $2$-closed, and the analysis is completed.}
				\item Else we have $T^{[T:H]}=T^{(2),[T:H]}$ for every $H\in \mathcal{S}_\ell$ with $\ell_T(H)= \ell$.
			\end{enumerate}
			\item Next we check the base size of $T^{[T:H]}$ for each $H\in \mathcal{S}_\ell$ with $\ell_T(H)= \ell$, and construct the new set $\mathcal{S}_{\ell+1}$.
			\begin{enumerate}[\upshape(a)]
				\item If, $T^{[T:H]}$ has base size $2$, then it follows from Proposition \ref{MainCor1}(3) that, for each $K<H$, we have $T^{[T:K]}=T^{(2),[T:K]}$. Thus proper subgroups of $H$ need no further testing.
				
				\item We remove from $\mathcal{S}_\ell$  each $H\in \mathcal{S}_\ell$ with $\ell_T(H)= \ell$ and, in addition, if $T^{[T:H]}$ has base size $2$, we also remove all proper subgroups of $H$. Let $\mathcal{S}_{\ell+1}$ denote the resulting set of subgroups.
				
				\item  If $\mathcal{S}_{\ell+1}=\emptyset$ then we have successfully shown that $T^{[T:H]}=T^{(2),[T:H]}$ for every proper subgroup $H$ of $T$, \emph{we conclude that $T$ is totally $2$-closed, and the analysis is completed.}
				
				\item On the other hand if  $\mathcal{S}_{\ell+1}\ne\emptyset$, then by its construction it is not difficult to check that $\mathcal{S}_{\ell+1}$ has property ($\ast$) with $\ell+1$ in place of $\ell$.
			\end{enumerate}
		\end{enumerate}
	\end{Strategy}
	
	We begin by applying this strategy to the group $\mathrm{HN}$, and here we do not need to proceed past stage (1) of Step 1 in Strategy \ref{Strategy}.
	
	\begin{proposition}\label{HNProp}
		The Harada Norton group $\mathrm{HN}$ is not totally $2$-closed.
	\end{proposition}
	
	\begin{proof}
		Let $T=\mathrm{HN}$, the Harada Norton group. One of the groups $H\in\mathcal{S}_1$ in Strategy~\ref{Strategy} is $H=4.\mathrm{HS}$, and $\Aut(T)=T.2$ has a primitive permutation representation on a set of size $1,539,000$, with point stabiliser $H.2$. For the normal subgroup $T$, this representation is permutationally isomorphic to its transitive action on $[T:H]$, which has rank $9$, with all subdegrees distinct, see \cite[Lemma 2.18.1]{Ivanov95}. Thus, in particular, $H.2$ must leave invariant each $H$-orbit in $[T:H]$. Hence $\Aut(T)$ also has rank $9$ on $[T:H]$. We conclude that $T^{(2),[T:H]}$ contains
		$T.2$, so $T$ is not totally $2$-closed.
	\end{proof}
	
	We now move on to the group $\mathrm{J}_1$, where we must use both stages of  Steps 1 and 2 of Strategy \ref{Strategy}.
	\begin{remark}
		At various stages in the remainder of the paper, we will us the following fact:  if $K$ is a proper nontrivial subgroup  of a group $T$ and $K$ is contained in a core-free
		subgroup $H$ of $T$ such that $T^{[T:H]}$ has base size $2$, then also $T^{[T:K]}$ has base size $2$. (This follows from the fact that  $T^{[T:H]}$ having a base of size $2$ means that, for some $g\in T$, $H\cap H^g$ acts trivially on $[T:H]$ and hence $H\cap H^g=1$ since $H$ is core-free in $T$, whence also its subgroup $K\cap K^g$ is trivial.)
	\end{remark}

	
	\begin{proposition}\label{J1Prop}
		The Janko group $\mathrm{J}_1$ is totally $2$-closed. Also  Proposition~$\ref{Base2Lemma}$, and in particular Table $\ref{tab:basesize}$,  holds for $\mathrm{J}_1$.
	\end{proposition}
	
	\begin{proof} Let $T=\mathrm{J}_1$. According to Strategy~\ref{Strategy} our task is to prove that $T^{\Omega}
		=T^{(2),\Omega}$, where $\Omega=[T:H]$, for each proper nontrivial subgroup $H$ of $T$. In Step 1 of Strategy~\ref{Strategy}, for each maximal subgroup $H$, we conclude that  $T^{[T:H]}
		=T^{(2),[T:H]}$ by Lemma \ref{BaseLemma}(b)(i),  since $\Aut(T)=T$ and $T$ has no $2$-transitive representations. Proceeding to stage (2) of Step 1, we find by \cite[Table 1]{BOBW} that, for a maximal subgroup $H$,  $T^{[T:H]}$ has base size 2 unless $H$ is conjugate to $L_2(11)$.
		Thus  the new subset $\mathcal{S}_2$ consists, up to conjugacy, of all proper nontrivial subgroups of a maximal subgroup $M:=L_2(11)$.
		
		We proceed to stage (1) of Step 2 in Strategy~\ref{Strategy}. Let $H$ be a maximal subgroup of $M$. Then  the group $T^{[T:H]}\cong T$ is imprimitive, and we use the notation in Notation \ref{not2} with $G=T$ and $\Omega=[T:H]$. So $H=G_\omega$, for a point $\omega$ contained in a part $\Delta$ in a $G$-invariant partition $\Sigma$ of $\Omega$,  $X =G^{(2),\Omega}$, $Y =X_{\Delta}^{\Delta}$, $L=G^\Sigma$, and we may assume that $G_\Delta$ (setwise stabiliser) is the subgroup $M=L_2(11)$. Thus  $L=L^{(2),\Sigma}$ by the previous paragraph. Using the Web Atlas~\cite{WebAtlas}, we see that the $M$-orbits in $\Sigma$ have lengths $1, 11, 12, 110, 132$. Thus since $|M|/110=6$ and $|M|/132=5$,  the greatest common divisor of the point-stabiliser orders in $M^\Sigma$ is $1$. Thus the first line of Table \ref{tab:basesize} holds, and it follows from Proposition \ref{MainCor1}(2) that $X=G$, that is,
		$T^{[T:H]}=T^{(2), [T:H]}$. Next, for each maximal subgroup $H$ of $M$, we use \texttt{Magma}~\cite{MAGMA} to compute the expression $\widehat{Q}(T,H,2)$ of Proposition~\ref{BOBWProp} and find that the value is less than $1$, and hence that $T^{[T:H]}$ has base size $2$.
		Thus the subset $\mathcal{S}_3$ constructed in stage (2) of Step 2 is the emptyset, and we conclude that $T$ is totally $2$-closed.
		
		To complete the proof that Proposition~$\ref{Base2Lemma}$ holds for $T=\mathrm{J}_1$, note that our argument has so far shown that each proper nontrivial subgroup $K$ of $T$ is either conjugate to $M=L_2(11)$, or is contained in a subgroup $H$ such that $T^{[T:H]}$ has base size $2$. In the latter case, as we observed above, $T^{[T:K]}$ also has base size $2$. This completes the proof.
	\end{proof}
	
	Next, we deal with the group $\mathrm{J}_3$, where three steps of Strategy \ref{Strategy} are required.
	
	\begin{proposition}\label{J3Prop}
		The Janko group $\mathrm{J}_3$ is totally $2$-closed. Moreover, Proposition~$\ref{Base2Lemma}$, and in particular Table $\ref{tab:basesize}$  holds for $\mathrm{J}_3$.
	\end{proposition}
	
	\begin{proof} Let $T:=\mathrm{J}_3$. According to Strategy~\ref{Strategy} our task is to prove that
		$T^\Omega=T^{(2),\Omega}$, where $\Omega=[T:H]$, for each proper nontrivial subgroup $H$ of $T$.
		
		In Step 1 of Strategy~\ref{Strategy},  $\mathcal{S}_1$ is the set of all proper nontrivial subgroups. If $H\in \mathcal{S}_1$ is maximal in $T$ then, by Corollary~\ref{BaseCor},  $T^\Omega=T^{(2),\Omega}$, and by \cite[Table 1]{BOBW}, $H\cong L_2(16):2$ is the only maximal subgroup such that $T^\Omega$ has base size greater than $2$ (and in fact has base size $3$). Thus the set $\mathcal{S}_2$ constructed in Strategy~\ref{Strategy} consists of all nontrivial, non-maximal subgroups $H$ such that, if $H<M<T$ with $M$ maximal, then $M$ is conjugate to $L_2(16):2$.
		
		We see from the Web Atlas \cite{WebAtlas} that the $M$-orbits in $[T:M]$ have lengths $1, 85, 120, 510, 680, 1360, 1360, 2040$. Thus, since $|M|/1360=6$ and $|M|/120=68$, it is easily checked that the greatest common divisor of the point stabiliser orders in $M^{[T:M]}$ is $2$. The second line of Table \ref{tab:basesize} then follows for $M= L_2(16):2$.
		From now on the subgroups $H$ we consider are not maximal, and we use  Notation \ref{not2} with $G=T$, $\Omega=[T:H], X=T^{(2),\Omega}, H=G_\omega$, with $\omega$ in a part $\Delta$ of a $G$-invariant partition $\Sigma$ of $\Omega$ such that $M=G_\Delta$ (setwise stabiliser), so $\Sigma = [T:M]$ and for $L=G^\Sigma$, we will always have $L=L^{(2),\Sigma}$ by Strategy~\ref{Strategy}.

		In Step $2$, we consider all maximal subgroups $H$ of $M=L_2(16).2$. Thus $M^\Delta$ is primitive. First we claim that, for each such $H$, $T^\Omega=T^{(2),\Omega}$. Suppose that this is false, so that, for some $H$, we have $X\ne G$. Then, by Proposition \ref{MainCor1}(2), $X=N\rtimes G$ with $N = X_{(\Sigma)}\ne 1$ (kernel of the $X$-action on $\Sigma$), and the $N^\Delta$-orbits have constant length $a$, where $a$ divides the greatest common divisor of the $2$-point stabilisers  of $L$.  By the previous paragraph $a$ divides $2$. Since $1\ne N^\Delta\unlhd X_\Delta^\Delta$ and $M^\Delta\leq X^\Delta$ is primitive, it follows that $N^\Delta$ is transitive, so $a=|\Delta|=2$. This implies that $|M:H|=2$ so $H$ is equal to the derived subgroup $M'=L_2(16)$.
		We check, using  \texttt{Magma}~\cite{MAGMA}, that the stabiliser (a subgroup of order $6$) in $M$ of a part $\Delta_j$ in one of the two $M$-orbits of length 1360 in $\Sigma$  lies in $H$, while the stabiliser of a part $\Delta_k$ in the other $M$-orbit of length 1360 is not contained in $H$. Firstly we note that, for $\omega_k\in\Delta_k$, this implies that the two-point stabiliser $G_{\omega, \omega_k}=H_{\omega_k} \cong C_3$, and it follows that the greatest common divisor of the point stabiliser orders in $H^{[T:H]}$ is $1$, and hence  Line 2 of Table \ref{tab:basesize}  holds also for the subgroup $L_2(16)$. Secondly, this means  that, in the notation of Proposition \ref{MainCor1}, $(M\cap M_j)G_\omega = H\ne M$ so that property (ii) in Proposition \ref{MainCor1}(1) fails, and hence $X=G$, which is a contradiction. Thus $T^{[T:H]}=T^{(2),[T:H]}$ for all maximal subgroups $H$ of $M$. For those subgroups $H\ne L_2(16)$, we use \texttt{Magma}~\cite{MAGMA} and compute that the value of the expression $\widehat{Q}(T,H,c)$ of  Proposition \ref{BOBWProp}, with $c=2$, is strictly less than $1$, and hence $T^{[T:H]}$ has base size $2$.
		
		Suppose now that $H=L_2(16)$. We compute as follows using {\sf GAP}~\cite{GAP} to prove that  $T^{[T:H]}$ has base size $3$. The generators for $T=\mathrm{J}_3$ viewed as the primitive group $T^{[T:M]}$ acting on $\Sigma=[T:M]$ are given in the ATLAS~\cite{Atlas, WebAtlas} and using these we compute with \cite{GAP} that the  point stabilisers of $H^\Sigma$ have orders $8, 6, 6, 6, 6, 2, 48, 6, 34$. This implies in particular that $H$ is transitive on the $M^\Sigma$-orbit $\Sigma_0$ of length $2040=|H|/2$. Let $K=H\cap M^g\cong C_2$, the stabiliser in $H$ of a block $\Delta'\in\Sigma_0$.
		If $T^{[T:H]}$ were to have base size $2$, then $H^{[T:H]}$ would have an orbit $\Omega_0$ of length $|H|=4080$ in $\Omega=[T:H]$, which would have to be the union of the $2040$ blocks of $\Sigma$ in $\Sigma_0$. If this were the case then the stabiliser $J:=H\cap H^g$ in $H$ of a point of $\Omega_0$ would be trivial.  Thus we (only) need to check that $J\ne 1$, and we do this using {\sf GAP}.  Now $K=H_{\Delta'}$, the stabiliser in $H$ of $\Delta'\in\Sigma_0$. In {\sf GAP} we construct $T_{\Delta'}$ (which is $M^g$) and then the  derived subgroup of $T_{\Delta'}$ (which is $H^g$).  Finally we check (using the `IsSubgroup' command in {\sf GAP}) that $K$ is a subgroup of $H^g$. Since $K$ by definition lies in $H$ we conclude that $K\leq H\cap H^g=J$, and hence $J \neq 1$, as claimed. Thus  $T^{[T:H]}$ has base size $3$.
		To complete stage (2) of Step 2 for Strategy~\ref{Strategy} we remove various subgroups from $\mathcal{S}_2$ to form $\mathcal{S}_3$: which consists, up to conjugacy, of those proper nontrivial subgroups $H$ of $M'=L_2(16)$ such that the only maximal subgroup of $M$ containing $H$ is  $M'$.
		
		In Step 3, for each $H\in\mathcal{S}_3$ such that $H$ is maximal in $M'=L_2(16)$, we use \texttt{Magma}~\cite{MAGMA} and compute that the value of the expression $\widehat{Q}(T,H,c)$ of  Proposition \ref{BOBWProp}, with $c=2$, is strictly less than $1$, and hence $T^{[T:H]}$ has base size $2$.
		Finally we prove that $T^{[T:H]}=T^{(2),[T:H]}$ by applying Proposition~\ref{MainCor1}(2) with
		$M'=L_2(16)$ instead of $M$, so $L=T^{[T:M']}=T^{(2),[T:M']}$. If $X\ne G$ then by  Proposition~\ref{MainCor1}(2), $X=N\rtimes G$ with $N$ leaving $\Delta$ invariant, and all $N$-orbits in $\Delta$ of constant length $a>1$ such that $a$ divides the greatest common divisor of the 2-point stabiliser orders in $T^{[T:M']}$. In the previous paragraph we proved that this greatest common divisor is equal to 1, which is a contradiction. Thus $T^{[T:H]}=T^{(2),[T:H]}$ with base size $2$,  for all  $H\in\mathcal{S}_3$  maximal in $M'=L_2(16)$. When we form the set $\mathcal{S}_4$ we obtain an empty set, and so we conclude that $T=\mathrm{J}_3$ is totally $2$-closed.
		
		We note that this completes the proof also of Proposition~\ref{Base2Lemma} for $T=\mathrm{J}_3$, for if $H<T$ and $T^{[T:H]}$ has base size at least $3$, then also, for all $H<M<T$, $T^{[T:M]}$ has base size at least $3$. Thus $H\leq M=L_2(16):2$, and if $H\ne M$ then it follows from the proofs of Steps 2 and 3 above that $H=M'$.
	\end{proof}
	
	Next, we will deal with the case where $T$ is the monster. First, we prove Proposition~\ref{Base2Lemma} for this group.
	
	\begin{proposition}\label{MonsterLemma}
		Let $T:=\mathbb{M}$ (the monster), and let $K<T$ be such that $T^{[T:K]}$ has base size at least $3$. Then  $K=2.\mathbb{B}$, a maximal subgroup of $T$, and the greatest common divisor of the point-stabiliser orders in $K^{[T:K]}$  is $209,018,880$. In particular, Proposition~$\ref{Base2Lemma}$  (expecially the final row of
		Table $\ref{tab:basesize}$) holds for   $\mathbb{M}$.
	\end{proposition}
	
	\begin{proof}
		Let $K\leq M<T$ with $M$ maximal in $T$. Then also $T^{[T:M]}$ has base size at least $3$. Hence
		by \cite[Table 2]{BOBW},  $M=2.\mathbb{B}$. Suppose that $K<M$, and let $H$ be a maximal subgroup of $M$ containing $K$. Then also $T^{[T:H]}$ has base size at least  $3$. For all maximal subgroups $H$ of $M$, we use the {\sf GAP} Character Table Library \cite{CharTable} (together with the available class fusions) and compute, using {\sf GAP} \cite{GAP}, that the value of the expression $\widehat{Q}(T,H,2)$ of  Proposition \ref{BOBWProp} is strictly less than $1$, and hence that $T^{[T:H]}$ has base size $2$. This is a contradiction, and hence $K=M=2.\mathbb{B}$. Finally, the point stabilisers in $K=2.\mathbb{B}$ of the nine  $K$-orbits in $[T:K]$ are given in  \cite[top of p. 101]{Ivanov} and we compute that the greatest common divisor of the point-stabiliser orders in $K^{[T:K]}$  is $209,018,880$, completing the proof.
	\end{proof}
	
	We are now ready to prove that the monster is totally $2$-closed.
	
	\begin{proposition}\label{MProp}
		The monster  $\mathbb{M}$ is totally $2$-closed.
	\end{proposition}
	
	\begin{proof}
		Let $T=\mathbb{M}$. As usual  our task, using Strategy~\ref{Strategy}, is to prove that
		$T^\Omega=T^{(2),\Omega}$, where $\Omega=[T:H]$, for each proper nontrivial subgroup $H$ of $T$. Since $T=\Aut(T)$  and $T$ has no $2$-transitive representations, it follows from
		Lemma \ref{BaseLemma}(b)(i) that this holds for all maximal subgroups $H$. Thus we proceed to stage (2) of Step 1 in Strategy~\ref{Strategy}, and here, by Proposition~\ref{MonsterLemma}, the only maximal subgroup $M$ such that $T^{[T:M]}$ has base size at least $3$ is $M=2.\mathbb{B}$. Hence the set $\mathcal{S}_2$ formed in Step 1 consists, up to conjugacy, of all proper nontrivial subgroups of $M$.
		
		In Step 2 of Strategy~\ref{Strategy}, we examine the maximal subgroups $H$ of $M$.
		We again use Notation~\ref{not2} with $G=T$, $\Omega=[T:H], X=T^{(2),\Omega}$,  $H=G_\omega$, with $\omega$  in a part $\Delta$ of a $T$-invariant partition $\Sigma$ of $\Omega$ such that $M=G_{\Delta}$, so $\Sigma=[T:M]$ and $L=G^\Sigma$ satisfies $L=L^{(2),\Sigma}$, and we set $Y=X_\Delta^\Delta$. Since $H$ is maximal in $M$ and $Z(M)=\Phi(M)\cong C_2$, we have
		$Z(M)=\core_M(H)=\Phi(M)$ and $M^{\Delta}\cong \mathbb{B}$. This implies that $M^{\Delta}=M^{(2),\Delta}$, by Lemma \ref{BaseLemma}(b)(i), since $\Aut(\mathbb{B})=\mathbb{B}$ and $\mathbb{B}$ has no $2$-transitive representations. It now follows from Lemma~\ref{lem:imprim}(b) that $Y=\mathbb{B}=M^\Delta$. Suppose now that $X\ne G$. Then by  Proposition \ref{RedProp}(a), $X=N\rtimes G$ with $N^\Delta$ a nontrivial normal subgroup of $Y$, and hence  $N^\Delta=Y$.  By the definition of the subgroup $N$ from Lemma \ref{Crucial}(b), this would imply that, for $1<j\leq s=|\Sigma|$, the group $(M\cap M_j)^{\Delta}$ is transitive.  Since $M\cap M_j$ is a proper subgroup of $M$ and $Z(M)=\core_M(H)\leq M\cap M_j$, the induced group $(M\cap M_j)^{\Delta}$ is a proper subgroup of $M^\Delta=Y$, so we have a proper factorisation $Y=Y_\omega (M\cap M_j)^{\Delta}$. This is a contradiction since $Y$ has no such factorizations, by \cite{LPS2}. Thus $X=G$, that is, $T^{[T:H]}=T^{(2),[T:H]}$. We now move to stage (2) of Step 2 for  Strategy~\ref{Strategy}. By Proposition~\ref{MonsterLemma}, each of the groups $T^{[T:H]}$ just examined has base size $2$. Hence the set $\mathcal{S}_3$ formed in Step 2 is empty, and we conclude that $T$ is totally $2$-closed.
	\end{proof}
	
	
	
	The next two groups we deal with are the O'Nan group and the Lyons group.
	
	\begin{proposition}\label{ONProp}
		The O'Nan group $\mathrm{ON}$ is not totally $2$-closed.
	\end{proposition}
	
	\begin{proof}
		Let $T=\mathrm{ON}$ and let $M$ be a maximal subgroup of $T$  with $M=L_3(7):2$. Also let $H=M'$ (the derived subgroup), so $|M:H|=2$.  By Corollary~\ref{BaseCor},  $T^{[T:M]}=T^{(2), [T:M]}$.
		We check, using GAP \cite{GAP}, that no two-point stabiliser  in $M^{[T:M]}$ is contained in $H$. Hence, $(M\cap M^t)H=M$ for all $t\in T$, that is to say, condition (ii) of Proposition \ref{MainCor1}(1) holds. Since $T^{[T:M]}=T^{(2), [T:M]}$, Proposition \ref{MainCor1}(1) applies, and we conclude that $T^{[T:H]}$ is not $2$-closed. Thus, $T$ is not totally $2$-closed.
	\end{proof}

	\begin{proposition}\label{LyProp} The Lyons group $\mathrm{Ly}$ is totally $2$-closed. Moreover, Pro\-position~$\ref{Base2Lemma}$, and in particular
		Table $\ref{tab:basesize}$ holds for   $T=\mathrm{Ly}$.
	\end{proposition}
	
	\begin{proof} Let $T=\mathrm{Ly}$. As in the other cases, our task is to prove, using Strategy~\ref{Strategy}, that $T^\Omega=T^{(2),\Omega}$, where $\Omega=[T:H]$, for each proper nontrivial subgroup $H$ of $T$. Since $T=\Aut(T)$ and $T$ has no $2$-transitive representations, it follows from
		Lemma \ref{BaseLemma}(b)(i) that this holds for all maximal subgroups $H$. Thus we proceed to stage (2) of Step 1 in Strategy~\ref{Strategy}: constructing the set $\mathcal{S}_2$. By \cite[Table 2]{BOBW}, there are only two conjugacy classes of maximal subgroups $M$ of $T$ such that $T^{[T:M]}$ has base size at least $3$, with representatives $M_1=G_2(5)$ and $M_2=3.\mathrm{McL}:2$. Therefore according to Strategy~\ref{Strategy}, the set $\mathcal{S}_2$ consists of all non-maximal subgroups $H$ such that the only maximal subgroups of $T$ containing $H$ are conjugate to $M_1$ or $M_2$.
		
		We now move on to Strategy \ref{Strategy}, Step 2, stage (1). Here, we need to prove that $T^{[T:H]}$ is $2$-closed for all $H\in\mathcal{S}_2$ with $\ell_T(H)=2$. Let $H\in\mathcal{S}_2$  with $\ell_T(H)=2$, and assume that $T^{[T:H]}\neq T^{(2),[T:H]}$. Let $M$ be a maximal subgroup of $T$ containing $H$, so $H$ is maximal in $M$. Without loss of generality we may assume that $M=M_i$ for $i=1$ or $2$. We adopt the rest of the notation in Notation \ref{not2}, with $T=G$. So $\Omega=[T:H]$, $\Sigma=[T:M]$, $L=T^{\Sigma}$, $M=T_{\Delta}$ (with $\Delta\in\Sigma$), $H=T_{\omega}$ (with $\omega\in\Delta$), etc. Note that $L^{(2),\Sigma}=L$, by the previous paragraph, and $M^\Delta$ is primitive since $H$ is maximal in $M$. We also set $X=T^{(2),\Omega}$, so $X>T$, and $Y=X_{\Delta}^{\Delta}\le M^{(2),\Delta}$.
		
		We begin with the case $M=M_1=G_2(5)$. We note that, by \cite[Lemma 2.19.1]{Ivanov95}, it follows that the greatest common divisor of the orders of the $2$-point stabilisers in $T^{[T:M]}$ is $48$, and hence the entry in Table~\ref{tab:basesize} for $M=G_2(5)$ is correct. Suppose first that $|M:H|\neq 5^3(5^3-1)/2$. Then $M\cong M^{\Delta}$, so \cite[Theorem 1]{LPS} implies that $Y\le M^{(2),\Delta}$ is almost simple with socle $G_2(5)$. Since $\Out(G_2(5))=1$ and $G_2(5)$ has no $2$-transitive representations, it follows from  Lemma \ref{BaseLemma}(b)(i) that $M^{\Delta}=M^{(2), \Delta}$. Therefore $Y\cong G_2(5)$ and so the nontrivial normal subgroup $A$ of $Y$ in  Proposition \ref{MainCor1}(2) is equal to $Y$. In particular, $A^{\Delta}$ is transitive. However this means,  by Proposition \ref{MainCor1}(2), that $|M:H|=|\Delta|$ divides the greatest common divisor of the orders of the $2$-point stabilisers in $T^{[T:M]}$, that is, $|M:H|$ divides $48$. This is a contradiction since the smallest degree of a transitive permutation representation of  $G_2(5)$ is $3906$, \cite[page 114]{Atlas}. Thus $|M:H|=5^3(5^3-1)/2$, and \cite[Theorem 1]{LPS} implies that $Y\le M^{(2),\Delta}\le \Aut(\Omega_7(5))=\Omega_7(5).2$. Since $M\le Y$ and the only maximal overgroups of $G_2(5)$ in $\Omega_7(5).2$ are $G_2(5)$, $\Omega_7(5)$, and $\Omega_7(5).2$, we have $Y\in\{G_2(5),\Omega_7(5),\Omega_7(5).2\}$. Hence  the nontrivial normal subgroup $A$ of $Y$ in  Proposition \ref{MainCor1}(2) satisfies $A\in\{G_2(5),\Omega_7(5),\Omega_7(5).2\}$. In all cases $A\geq G_2(5)$, so $A^\Delta$ is transitive since $M^\Delta$ is primitive. Then the same argument as in the previous shows that $|\Delta|$ divides $48$ and we get the same contradiction.
		Thus, we have shown that $T=T^{(2),[T:H]}$ for all maximal subgroups $H$ of $M=M_1$, completing Step 2 stage (1) for $M=M_1$.
		
		Now we consider the case $M=M_2=3.\mathrm{McL}:2$. We note that, by \cite[Lemma 2.19.1]{Ivanov95}, it follows that the greatest common divisor of the orders of the $2$-point stabilisers in $T^{[T:M]}$ is $30$, and hence the entry in Table~\ref{tab:basesize} for $M=3.\mathrm{McL}:2$ is correct. Let $Z=Z(M')$, where $M'=3.\mathrm{McL}$ is the derived subgroup of $M$, so $|Z|=3$. Since $Z\le\Phi(M)$ and $H$ is maximal in $M$, we have $Z\le \core_M(H)$. Suppose first that $H=M'$, so $|{\Delta}|=2$. We will use Proposition \ref{MainCor1}(1): by \cite[Lemma 2.19.1]{Ivanov95}, $M^{\Sigma}$ has an orbit such that the stabiliser $M\cap M_j$ in $M$ of a point in this orbit has shape $2.\Sym(7)$. A computation using MAGMA shows that $(M\cap M_j)H\neq M$. Thus the condition in Proposition \ref{MainCor1}(1)(ii) fails, and hence $X=T$, which is a contradiction. Therefore $H\neq M'$, and hence $\core_M(H)=Z\cong C_3$ and $M^{\Delta}\cong \mathrm{McL}:2$. Then $M^{\Delta}=M^{(2),\Delta}$, by \cite[Theorem 1]{LPS}, and since $M^{\Delta}\le Y\le M^{(2),\Delta}$, this means that $Y=M^{(2),\Delta}\cong \mathrm{McL}:2$. Thus, the nontrivial normal subgroup $A$ of $Y$ in Proposition \ref{MainCor1}(2) is either $\mathrm{McL}$ or $\mathrm{McL}:2$. In particular, $A^{\Delta}$ is transitive, and hence,  by Proposition \ref{MainCor1}(2), $|M:H|=|\Delta|$ divides the greatest common divisor of the orders of the $2$-point stabilisers in $T^{[T:M]}$, that is, $|M:H|$ divides $30$. This is a contradiction since the smallest degree of a transitive permutation representation of  $\mathrm{McL}:2$ is $275$, \cite[page 100]{Atlas}.
		
		This completes stage (1) of Step 2.
		We now proceed to stage (2) of Step 2: constructing the set $\mathcal{S}_3$.   To do this we prove the following claim.
		
		\medskip
		\noindent
		{\it Claim: If $H\leq T$ and either (i) $\ell_T(H)=2$ but $H \not\cong M_2'=3.\mathrm{McL}$, or
			(ii) $H<M_2'$ with $H$ conjugate to $3.\mathrm{M}_{22}$ or $3.2^4.A_7$, then $T^{[T:H]}$ has base size  $2$.}

		\smallskip\noindent
		{\it Proof of Claim:}
		Let $M$ be a maximal subgroup of $T$ containing $H$. If $T^{[T:M]}$ has base size $2$ then so does $T^{[T:H]}$, so we may assume that $T^{[T:M]}$ has base size greater than $2$. As we saw above, this means (replacing $H, M$ by some $T$-conjugate if necessary) that  $M=M_i$ for $i=1$ or $2$. Suppose first that $\ell_T(H)=2$, that is, $H$ is maximal in $M$, but $H\ne M_2'$. Then $H$ is either a maximal subgroup of $M_1=G_2(5)$, or $M=M_2$, $H$ contains $Z:=Z(M')\cong C_3$, and $H/Z$ is a maximal subgroup of $M/Z\cong \mathrm{McL}:2$ of index greater than $2$. Such subgroups $H$ can be read off using the Web Atlas \cite{WebAtlas}.
		As in Proposition \ref{MProp}, an upper bound for the base size of $T^{[T:H]}$ can be computed using Proposition \ref{BOBWProp}.
		Using the character table library (together with the available class fusions), we check using {\sf GAP} \cite{GAP} that $\hat{Q}(T,H,2)<1$ for each choice of $H$. Hence by Proposition \ref{BOBWProp}, $T^{[T:H]}$ has base size $2$. This proves part (i). Finally suppose that $H<M_2'$ with $H$ conjugate to $3.\mathrm{M}_{22}$ or $3.2^4.A_7$. The same method,  using {\sf GAP}, shows that $\hat{Q}(T,H,2)<1$, and hence that $T^{[T:H]}$ has base size $2$, proving part (ii).
		
		\medskip
		It follows from part (i) of the claim, and the definition of $\mathcal{S}_3$, that $\mathcal{S}_3$ consists of those subgroups conjugate to a proper nontrivial subgroup $H$ of $M_2'=3.\mathrm{McL}$ such that the only subgroup $K$ containing $H$ with $\ell_T(K)=2$ is $K=M_2'$. Examining the Atlas \cite[page 100]{Atlas}, we see that $H$ must be conjugate to a subgroup of $N_1:=3.\mathrm{M}_{22}$ or $N_2:=3.2^4:A_7$, since all other proper subgroups of $M_2'$ are contained in a subgroup $K$ of $M_2$ with $\ell_T(K)=2$ and $K\ne M_2'$.
		
		Now we move on to Strategy \ref{Strategy}, Step 3, stage (1). We need to prove that $T^{[T:H]}$ is $2$-closed for all $H\in\mathcal{S}_3$ with $\ell_T(H)=3$. By the previous paragraph each such subgroup $H$ is conjugate to $N_1$ or $N_2$ (both maximal subgroups of $3.\mathrm{McL}$). The argument showing that in these two cases $T^{[T:H]}$ is $2$-closed is almost identical to that given in Step 2, stage (1) for maximal subgroups of $M_2$ other than $M_2'$. We assume that $H=N_i$ and $T^{[T:H]}$ is not $2$-closed. We have $\Core_{M_2'}(H)=Z\cong C_3$ and (using Notation~\ref{not2}) $(M_2')^\Delta\cong \mathrm{McL}$ where $\Delta=[M_2':H]$. We again apply \cite[Theorem 1]{LPS} to deduce that $Y=(M_2')^{(2),\Delta}=\mathrm{McL}$. This implies that the subgroup $A$ in Proposition~\ref{MainCor1}(2) is transitive on $\Delta$ and hence that $|\Delta|=|M_2':H|$ divides the greatest common divisor of the orders of the $2$-point stabilisers in $T^{[T:M_2']}$. Hence $|M_2':H|$ divides $30$, which is a contradiction.  This completes  Step 3, stage (1).
		
		Finally to complete  Step 3, stage (2), we construct $\mathcal{S}_4$. A subgroup $H\in \mathcal{S}_4$ must satisfy $H\in \mathcal{S}_3$ and $\ell_T(H)\geq 4$, and further, if $\ell_T(H)=4$ and $H<K<T$, then in particular $T^{[T:K]}$ has base size at least $3$. Now $H\in \mathcal{S}_3$  and $\ell_T(H)\geq 4$ imply that $H$ is conjugate to a proper subgroup of $N_i$ for some $i$, and the second condition implies that  $T^{[T:N_i]}$ has base size at least $3$. This however contradicts part (ii) of the claim, and we conclude that  $\mathcal{S}_4$ is empty. We have therefore shown that $T$ is totally $2$-closed.

		Finally we prove Proposition \ref{Base2Lemma} for $T=\mathrm{Ly}$. The information deduced above implies that, if $H$ is a proper nontrivial subgroup of $T$ with $H\not\in\{G_2(5),3.\mathrm{McL},3.\mathrm{McL}:2\}$, then $T^{[T:H]}$ has base size $2$. We know from \cite[Table 2]{BOBW} that $T^{[T:H]}$ has base size  $3$ when $H\in\{G_2(5),3.\mathrm{McL}:2\}$, and in these cases the entries in Table~\ref{tab:basesize} is correct. Suppose then that $H=3.\mathrm{McL}<M_2=3.\mathrm{McL}:2$. Then, for each $g\in T$, $|M_2\cap M_2^g:H\cap H^g|$ divides $4$. However, by \cite[Lemma 2.19.1]{Ivanov95}, $|M_2\cap M_2^g|>4$ for all $g\in T$. Thus $H\cap H^g\ne 1$ for each $g\in T$, and hence $T^{[T:H]}$ has base size greater than $2$.
		Finally, the greatest common divisor of the orders of the $2$-point stabilisers in $T^{[T:H]}$ divides the analogous number for the action $T^{[T:M_2]}$ which, as we showed above, is $30$. Thus  the entry in Table~\ref{tab:basesize} is correct also for $3.\mathrm{McL}$.
	\end{proof}
	
	Now we deal with the Conway and Fischer groups $\mathrm{Co}_2$ and $\mathrm{Fi}_{23}$, and then the baby monster $\mathbb{B}$.
	
	\begin{proposition}\label{Co2Prop} Neither the second Conway group $\mathrm{Co}_2$ nor the Fischer group $\mathrm{Fi}_{23}$ is totally $2$-closed.
	\end{proposition}
	
	\begin{proof} Let $T\in\{\mathrm{Co}_2,\mathrm{Fi}_{23}\}$. Consider the following maximal subgroup $M$ of $T$, namely
		$M = U_6(2):2$ if $T=\mathrm{Co}_2$, and $M=O^+_8(3):\Sym(3)$ if $T=\mathrm{Fi}_{23}$. In each case, $T^{[T:M]}$ has rank $3$. Set $H:=M'$, the derived subgroup, so that $|M:H|=2$, and adopt the notation in Notation \ref{not2}, with $T=G$. We argue as in the case of $T=\mathrm{ON}$ in Proposition~\ref{ONProp}: we check, using {\sf GAP} \cite{GAP}, that each of the $2$-point stabilisers $M\cap M_j$ in the $T$-action on $[T:M]$ satisfies
		$(M\cap M_j)H=M$,  that is to say, condition (ii) of Proposition \ref{MainCor1}(1) holds. On the other hand $T^{[T:M]}=T^{(2), [T:M]}$ by Lemma \ref{BaseLemma}(b)(i) since $T=\Aut(T)$  and $T^{[T:M]}$ is not $2$-transitive. Hence Proposition \ref{MainCor1}(1) applies, and we conclude that $T^{[T:H]}$ is not $2$-closed.
		Hence, $T$ is not totally $2$-closed.
	\end{proof}
	
	\begin{proposition}\label{BProp}
		Fischer's baby monster $\mathbb{B}$ is not totally $2$-closed.
	\end{proposition}
	
	\begin{proof} Let $T=\mathbb{B}$, and consider the     maximal subgroup $M=2.{}^2E_6(2).2$.
		Then $M/Z(M)$ is an extension of ${}^2E_6(2)$ by a graph automorphism. Let $H=M'\cong 2.{}^2E_6(2)$, of index $2$ in $M$.  The stabilisers in $T^{[T:M]}$ of two distinct points are given in \cite[Lemma 2.25.1]{Ivanov95}, namely
		\[
		(2\times 2^{1+20}.U_6(2)).2,\quad 2^2\times F_4(2),\quad  \mathrm{Fi}_{22}:2,\quad 2^{1+20}.U_4(3).2^2.
		\]
		We call these groups $E_i=M\cap M_i$, for $i=2,3,4,5$, respectively. We claim that $HE_i=M$ for $i=2,\dots,5$. All assertions in the next two paragraphs are from the short proof of the main theorem in \cite{LeonSims}. The groups $E_i$ are the same as in \cite{LeonSims}, but there the group $M$ is denoted by $E$.
		
		First, $|M:E_2|$ is odd, so certainly $E_2H=M$. The group $E_3$ contains $Z(M)$, and $E_3/Z(M)$ is the centraliser in ${}^2E_6(2).2$ of a graph automorphism. In particular, $E_3$ contains a graph automorphism, so we also have $E_3 H=M$. Next, $M$ has precisely two conjugacy classes of subgroups isomorphic to $\mathrm{Fi}_{22}:2$, and they are interchanged by an automorphism of $M$ which acts trivially on $H$ (see \cite[top of page 1040]{LeonSims}). It follows, since $E_4$ is one of these groups, that $E_4\not\le H$, and hence $E_4H=M$.
		
		Finally, the group $E_5$ can be written as $E_5=KQ$ as follows (see \cite[page 1040]{LeonSims}). We may take $E_2=C_M(d_2)$ with $d_2\in E_4$ being a $2A$-involution of $E_4=\mathrm{Fi}_{22}:2$; and $K=C_{E_4}(d_2)\cap C_{E_4}(d_2')$ with $d_2'$ a $2B$-involution of $E_4$ which does not commute with $d_2$ (using ATLAS notation). Then $Q$ is a unique index $2$ subgroup of $O_2(E_2)=2\times 2^{1+20}$ which is normalised by $K$.  A computation in {\sf GAP} \cite{GAP} shows that 
		$KE_4'=E_4$. Hence, since $Q\leq M'=H,\ E_4'\le M'= H$ and $E_4H=M$, we deduce that
		\[
		M=E_4H = (KE_4')H= K (E_4'H) = KH = K(QH) = (KQ)H= E_5H.
		\]
		Thus $E_iH=M$ for all $i$, proving the claim. Hence the condition of Proposition \ref{MainCor1}(1)(ii) holds for the action $T^{[T:H]}$,  and it follows from Proposition \ref{MainCor1}(1) that $T^{[T:H]}$ is not $2$-closed. The result follows.
	\end{proof}
	
	We now move on to our second to last case: the Thompson group $\mathrm{Th}$. We begin by showing that a particular transitive action of $\mathrm{Th}$ is $2$-closed.
	
	\begin{proposition}\label{ThPrelim}
		Suppose that $H={}^3D_4(2)$ is a subgroup of $T=\mathrm{Th}$ with $\ell_T(H)=2$, with $H$ contained in a maximal subgroup $M:={}^3D_4(2).3$. Then $T^{[T:H]}$ is $2$-closed.
	\end{proposition}
	
	\begin{proof} Let $T$, $H$, and $M$ be as in the statement, and use the notation in Notation \ref{not2}, with $T=G$. So $\Omega=[T:H]$, $\Sigma=[T:M]$ with $s=|\Sigma|$, $L=T^{\Sigma}$, $M=T_{\Delta}$ (with $\Delta\in\Sigma$), $H=T_{\omega}$ (with $\omega\in\Delta$), etc. By Lemma~\ref{BaseLemma}(b)(i), $L^{(2),\Sigma}=L$, since $T=\Aut(T)$  and $T$ has no $2$-transitive representations. 
		We also set $X=T^{(2),\Omega}$ and $Y=X_{\Delta}^{\Delta}\le M^{(2),\Delta}$. Since $|\Delta|=3$, we have $R=M^\Delta=Y=M^{(2),\Delta}\cong C_3$.
		
		Assume that $T^{[T:H]}$ is not $2$-closed. Then, by Lemma~\ref{Crucial}, $X=N\rtimes T$ with $N$ a nontrivial subdirect subgroup of $B_Y=C_3^s$, and moreover, for each $j=2,\dots,s$, $(\rho_{1,j}(M\cap M_j))^{\Delta\times\Delta_j}$ is either equal to $Y\times Y$ or is a diagonal subgroup of  $Y\times Y$. By Proposition~\ref{ThPropCount}, this group is diagonal for strictly more than $s/2$ values of $j$, and hence by Lemma~\ref{PrimeCor}, $N$ is a diagonal subgroup of $B_Y$.  Thus $N\cong C_3$ and, since $X=N\rtimes T$ with $T$ a nonabelian simple group, the group $T$ centralises $N$.
		
		We now examine a part $\Delta_2\in\Sigma$ such that $M\cap M_2= 2^2.[2^9].(\Sym(3)\times 3)$. As in the proof of Proposition~\ref{ThPropCount}, $(\rho_{1,2}(M\cap M_2))^{\Delta\times\Delta_2}$ is a diagonal
		subgroup of  $Y\times Y$. Further, the subdegrees of $T^{[T:M]}$ are pairwise distinct, so each $T$-orbit in $\Sigma^{(2)}=\{(\alpha,\beta)\text{ : }\alpha,\beta\in\Sigma,\alpha\neq\beta\}$ is self-paired. It follows that the setwise stabiliser  $T_{\{\Delta\cup \Delta_2\}}$ in $T$ of $\Delta\cup\Delta_2$ contains an element, say $g$, which interchanges $\Delta$ and $\Delta_2$. Thus the permutation group
		induced by $T_{\{\Delta\cup \Delta_2\}}$ on  $\Delta\cup\Delta_2$ is regular of order $6$, so is either $C_6$ or $\Sym(3)$. Moreover, since $g\in T_{\{\Delta\cup \Delta_2\}}$ and $T$  centralises $N$, it follows from  Lemma~\ref{Crucial}(b)  that  $g$ centralises $\rho_{1,2}(M\cap M_2))^{\Delta\times\Delta_2}$, and hence  the permutation group induced by $T_{\{\Delta\cup \Delta_2\}}$ on  $\Delta\cup\Delta_2$ is $C_6$, and
		\begin{align}\label{th1}
			\text{$T_{\{\Delta\cup \Delta_2\}}/(H\cap H^g)\cong C_6$.}
		\end{align}
		
		Let $P :=O_2(M\cap M_2) = 2^2.[2^9]$, and $Q:= N_T(P)$. Since $|M\cap M_2:H\cap H^g|=3$ it follows that $P < H\cap H^g < H < Q$. Also, since $P$ is characteristic in $M\cap M_2$ and $M\cap M_2\unlhd T_{\{\Delta\cup \Delta_2\}}$, we have $P<T_{\{\Delta\cup \Delta_2\}}<Q$. By \cite[Section 2.2]{Linton}, $Q$ has shape $2^5.(2^6:\Sym(3)\times L_3(2))$ and $Q$ is contained in a maximal subgroup $K:=2^5.\mathrm{L}_5(2)$  of $T$.
		More precisely, $Q$ contains $O_2(K)\cong 2^5$, and $Q/O_2(K)$ is a maximal parabolic subgroup of $\mathrm{L}_5(2)$ of shape $2^6:\Sym(3)\times L_3(2)$. Further, $O_2(Q)=O_2(M\cap M_2)=P$ and  $(M\cap M_2)/P$ has shape $\Sym(3)\times 3$. Hence, since $M\cap M_2$ has index $2$ in $T_{\{\Delta\cup \Delta_2\}}$, it follows that $T_{\{\Delta\cup \Delta_2\}}/P$ is isomorphic to a subgroup of $\Sym(3)\times L_3(2)$ of shape $(\Sym(3)\times 3).2$.
		The only such subgroups of $\Sym(3)\times L_3(2)$ are the Sylow $3$-normalisers, and they are isomorphic to $\Sym(3)\times \Sym(3)$. However, since $H\cap H^g$ contains $P$, $T_{\{\Delta\cup \Delta_2\}}/(H\cap H^{g})$ is isomorphic to a quotient of $T_{\{\Delta\cup \Delta_2\}}/P\cong \Sym(3)\times \Sym(3)$. This contradicts (\ref{th1}), since $\Sym(3)\times \Sym(3)$ has no cyclic quotients of order $6$.
		Thus $T^{[T:H]}$ is $2$-closed.
	\end{proof}
	
	\begin{proposition}\label{ThProp}
		The Thompson group $\mathrm{Th}$ is totally $2$-closed. Furthermore, Proposition $\ref{Base2Lemma}$, and in particular Table $\ref{tab:basesize}$, holds for $T=\mathrm{Th}$.
	\end{proposition}
	
	\begin{proof}
		Let $T=\mathrm{Th}$. Our task is to prove, using Strategy~\ref{Strategy}, that $T^\Omega=T^{(2),\Omega}$, where $\Omega=[T:H]$, for each proper nontrivial subgroup $H$ of $T$.
		Since  $T=\Aut(T)$  and $T$ has no $2$-transitive representations, it follows from
		Lemma \ref{BaseLemma}(b)(i) that this holds for all maximal subgroups $H$. Thus we proceed to stage (2) of Step 1 in Strategy~\ref{Strategy}: constructing the set $\mathcal{S}_2$.
		By \cite[Table 2]{BOBW}, there are only two conjugacy classes of maximal subgroups $M$ of $T$ such that $T^{[T:M]}$ has base size at least $3$, with representatives $M_1={}^3D_4(2):3$ and $M_2=2^5.\mathrm{L}_5(2)$. Therefore according to Strategy~\ref{Strategy}, the set $\mathcal{S}_2$ consists of all non-maximal subgroups $H$ such that the only maximal subgroups of $T$ containing $H$ are conjugate to $M_1$ or $M_2$.
		
		We now move on to Strategy \ref{Strategy}, Step 2, stage (1). We need to prove that $T^{[T:H]}$ is $2$-closed for all $H\in\mathcal{S}_2$ with $l_T(H)=2$. Let $H\in\mathcal{S}_2$, and assume that $T^{[T:H]}\neq T^{(2),[T:H]}$. Let $M$ be a maximal subgroup of $T$ containing $H$, so without loss of generality we may assume that $M\in\{M_1, M_2\}$. We use the rest of the notation in Notation \ref{not2}, with $T=G$. So $\Omega=[T:H]$, $\Sigma=[T:M]$, $L=T^{\Sigma}$, $M=T_{\Delta}$ (with $\Delta\in\Sigma$), $H=T_{\omega}$ (with $\omega\in\Delta$), etc. Note that $L^{(2),\Sigma}=L$, by Step 1. We also set $X=T^{(2),\Omega}$, so $X>T$, and $Y=X_{\Delta}^{\Delta}\le M^{(2),\Delta}$. Suppose first that $M=M_2$.  We note that, by
		\cite[Lemma 2.20.1(ii)]{Ivanov95}, it follows that the greatest common divisor of the orders of the
		two-point stabilisers in $T^{[T:M]}$ is $1$, and hence the entry in Table~\ref{tab:basesize} for $M=
		2^5.\mathrm{L}_5(2)$ is correct. This implies, by Proposition \ref{MainCor1}(2), that $X=T$, which is a contradiction. We note here that this argument is valid for all proper subgroups $H$ of $M_2$ (without the condition $\ell_T(H)=2$).
		
		Thus $M=M_1$. At this point, we note that it follows from
		\cite[Lemma 2.20.1(i)]{Ivanov95} that the greatest common divisor of the orders of the
		$2$-point stabilisers in $T^{[T:M]}$ is $6$, and hence the entry in Table~\ref{tab:basesize} for $M=
		{}^3D_4(2):3$ is correct. Further this also implies correctness of the entry in that table for the derived subgroup $M'={}^3D_4(2)$ (we show below that $T^{[T:M']}$ has base size at least $3$).
		If $H=M'$ then $T^{[T:H]}$ is $2$-closed by Proposition \ref{ThPrelim}. Thus we may assume that $H\neq M'$, and we assume also that $\ell_T(H)=2$. Then, by \cite[Theorem 1]{LPS}, $M^{(2),\Delta}$, and hence also its subgroup $Y$, is almost simple with socle ${}^3D_4(2)$. Since $X>T$, the nontrivial normal subgroup $A$ of $Y$ in Proposition \ref{MainCor1}(2) must contain ${}^3D_4(2)$. In particular, $A^{\Delta}$ is transitive, and so by Proposition \ref{MainCor1}(2), $|M:H|=|\Delta|$ divides  the greatest common divisor of the orders of the $2$-point stabilisers in $T^{[T:M]}$, which as noted above is $6$. This is a contradiction, and hence $T^{[T:H]}$ is $2$-closed, completing Step 2, stage (1).
		
		
		We now proceed to stage (2) of Step 2: constructing the set $\mathcal{S}_3$. Given the results of Step 1 above, the set $\mathcal{S}_3$ will consist of all conjugates of proper nontrivial subgroups $H$ of $M=M_1$ or $M=M_2$ such that $\ell_T(H)\geq 3$, and if $\ell_T(H)= 3$ then, for any maximal subgroup  $J$ of $M$ containing $H$,  $T^{[T:J]}$ has base size at least $3$.
		If $M=M_1$, then one such subgroup $J$ is $M'={}^3D_4(2)$.
		However if $H<M'$ is maximal then, by the Atlas \cite[p.89]{Atlas}, $H< N_M(H)<M$, and $N_M(H)$ is another maximal subgroup of $M$ containing $H$. Thus each subgroup $H<M$ with $\ell_T(H)=3$ is contained in a maximal subgroup $J$ of $M$ with $J$ lying in the union $\mathcal{C}$ of the nine conjugacy classes of maximal subgroups of $M$ of index greater than $3$ (that is, maximal subgroups not equal to $M'$). Thus if we can prove that $T^{[T:J]}$ has base size $2$ for all $J\in\mathcal{C}$, and also for all maximal subgroups $J$ of $M_2=2^5.\mathrm{L}_5(2)$, then the set $\mathcal{S}_3$ would be empty, and we will have proven that $T$ is totally $2$-closed.
		
		So let $J$ be either a maximal subgroup of $M_2=2^5.\mathrm{L}_5(2)$, or a member of $\mathcal{C}$. We use Proposition \ref{BOBWProp} to show that $T^{[T:J]}$ has base size $2$. So let $x_1,\hdots,x_m$ be representatives for the $T$-conjugacy classes of elements of prime order. First we note the following crude upper bound:
		\begin{align}\label{Crude0}
			\widehat{Q}(T,J,2)\le \sum_{p \text{ prime}} \frac{|\{g\in J\text{ : }|g|=p\}|^2}{\min_i\{|x_i^T|\text{ : }|x_i|=p\}}.
		\end{align}
		We use \texttt{Magma}~\cite{MAGMA} to compute the right hand side of (\ref{Crude0}), and find that it is less than $1$ for each subgroup $J$, and hence, by Proposition \ref{BOBWProp}, $T^{[T:J]}$ has base size $2$ for each $J$. Thus, $T$ is totally $2$-closed. (For $J$ a maximal subgroup of $M_2={}^3D_4(2):3$, we work with the degree $819$ permutation representation of ${}^3D_4(2)$ from the Web Atlas to construct ${}^3D_4(2)$ in \texttt{Magma}~\cite{MAGMA}, and then we take its normaliser in $\Sym(819)$ to get ${}^3D_4(2):3$. For $J$ contained in  the \emph{Dempwolff group} $M_2=2^5.\mathrm{L}_5(2)$, we work with the degree $7440$ permutation representation of $M_2$ given in the Web Atlas.)
		
		Finally we complete the proof of Proposition~\ref{Base2Lemma} for the group $\mathrm{Th}$. Suppose that $H$ is a proper subgroup of $T$ for which  $T^{[T:H]}$ has base size at least $3$. Let $M$ be a maximal subgroup of $T$ containing $H$, so also  $T^{[T:M]}$ has base size at least $3$. Then, as above, $M$ is conjugate to $M_1$ or $M_2$, so  we may assume that $M\in\{M_1, M_2\}$ and $H\ne M$. Let $J$ be a maximal subgroup of $M$ containing $H$. If either $M=M_2$ or $J\in\mathcal{C}$, then we showed above that  $T^{[T:J]}$ has base size $2$, and hence also $T^{[T:H]}$ has base size $2$, which is a contradiction. Thus $M=M_1$ and $J=M'$ is the only maximal subgroup of $M$ containing $H$. As discussed above, by  \cite[p.89]{Atlas}, the only possibility is $H=J=M'$.
		Now $|M:H|=3$ and hence, for each $g\in T$, the index $|M\cap M^g:H\cap H^g|$ divides $9$.
		Moreover, it follows from \cite[Lemma 2.20.1(i)]{Ivanov95} that, for each $g\in T$,
		$\gcd(|M\cap M^g|,9)$ is strictly less than $|M\cap M^g|$, and hence $H\cap H^g\ne 1$. Thus, $T^{[T:H]}$ has base size at least $3$. This proves Proposition~\ref{Base2Lemma}(a), and the assertions of Proposition~\ref{Base2Lemma}(b) have been verified already in the course of the proof.
	\end{proof}
	
	\section{The case $T=\mathrm{J}_4$}\label{J4Section}
	
	In this section, we prove that the fourth Janko group $\mathrm{J}_4$ is totally $2$-closed, and in so-doing we prove Proposition~\ref{Base2Lemma} for $\mathrm{J}_4$. This will complete the proofs of Theorem \ref{thm:J} and Proposition~\ref{Base2Lemma}. Our main reference is Janko's original paper \cite{Janko4}, and  wherever possible we try to be consistent with his notation, or draw attention to it. See Remark \ref{rem:JankoRemark} for more information.
	
	Throughout the section we set $T:=\mathrm{J}_4$, and we adopt Strategy~\ref{Strategy}.
	To complete Step 1 of Strategy~\ref{Strategy}, we need to know all maximal subgroups $M$ of $T$ such that $T^{[T:M]}$ has base size at least $3$.  By \cite[Table 2]{BOBW}, each such subgroup $M$ is conjugate to one of the following three maximal subgroups $D_1,D_2,D_3$, where  $z$ is a representative of the $2A$ conjugacy class of involutions in $T$, and in fact $T^{[T:M]}$ has base size $3$ in each case.
	
	\begin{itemize}
		\item $D_1:=N_T(V)=V\rtimes F_1 \cong 2^{11}:\mathrm{M}_{24}$, where $V=O_2(D_1)$ lies in the unique conjugacy class of elementary abelian subgroups of $T$ of order $2^{11}$ and $F_1\cong \mathrm{M}_{24}$, \cite[Propositions 7 and 18]{Janko4}.
		
		\item $D_2:=C_T(z)\cong 2^{1+12}.3.\mathrm{M}_{22}:2$. The subgroup $O_2(D_2)$ is not complemented in $D_2$, but has a supplement $F_2\cong 6.\mathrm{M}_{22}:2$ such that $F_2\cap O_2(D_2)=\langle z\rangle\cong C_2$,  \cite[Proposition 1]{Janko4}.
		
		\item $D_3:=N_T(A)=A\rtimes F_3\cong  2^{10}:\mathrm{L}_5(2)$, where $A=O_2(D_3)$ is an elementary abelian self-centralizing subgroup of order $2^{10}$, contained in $D_2$, and $F_3\cong \mathrm{L}_5(2)$, \cite[Propositions 20 and 21]{Janko4}.
	\end{itemize}
	
	Maximality of the subgroups $D_1$, $D_2$, and $D_3$ can also be checked from the Web Atlas \cite{WebAtlas}. For Strategy~\ref{Strategy} we also need to examine transitive actions $T^{[T:H]}$ with base size at least 3 for non-maximal subgroups $H$. For such a subgroup, if $H<M<T$ with $M$ maximal, then $T^{[T:M]}$ also has base size at least 3, and hence $M$ is conjugate to $D_i$ for
	some $i$.  In similar situations in Section \ref{SporadicSection}, we determined whether or not such an action  $T^{[T:H]}$ has base size $2$ by finding crude upper bounds on $|x^T\cap H|$ for each element $x\in H$ of prime order, and then applying Proposition \ref{BOBWProp}. For $T=\mathrm{J}_4$, we need to be much more careful when finding upper bounds for $|x^T\cap H|$ in the case where $x$ is an involution. Thus the  first lemma in this section counts involutions in subgroups of the maximal subgroups $D_1$, $D_2$, and $D_3$. Since we make frequent reference to Janko's paper \cite{Janko4}, we summarise  Janko's notation in Remark~\ref{rem:JankoRemark}.

	\begin{remark}\label{rem:JankoRemark}{\rm
			The Janko group $T=\mathrm{J}_4$ has precisely two classes of involutions, denoted class $2A$ and $2B$, and Janko uses $z$ and $f$, respectively, to denote representatives of these classes, \cite[Theorem A(5)]{Janko4}. He also uses the following notation for the subgroups $D_1$, $D_2$, and $D_3$:
			
			\smallskip\noindent
			(1) Janko writes $M=D_1$, and $K=F_1$, so $M=V\rtimes K\cong 2^{11}\rtimes \mathrm{M}_{24}$ with $V\cap K=1$, \cite[Proposition 18]{Janko4}.
			
			\smallskip\noindent
			(2) Janko writes $H=D_2=C_T(z)$, $E=O_2(D_2)\cong 2^{1+12}$, and $H_0=[D_2,D_2]\cong 2^{1+12}.3.\mathrm{M}_{22}:2$ for the (index $2$) derived subgroup of $H$, \cite[Section 3]{Janko4}.   In \cite[last 13 lines of Proposition 1]{Janko4}, Janko describes the $D_2$-conjugacy classes of involutions in $H_0$:  there are two classes in $E\setminus\langle z\rangle$, with representatives $e$ and $f$, and four classes in $H_0\setminus E$, with representatives $t'$, $t'z$, $t_1$, and $t_2$. Moreover all of $t'$, $t'z$, $t_1$,  $t_2$ lie in the same coset of $E$, \cite[Page 570, line-11]{Janko4}.
			As noted above, $f\in 2B$; by \cite [Proposition 9]{Janko4}, the elements $e, t', t_1$ all lie in class $2A$, and $t'z\in 2B$; while by \cite [Proposition 15]{Janko4},  $t_2\in 2B$.
			Finally, by \cite[Propositions 6 and 10]{Janko4}, there are three $D_2$-conjugacy classes of involutions in $H\setminus H_0$, with representatives $\widetilde{t_1}$, $\widetilde{t_1}z$, and $\widetilde{t_2}$, and by \cite [Proposition 15]{Janko4}, $\widetilde{t_1}\in 2A$, and $\widetilde{t_1}z$, $\widetilde{t_2}\in 2B$.
			
			\smallskip\noindent
			(3) Janko writes $D_3=N_T(A)=A\rtimes B$, where $A=2^{10}$ is self-centralizing with $A<D_2$ and $B=F_3\cong \mathrm{L}_5(2)$, \cite[Propositions 20 and 21]{Janko4}. He notes that $O_2(D_2)\cong 2^{1+12}$ normalises $A=O_2(D_3)$, and $N_{D_2}(A)/A$ is isomorphic to the parabolic subgroup $2^6:(\Sym(3)\times L_3(2))$ of $D_3/A\cong B\cong \mathrm{L}_5(2)$, \cite[Page 587, lines 4--7]{Janko4}.
		}
	\end{remark}
	
	To help us find upper bounds for $|x^T\cap H|$ in the case where $x$ is an involution in $H<T$, we  introduce the following (non-standard) notation for counting involutions in finite groups.
	
	\begin{notation}\label{not5.2} For a subset $S$ of a finite group $G$, define
		$$
		\Inv(S\text{ : }\mathcal{P}_1,\hdots,\mathcal{P}_r)
		$$
		to be the number of involutions in $S$ with properties $\mathcal{P}_1,\hdots,\mathcal{P}_r$. For example, if $H$ is a subgroup of $D_j$, for some $j$, then
		$$
		\Inv({H}\backslash [{D_j},{D_j}]\text{ : }2\text{-central in }{D_j})
		$$
		is the number of involutions in ${H}\backslash [{D_j},{D_j}]$ which are centralised by a Sylow $2$-subgroup of ${D_j}$. We also define
		\begin{center}
			$\Inv(S)$ to be the total number of involutions in $S$,
		\end{center}
		and for $S\subseteq \mathrm{J}_4$, we denote by $\Inv(S\text{ : }2A)$ and $\Inv(S\text{ : }2B)$ the number of involutions in $S$ in the conjugacy class $2A$ and $2B$ of $\mathrm{J}_4$, respectively, so that
		$$
		\Inv(S)=\Inv(S\text{ : }2A)+\Inv(S\text{ : }2B).
		$$
	\end{notation}
	
	
	\begin{remark}\label{rem:J4Magma}
		{\rm
			In some of our proofs  we use \texttt{Magma}~\cite{MAGMA} to compute certain quantities in $T=\mathrm{J}_4$. More precisely, we use the $112$-dimensional representation for $T$ over $\mathbb{F}_2$. The Magma code for this representation is available from the Web Atlas \cite{WebAtlas}, and it conveniently includes constructions of the subgroups $D_1$, $D_2$, $D_3$, and $F_3$ defined above.
		}
	\end{remark}

	\begin{lemma}\label{J4Lemma} For $j=1,2,3$, let $D_j$ be as described above, let $E_j=O_2(D_j)$ and consider a subgroup  $H$ satisfying $E_j\leq H\leq D_j$. For each such subgroup let $\overline{H}= H/E_j$.
		Then the following assertions hold.\begin{enumerate}[\upshape(1)]
			\item If $j=1$, then
			\begin{align*}
				\Inv(H\text{ : }2A) =&7.11.23+16\,\Inv(\ol{H}\text{ : }2\text{-central in }\mathrm{M}_{24})+\\
				& 32\,\Inv(\ol{H}\text{ : non }2\text{-central in }\mathrm{M}_{24})\text{, and } \\
				\Inv(H\text{ : }2B) =& 4.3.23+7.16\,\Inv(\ol{H}\text{ : }2\text{-central in }\mathrm{M}_{24})+\\
				& 32\,\Inv(\ol{H}\text{ : non }2\text{-central in }\mathrm{M}_{24}).
			\end{align*}
			
			\item For $j=2$, and for $k\geq 1$, let $\mathcal{P}_k$ be the condition on an involution $\ol{h}\in \ol{H}$ that $|C_{E_2}(h)|=2^k$ for some preimage $h\in D_2$ of $\ol{h}$, and set $H_0:= [D_2,D_2]$. Then
			\begin{align*}
				\Inv(H\text{ : }2A) =&1387+112\,\Inv(\ol{H}_0)+64\,\Inv(\ol{H}\setminus \ol{H}_0\text{ : }\mathcal{P}_7)\text{, and }\\
				\Inv(H\text{ : }2B) =&2772+208\,\Inv(\ol{H}_0)+128\,\Inv(\ol{H}\setminus \ol{H}_0\text{ : }\mathcal{P}_6)+\\
				& 64\,\Inv(\ol{H}\setminus \ol{H}_0\text{ : }\mathcal{P}_7).
			\end{align*}
			
			\item If $j=3$, then
			\begin{align*}
				\Inv(H\text{ : }2A) =&5.31+64\,\Inv(\ol{H}\text{ : }2\text{-central in }\mathrm{L}_5(2))+\\
				& 16\,\Inv(\ol{H}\text{ : non }2\text{-central in }\mathrm{L}_5(2))\text{, and }\\
				\Inv(H\text{ : }2B) =& 4.7.31+64\,\Inv(\ol{H}\text{ : }2\text{-central in }\mathrm{L}_5(2))+\\
				& 48\,\Inv(\ol{H}\text{ : non }2\text{-central in }\mathrm{L}_5(2)).
			\end{align*}
		\end{enumerate}
	\end{lemma}
	
	\begin{proof} For $i\in\{1,2,3\}$, and $X\in\{A,B\}$, let $[H:E_i]$ denote a set of representatives for the $E_i$-cosets in $H$. Then we have
		\begin{align}\label{Count}
			\Inv(H\text{ : }2X)=\sum_{h\in [H:E_i]}\Inv(E_ih\text{ : }2X).
		\end{align}
		If $h_1$ and $h_2$ are conjugate in $D_i$, then  $\Inv(E_ih_1\text{ : }2X)$ and $\Inv(E_ih_2\text{ : }2X)$ are equal. If  $h\in D_i\setminus E_i$ and $E_ih$ contains an involution, then we may assume that
		$h$ is an involution, and also in this case $\ol{h}$ is an involution in $\ol{D_i}$. On the other hand, if $i=1$ or $3$, then $D_i=E_i\rtimes F_i$, and in these cases if $\ol{h}$ is an involution then  we may assume that $h\in F_i$ and hence $h$ is also an involution.
		If $i=2$, then $D_i=E_i F_i$ with $E_i\cap F_i=Z(E_i)=\langle z\rangle$ of order 2. In this case, if $E_2h$ is such that $\ol{h}$ is an involution, then we have checked using Magma (see Remark~\ref{rem:J4Magma}) that there is an involution $x\in D_2$ such that $E_2h=E_2x$. Thus, for all $i$, a nontrivial coset $E_ih$ contains an involution if and only if $\ol{h}$ is an involution.
		We use these facts frequently throughout the proof.
		We now prove parts (1), (2), and (3) separately.
		
		\medskip\noindent
		(1) Here, we are counting involutions in $D_1=2^{11}:\mathrm{M}_{24}$ with $E_1\cong 2^{11}$.
		By (\ref{Count}), we just need to compute $\Inv(E_1h)$ for each $h\in [H:E_2]$.
		We first deal with the case $h=1$. Here $\Inv(E_1\text{ : }2A)=7.11.23$ and $\Inv(E_1\text{ : }2B)=4.3.23$, by \cite[Proposition 9]{Janko4} (see also Remark \ref{rem:JankoRemark}). Suppose now that $1\neq h\in [H:E_1]$. As mentioned above, we may assume that $h$ is an involution and also that $\ol{h}$ is an involution. Since $E_1$ is an elementary abelian $2$-group, the involutions in $E_1h$ are precisely the elements of $C_{E_1}(\ol{h})h$. Thus, $\Inv(E_1h)=|C_{E_1}(\ol{h})|$. We now use \cite[Proposition 18]{Janko4}: if $\ol{h}$ is $2$-central in $\ol{D_1}\cong \mathrm{M}_{24}$, then $|C_{E_1}(\ol{h})|=2^7$, and $2^4$ of the $2^7$ involutions in $E_1h$ are of type $2A$, with the other $7.2^4$ of type $2B$. If $\ol{h}$ is not $2$-central in $\mathrm{M}_{24}$, then $|C_{E_1}(\ol{h})|=2^6$, and $E_1h$ contains $2^5$ involutions of type $2A$ and $2^5$ of type $2B$. The stated formulae in (1) then follow from (\ref{Count}).
		
		\medskip\noindent
		(2) Here, we are counting involutions in $D_2=2^{1+12}.3.\mathrm{M}_{22}:2$ with $E_2\cong 2^{1+12}$. As mentioned above, for each coset $E_2h$ in $H$ we may assume that $h\in F_2=6.\mathrm{M}_{22}:2$, where $F_2\cap E_2 = \langle z\rangle$.
		By Remark \ref{rem:JankoRemark}(2), there are two $D_2$-conjugacy classes of involutions in $E_2\setminus\langle z\rangle$, with representatives $e$ and $f$, in $2A$ and $2B$ respectively, and from the information in  \cite[Proposition 1 (from line --13)]{Janko4} about $|C_{D_2}(e)|, |C_{D_2}(f)|$, we compute their class sizes as $|e^{D_2}|=1386$ and $|f^{D_2}|=2772$. Thus we have the following, which will also be important for part (3):
		\begin{align}\label{lab:E2}
			&\text{The group $E_2$ contains $4159$ involutions, with $1+1386=1387$ of type $2A$} \nonumber \\
			&\text{ and $2772$ of type $2B$.}
		\end{align}
		Next, we count the involutions in $H_0\setminus E_2$, where $H_0=[D_2, D_2]$. By Remark \ref{rem:JankoRemark}(2), there are four $D_2$-conjugacy classes of involutions in $H_0\setminus E_2$ with representatives $t', t_1$ (in class $2A$) and $t'z, t_2$ (in class $2B$), and all four of these elements lie in the same coset $E_2t'$. Thus, by our remarks above, the only nontrivial $E_2$-cosets containing involutions are $E_2h$ with $h$ conjugate in $D_2$ to $t'$. By \cite[Page 569, line -3]{Janko4}, $E_2t'$ contains $16$ involutions conjugate to $t'$ and $16$ conjugate to $t'z$, and by \cite[Page 570, line -11, -10]{Janko4}, $E_2t'$ contains $2^5.3$ involutions conjugate to $t_1$ and $2^6.3$ conjugate to $t_2$. Hence, for each involution  $\ol{h}$ in $\ol{H}_0$,
		\begin{align}\label{t'N}
			&\text{$E_2h$ contains $2^6.5$ involutions, with $2^4+2^5.3=112$ of type $2A$, and} \nonumber\\
			&\text{$2^4+2^6.3=208$ of type $2B$.}
		\end{align}
		Finally, we count the involutions in $H\setminus H_0$. Again, as above, the only $E_2$-cosets in $H\setminus H_0$ containing involutions are those of the form $E_2h$ for an involution $h\in H\setminus H_0$.
		By Remark \ref{rem:JankoRemark}(2), there are three $D_2$-conjugacy classes of involutions in $H\setminus H_0$ with representatives $\widetilde{t_1}$ (in class $2A$) and $\widetilde{t}_1z, \widetilde{t}_2$ (in class $2B$). By \cite[Proposition 6]{Janko4}, $\widetilde{t}_1$ and $\widetilde{t}_1z$ have property $\mathcal{P}_7$ while $\widetilde{t}_2$ has property $\mathcal{P}_6$.
		Let $s\in\{\widetilde{t_1}, \widetilde{t_1}z, \widetilde{t_2}\}$, so $s$ has property $\mathcal{P}_k$ with $k=6$ or $7$. An element $xs\in E_2s$ is an involution if and only if $x^s=x^{-1}$ and this, in turn, is equivalent to either (i) $x\in C_{E_2}(s)$ with $x^2=1$, or (ii) $|x|=4$ and $s$ inverts $x$. Now,  for each $x\in E_2$ of order $4$, we have $x^{-1}=zx$. We claim that, for $s=\wt{t_1}$, there are  no possibilities in case (ii). Assume to the contrary that $x$ is such an element. Then
		\[
		\wt{t_1}^x=x^{-1}\wt{t_1}x=zx\wt{t_1}x=z\wt{t_1} x^{\wt{t_1}}x=z\wt{t_1},
		\]
		contradicting the fact that $\wt{t_1}$ and $\wt{t_1}z=z\wt{t_1}$ are not conjugate. Thus, for $s=\wt{t_1}$, the involutions in $E_2\widetilde{t_1}$ are precisely the elements $x\widetilde{t_1}$, where $x\in C_{E_2}(\wt{t_1})$ and $x^2=1$. Since $C_{E_2}(\wt{t_1})$ is elementary abelian of order $2^7$ by \cite[Proposition 6]{Janko4}, we deduce that $E_2\wt{t_1}$ contains $2^7$ involutions. Since $D_2=C_T(z)$ and $C_{E_2}(\wt{t_1})$ contains both $\wt{t_1}$ and $\wt{t_1}z$, exactly half (that is, $2^6$) of these involutions are $D_2$-conjugate to $\wt{t_1}$, and the other $2^6$ involutions are conjugate to $\wt{t_1}z$.
		
		Finally consider $s=\wt{t_2}$. The element $\wt{t_2}$ inverts $2^6$ elements of $E_2$ of order $4$ (this can be checked using \texttt{Magma}~\cite{MAGMA}, see Remark \ref{rem:J4Magma} above). In particular, there exists $x\in E_2$ of order $4$ such that $x^{\wt{t_2}}=xz$. Also, $C_{E_2}(\wt{t_2})$ is elementary abelian of order $2^6$ by \cite[Propositions 6 and 10]{Janko4}. Thus, there are $2^6+2^6=2^7$ involutions in $E_2\widetilde{t_2}$, consisting of the elements of the form $x\widetilde{t_2}$, where either $x\in C_{E_2}(\wt{t_2})$, or $|x|=4$ and $x^{\wt{t_2}}=xz$. If $x\in C_{E_2}(\wt{t_2})$, then $C_{E_2}(x\wt{t_2})=C_{E_2}(\wt{t_2})\cong 2^6$, so $x\wt{t_2}$ cannot be $D_2$-conjugate to $\wt{t_1}$ or $\wt{t_1}z$ (since these elements have centralisers of order $2^7$ in $E$). If $|x|=4$ and $x^{\wt{t_2}}=x^{-1}$, then $x\wt{t_2}$ inverts the element $x$ of order $4$, so again $x\wt{t_2}$ cannot be $D_2$-conjugate to $\wt{t_1}$ or $\wt{t_1}z$. Hence all of the $2^7$ involutions in $E_2\wt{t_2}$ are $D_2$-conjugate to $\wt{t_2}$.
		
		It follows from the last two paragraphs above that if $h\in D_2\setminus H_0$ and $h$ is an involution, then
		\begin{itemize}
			\item $\Inv(Eh\text{ : }2A)=\Inv(Eh\text{ : }2B)=2^6$ if $|C_{E_2}(h)|=2^7$; and
			\item $\Inv(Eh\text{ : }2A)=0$ and $\Inv(Eh\text{ : }2B)=2^7$ if $|C_{E_2}(h)|=2^6$.
		\end{itemize}
		Putting all these counts together and applying (\ref{Count}) yields part (2).
		
		\medskip\noindent
		(3) Here, we are counting involutions in $D_3=2^{10}:\mathrm{L}_5(2)$ (so $E_3\cong 2^{10}$). Recall that $D_3=E_3\rtimes F_3$ with $F_3\cong \mathrm{L}_5(2)$. Thus by our discussion above, the $E_3$-cosets in $D_3$
		containing involutions are precisely the cosets $E_3h$ with $h\in F_3$ such that $h^2=1$.
		We first consider the case $h=1$. By \cite[Proposition 20]{Janko4}, we have
		\begin{align}\label{F}
			\Inv(E_3\text{ : }2A)=5.31\text{ and }\Inv(E_3\text{ : }2B)=4.7.31.
		\end{align}
		Assume now that $h\in F_3\backslash \{1\}$ and $h$ is an involution. As mentioned in Remark \ref{rem:JankoRemark}, the group $E_3$ is normalised by $E_2$. Furthermore, $|E_2E_3|=2^{16}$, and $N_{D_2}(E_3)=E_3Q$, where $Q$ is a subgroup of $F_3\cong \mathrm{L}_5(2)$ of shape $2^6: (L_3(2)\times \Sym(3))$ (see \cite[Section 5, fourth paragraph]{Janko4}).
		
		Now, $F_3\cong \mathrm{L}_5(2)$ has two conjugacy classes of involutions: the $2$-central involutions and the non $2$-central involutions. Moreover, each such involution is $F_3$-conjugate to an element of the  normal subgroup $O_2(Q)\cong 2^6$ of $Q\le N_{D_2}(E_3)$. Indeed, $21$ of the $2^6-1$ involutions in $O_2(Q)$ are $2$-central in $F_3$, and $42$ are non $2$-central. Note that $|E_2E_3:E_2|=2^3$, and ${E_2E_3}\le [{D_2},{D_2}]$. Let $\{h_1=1, h_2,\hdots,h_8\}\subseteq [D_2,D_2]$ be a set of representatives for the $E_2$-cosets in $E_2E_3$. Then, for $X\in\{A,B\}$,
		\begin{align*}
			\Inv(E_2E_3\backslash E_3\text{ : }2X)&= \Inv(E_2E_3\text{ : }2X)- \Inv(E_3\text{ : }2X)\\
			&=\left(\sum_{i=1}^8 \Inv(E_2h_i\text{ : }2X)\right)- \Inv(E_3\text{ : }2X).
		\end{align*}
		It then follows from (\ref{lab:E2}), (\ref{t'N}), and (\ref{F}) that
		\begin{align*}
			\Inv(E_2E_3\backslash E_3\text{ : }2A)& =1387+7.112-5.31=2016.
		\end{align*}
		In exactly the same way, we calculate that $\Inv(E_2E_3\backslash E_3\text{ : }2B)=3360$.
		On the other hand, $E_2E_3=E_3O_2(Q)$, so
		$$
		\Inv(E_2E_3\setminus E_3)=\sum_{g\in O_2(Q)\backslash \{1\}}\Inv(E_2g).
		$$
		As mentioned above, $21$ of the $2^6-1$ involutions in $O_2(Q)\cong 2^6$ are $2$-central in $F_3$, with $42$ non $2$-central. It follows that, for  $x, y\in O_2(Q)\le F_3\cong \mathrm{L}_5(2)$ with $x$ a $2$-central involution, and $y$ a non $2$-central involution,
		\begin{align}\label{Eq1}
			21\,\Inv(E_3x\text{ : }2A)+42\,\Inv(E_3y\text{ : }2A)&=2016,\text{ and }\nonumber\\
			21\,\Inv(E_3x\text{ : }2B)+42\,\Inv(E_3y\text{ : }2B)&=3360.
		\end{align}
		In particular, $\Inv(E_2E_3\backslash E_3)=5376=21\,\Inv(E_3x)+42\,\Inv(E_3y)$.
		Since $\Inv(E_3x)=|C_{E_3}(x)|$ and $\Inv(E_3y)=|C_{E_3}(y)|$ are powers of $2$, it is an easy exercise to show that the only possibility is $\Inv(E_3x)=2^7$ and $\Inv(E_3y)=2^6$.
		
		Finally, we use \texttt{Magma}~\cite{MAGMA} to compute the $\mathrm{J}_4$-conjugacy classes of involutions in which the involutions in $D_3$ lie (see Remark \ref{rem:J4Magma} above). We find that the group $D_3$ contains $134705$ involutions in the class $2A$, and $343108$ involutions in the class $2B$. Since $\mathrm{L}_5(2)$ has $465$ $2$-central involutions and $6510$ non $2$-central involutions, we deduce from (\ref{Count}) and (\ref{F}) that
		\begin{align}\label{Eq2}
			465\,\Inv(E_3x\text{ : }2A)+6510\,\Inv(E_3y\text{ : }2A)+5.31&=134705\text{, and }\nonumber \\
			465\,\Inv(E_3x\text{ : }2B) + 6510\,\Inv(E_3y\text{ : }2B)+4.7.31&=343108.
		\end{align}
		The system of equations given by (\ref{Eq1}) and (\ref{Eq2}) can be solved to give $\Inv(E_3x\text{ : }2A)=64\text{, }\Inv(E_3y\text{ : }2A)=16\text{, }\Inv(E_3x\text{ : }2B)=64\text{, and }\\ \Inv(E_3y\text{ : }2B)=48.$
		Part (3) then follows from (\ref{Count}).
	\end{proof}

	\begin{corollary}\label{J4LemmaCor} Let $D_1, D_2, D_3$ be as described above, and let $H$ be a subgroup of $T=\mathrm{J}_4$. Then $T^{[T:H]}$ has base size $2$ for each of the following subgroups $H$.\begin{enumerate}[\upshape(1)]
			\item $O_2(D_1)< H < D_1$, with $H/O_2(D_1)$ maximal in $D_1/O_2(D_1)\cong \mathrm{M}_{24}$;
			\item $O_2(D_2)< H < D_2$ such that one of the following holds:
			\begin{enumerate}[\upshape(a)]
				\item $H/O_2(D_2)$ is maximal in $D_2/O_2(D_2)\cong 3.\mathrm{M}_{22}:2$, and $H/O_2(D_2)$ is not of shape $3.\mathrm{L}_3(4).2$ or $3.\mathrm{M}_{22}$;
				\item $H/O_2(D_2)$ is a maximal subgroup of $3.\mathrm{L}_3(4).2<3.\mathrm{M}_{22}:2$, and $H/O_2(D_2)\neq 3.\mathrm{L}_3(4)$;
				\item $H/O_2(D_2)$ is a maximal subgroup of $3.\mathrm{M}_{22}<3.\mathrm{M}_{22}:2$ isomorphic to $3.\Alt(7)$;\end{enumerate}
			\item  $O_2(D_3)< H < D_3$, with $H/O_2(D_3)$ maximal in $D_3/O_2(D_3)\cong  \mathrm{L}_5(2)$;
			\item $H$ conjugate to a maximal subgroup $\mathrm{M}_{24}$ of $D_1$ (in one of two $D_1$-conjugacy classes); or the maximal $F_2\cong 6.\mathrm{M}_{22}:2$ in $D_2$; or the maximal $F_3\cong \mathrm{L}_5(2)$ in $D_3$;
			\item $H$ maximal of shape $[2^7].3.\mathrm{L}_3(4).2$ in the maximal subgroup of the form $2^{1+12}.3.\mathrm{L}_3(4).2$ of $D_2$.
		\end{enumerate}
	\end{corollary}

	\begin{proof} Let $H$ be one of the subgroups in the statement. We will show that
		\begin{align}\label{lab:WTP}
			\widehat{Q}(T, H, 2):=\sum_{i=1}^m\frac{|x_i^T\cap H|^2}{|x_i^T|}\text{ is less than $1$,}
		\end{align}
		where $x_1$, $\hdots$, $x_m$ are representatives for the conjugacy classes of elements of $T$ of prime order. By Proposition \ref{BOBWProp}, this suffices to show that $T^{[T:H]}$ has base size $2$.
		Let $\mathcal{X}=\{x_1,\dots, x_m\}$.
		
		To prove (\ref{lab:WTP}), we compute upper bounds for
		the quantities $\frac{|x^T\cap H|^2}{|x^T|}$ for each $x\in\mathcal{X}$.
		For any fixed prime $p$, the following crude bound is available:
		\begin{align}\label{Crude}
			\sum_{x\in \mathcal{X},|x|=p}\frac{|x^T\cap H|^2}{|x^T|}\le \frac{|\{y\in H\text{ : }|y|=p\}|^2}{\min{\{ |x^T| \text{ : }x\in \mathcal{X},|x|=p\}}}.
		\end{align}
		The right-hand side above can be found, for each $H$ and each prime $p$, using Magma (see Remark \ref{rem:J4Magma} above). The bounds at (\ref{Crude}) in fact suffice to prove (\ref{lab:WTP}) for all subgroups $H$ in parts (4) and (5) of the corollary.
		
		In parts (1), (2) and (3), however, we need to be more careful. We can use \texttt{Magma}~\cite{MAGMA} (see Remark \ref{rem:J4Magma} above) and Lemma \ref{J4Lemma}, to compute the quantities
		$$
		\frac{|z^T\cap H|^2}{|z^T|}\text{ and }\frac{|f^T\cap H|^2}{|f^T|}
		$$
		precisely for each subgroup $H$, (recalling that $z\in 2A$ and $f\in 2B$).
		Putting these together with the bounds for odd prime orders coming from (\ref{Crude}) above then yields $\widehat{Q}(T,H,2)<1$ in each case, as required.
	\end{proof}
	
	We now have the tools needed to prove Lemma~\ref{Base2Lemma} for $\mathrm{J}_4$. Recall the notion of the depth $\ell_G(H)$ of a subgroup $H$ of a group $G$ introduced just before Strategy~\ref{Strategy}.

	\begin{proposition}\label{J4LemmaCor2} Let $H$ be a proper subgroup of $T=\mathrm{J}_4$.
		\begin{enumerate}[\upshape(1)]
			\item If $T^{[T:H]}$ has base size at least $3$, then one of the following holds:
			\begin{enumerate}[\upshape(a)]
				\item $\ell_T(H)=1$ and $H$ is $T$-conjugate to one of $D_1$, $D_2$, or $D_3$;
				\item $\ell_T(H)=2$ and either $H$ is $T$-conjugate to $[D_2,D_2]\cong 2^{1+12}.3.\mathrm{M}_{22}$; or $H$ is $T$-conjugate to a maximal subgroup of $D_2$ containing $O_2(D_2)$ of shape $2^{1+12}.3.\mathrm{L}_3(4):2$.
				\item $\ell_T(H)=3$ and $H$ is $T$-conjugate to a second maximal subgroup of $D_2$ containing $O_2(D_2)$ of shape $2^{1+12}.3.\mathrm{L}_3(4)$.
			\end{enumerate}
			\item If $\ell_T(H)\geq 4$, then $H$ is properly contained in a subgroup $K$ of $T$ such that $T^{[T:K]}$ has base size $2$.
			\item Proposition $\ref{Base2Lemma}$, and in particular Table~$\ref{tab:basesize}$ holds for $T=\mathrm{J}_4$.
		\end{enumerate}
	\end{proposition}
	
	Since it is not necessary for our purposes, we have not attempted to resolve whether or not the permutation group $T^{[T:H]}$ has base size $2$ in the cases where $H$ has shape $2^{1+12}.3.\mathrm{L}_3(4):2$ (in Part (1)(b)), or $2^{1+12}.3.\mathrm{L}_3(4)$ (in Part (1)(c)).
	
	\begin{proof}
		We first prove part (1) for subgroups $H$ with $\ell(H)\leq 3$. Let $H$ be a proper subgroup of $T$ such that $T^{[T:H]}$ has base size larger than $2$. Then by \cite{BOBW}, and as noted at the beginning of the section, $H$ is conjugate to a subgroup of $D_1$, $D_2$, or $D_3$. Thus if $\ell_T(H)=1$ then Part (1)(a) holds.  Therefore we may assume both that $\ell_T(H)\geq2$, and that $H<D_i$ for some $i=1,2,3$. Recall that, if $H<K<T$ then also $T^{[T:K]}$ has base size larger than $2$. Also $D_i=E_iF_i$, where $E_i=O_2(D_i)$, and $E_i\cap F_i$ is trivial with $E_i$ elementary abelian and $F_i$ acting irreducibly on it, if $i=1$ or $3$, and  $E_2\cap F_2=Z(E_2)=\langle z\rangle$ of order $2$ with $F_2$ irreducble on $E_2/Z(E_2)$, (see Remark~\ref{rem:JankoRemark}). Also using \texttt{Magma}~\cite{MAGMA}, we find that there are two $D_1$-conjugacy classes of complements to $E_1$, with representatives $F_1$ and $F_{11}$, and we obtain,
		for each $i$,  a complete set $\mathcal{M}(D_i)$  of representatives for the $D_i$-conjugacy classes of maximal subgroups of $D_i$, as:
		\begin{align*}
			\mathcal{M}(D_1) &= \{F_1, F_{11}\}\cup \{E_1I\text{ : }I<F_1\text{  maximal}\}\cup \{E_1I\text{ : }I<F_{11}\text{  maximal}\},\\
			\mathcal{M}(D_2) &= \{F_2\}\cup \{E_2I\text{ : }I<F_2\text{  maximal}\},\\
			\mathcal{M}(D_3) &= \{F_3\}\cup \{E_3I\text{ : }I<F_3\text{  maximal}\}.\\
		\end{align*}
		Assume now that $\ell_T(H)=2$ so that $H\in\mathcal{M}(D_i)$ for some $i$.  It then follows from Corollary \ref{J4LemmaCor} that $i=2$ and, up to $T$-conjugacy, $H=E_2I$ with $I/\langle z\rangle$ maximal in $F_2/\langle z\rangle\cong 3.\mathrm{M}_{22}.2$ of shape $3.\mathrm{L}_3(4).2$ or $3.\mathrm{M}_{22}$. This proves (1)(b).
		
		Suppose next that $\ell_T(H)=3$. Then by part (1)(b), $H$ is  $T$-conjugate to a maximal subgroup of either $[D_2,D_2]$ or a subgroup $K=E_2I$ with $I\cong 3.\mathrm{L}_3(4).2<3.\mathrm{M}_{22}.2$. Consider first the latter case. Using Magma, we see that $E_2/Z(E_2)$ has composition length $2$ as an $\mathbb{F}_2[I]$-module, with both composition factors of dimension $6$. Let $E_0/Z(E_2)$ be an $I$-submodule of $E_2/Z(E_2)$ of dimension $6$. Then
		using Magma again, we see that the set
		\begin{equation}\label{J4-ell31}
			\{E_0I\}\cup\{E_2I_1\text{ : }I_1<I\text{  maximal}\}
		\end{equation}
		is a complete set of representatives for the $K$-conjugacy classes of maximal subgroups of $K$. The subgroup $E_0I$ has shape $[2^7].3.\mathrm{L}_3(4).2$, and  by Corollary \ref{J4LemmaCor}(5), $T^{[T:E_0I]}$ has base size $2$. Thus, $H$ is $K$-conjugate to $E_2I_1$ for some maximal subgroup $I_1$ of $I$. By Corollary \ref{J4LemmaCor}(2)(b), it follows that $I_1\cong 3.\mathrm{L}_3(4)$, as in Part (1)(c).
		Suppose now that $H$ is maximal in $[D_2,D_2]$.
		By using Magma, we find that the set
		\begin{equation}\label{J4-ell32}
			\{[F_2,F_2]\cong 3.\mathrm{M}_{22}\}\cup\{E_2J\text{ : }J<[F_2,F_2]\text{ maximal}\}
		\end{equation}
		is a complete set of representatives for the  $[D_2,D_2]$-conjugacy classes of maximal subgroups of $[D_2,D_2]$. Since,  by Corollary \ref{J4LemmaCor}(4), $T^{[T:F_2]}$ has base size $2$, it follows that also $T^{[T:L]}$ has base size $2$, where $L=[F_2,F_2]< F_2$. Thus $H$ is not conjugate to $[F_2,F_2]L$, and we may assume that $H=E_2J$ for some maximal subgroup $J$ of $[F_2,F_2]\cong 3.\mathrm{M}_{22}$. If $J<M<3.\mathrm{M}_{22}:2$ with $M$ maximal and $M\not\cong 3.\mathrm{L}_3(4).2$ or $3,\mathrm{M}_{22}$, then we would have $T^{[T:H]}$ of base size $2$ on applying  Corollary \ref{J4LemmaCor}(2)(a). Thus this is not the case, and, by using the Web Atlas, we see that the only possibilities for $J$ are the subgroups of shape $3.\Alt(7)$ or $3.\mathrm{L}_3(4)$. By Corollary \ref{J4LemmaCor}(2)(c), it follows that $H\cong E_2J$, where $J=3.\mathrm{L}_3(4)$. This proves completes the proof of Part (1) when $\ell(H)\leq 3$.
		
		Next we complete the proof of (1) in the case of subgroups $H$ with $\ell(H)\geq 4$, and prove (2).
		Note that, if (2) holds then, for any subgroup $H$ with $\ell_T(H)\geq 4$, the base size of $T^{[T:H]}$ is at most 2, and hence Part (1) lists all possibilities for $H$ where the base size of $T^{[T:H]}$ could possibly be greater than $2$.
		So assume that $H$ is a subgroup of $T$, $\ell_T(H)\geq 4$, and $H$ is not properly contained in a subgroup $K$ of $T$ with $T^{[T:K]}$ having base size $2$. Then, by our arguments up to this point,  $H$ is $T$-conjugate to a subgroup of $K_1:=E_2I_1$, where $I_1\cong 3.\mathrm{L}_3(4)$. Since $\ell_T(H)\geq 4$, we may assume that $H<K_1=E_2I_1$. Using Magma, we find that
		$$
		\{E_0I_1\}\cup\{E_2I_2\text{ : }I_2<I_1\text{  maximal}\}
		$$
		is a complete set of representatives for the $K_1$-conjugacy classes of maximal subgroups of $K_1$, where $E_0$ is as defined in the previous paragraph. Consider  $M=E_2I_2$, for some maximal subgroup $I_2$ of $I_1$. We can check using Magma that, for every subgroup $I_2<I_1<[F_2,F_2]\cong 3.\mathrm{M}_{22}$, $I_2$ is contained in either a maximal subgroup of $F_2$ which is not isomorphic to $3.\mathrm{M}_{22}$ or $3.\mathrm{L}_3(4).2$; or a maximal $3.\Alt(7)$ subgroup of $3.\mathrm{M}_{22}$. Thus $T^{[T:M]}$ has base size 2 by Corollary~\ref{J4LemmaCor}, and hence $H$ is not contained in any such subgroup $M$. Thus we may assume that $H\le E_0I_1$ and $H$ projects onto $I_1$. Now $E_0I_1$ is an index 2 subgroup of $E_0I$, with $I=3.\mathrm{L}_3(4).2$ as above, and by Corollary~\ref{J4LemmaCor}(5), the permutation group $T^{[T:E_0I]}$ has base size $2$.  Hence also $T^{[T:J]}$ has base size $2$, for all subgroups $J$ of $E_0I_1$.
		This is a contradiction to our assumption that $T^{[T:H]}$ has base size larger than $2$, and hence Part (2) is proved.

		Finally, we complete the proof of Part (3). If $M=D_1\cong 2^{11}:\mathrm{M}_{24}$, then the subdegrees of $T^{[T:M]}$ are given in \cite[Lemma 2.23.1]{Ivanov95}. Thus, we can compute $d(T,M)$ precisely (the greatest common divisor of the orders of the stabilisers in  $M^{[T:M]}$ of points in  $[T:M]\setminus\{M\}$): we get $d(T,M)=24$. The first entry in the final column of the third row of Table \ref{tab:basesize} follows.
		
		Next, suppose that $M=D_3=E_3F_3=2^{10}:\mathrm{L}_5(2)$. As noted above, the maximal subgroups of $M$ are $T$-conjugate to either $F_3$, or $E_3I$, where $I$ is a maximal subgroup of $F_3$. Now, since $(|M|,|T:M|-1)=1$, there exists, for each prime $p$ dividing $|M|$, a point stabiliser $M_{(p)}:=M\cap M^{t_p}$ in $M^{[T:M]}$ containing a Sylow $p$-subgroup of $M$. If $L$ is a maximal subgroup of $M$ containing a Sylow $31$-subgroup of $M$, then either $L$ is $T$-conjugate to $F_3$, or $L$ contains $E_3$ and $|L/E_3|=5.31$ (one can see this by inspection of the maximal subgroups of $\mathrm{L}_5(2)\cong M/O_2(M)$). Thus, $|M_{(31)}|$ divides $2^{10}.3^2.5.7.31$. Similarly, $M_{(2)}$ has order dividing $2^{20}.3^2.5.7$. Thus, the greatest common divisor of the orders of the point stabilisers in $M^{[M:T]}$ divides $2^{10}.3^2.5.7$. On the other hand, we can compute the rank of $T^{[T:M]}$ using the permutation character and GAP \cite{GAP} -- it is $27$. Thus, at least one of the $M$-orbits in $[T:M]$ must be of size greater than or equal to $(|T:M|-1)/26$. It follows that at least one point stabiliser in $M$ has order less than $26|M|/(|T:M|-1)<32$. Thus, $d(T,M)$ divides $2^{10}.3^2.5.7$ and is less than $32$. It follows that $d(T,M)$ divides $2^4.3^2.5.7$ and $d(T,M)\le 30$, as required.
		
		To prove Part (3), suppose that
		$$
		M\in\{D_2,[D_2,D_2],2^{1+12}.3.\mathrm{L}_3(4):2,2^{1+12}.3.\mathrm{L}_3(4)\}.
		$$
		Then $M$ is $T$-conjugate to a subgroup of $D_2$. Thus, $d(T,M)$ divides $d(T,D_2)$. The subdegrees of $T^{[T:D_2]}$ are given in \cite[Table 2]{BrayCameron}, and the number $d(T,D_2)$ can then be computed immediately: we get $d(T,D_2)=2$. The remaining entries in the third column of the $\mathrm{J}_4$-row of Table \ref{tab:basesize} follow. In the case where $M=D_2$, the base size of $T^{[T:M]}$ is  $3$, by \cite[Table 2]{BOBW}. This action is equivalent to the conjugation action of $T$ on the class of $2A$-involutions of $T$. The orders of the $2$-point stabilisers in this action have been computed by Bray et al in \cite{BrayCameron} and are listed in the column headed $s_2$ in \cite[Table 2]{BrayCameron}. The smallest order is $16$. If now $H= [D_2,D_2]$, of index $2$ in $M$, then each $2$-point stabiliser $H\cap H^x$ in the permutation group $T^{[T:H]}$ has index at most $4$ in $M\cap M^x$, and since $|M\cap M^x|\geq 16$, it follows that $H\cap H^x\ne 1$. Thus $T^{[T:H]}$ also has base size greater than $2$.
	\end{proof}
	
	We remark that by proving Proposition \ref{J4LemmaCor2}(3), we have completed the proof of Proposition \ref{Base2Lemma}.
	As final preparation for proving that $T=\mathrm{J}_4$ is totally $2$-closed,  we deal separately with a specific coset action.
	
	\begin{proposition}\label{J4Prelim}
		Let $T=\mathrm{J}_4$ and $D_1=E_1\rtimes F_1=2^{11}:\mathrm{M}_{24}$ as above, and consider $H=E_1\rtimes \mathrm{M}_{23}<D_1$ (so $\ell_T(H)=2$). Then $T^{[T:H]}$ is $2$-closed.
	\end{proposition}
	
	\begin{proof} We use the notation in Notation \ref{not2}, with $G=T$ and $M=D_1$. So $\Omega=[T:H]$, $\Sigma=[T:M]$ with $s=|\Sigma|$, $L=T^{\Sigma}$, $M=T_{\Delta}$ (with $\Delta\in\Sigma$), $H=T_{\omega}$ (with $\omega\in\Delta$), etc. By Lemma~\ref{BaseLemma}(b)(i), $L^{(2),\Sigma}=L$, since $T=\Aut(T)$  and $T$ has no $2$-transitive representations. 
		We also set $X=T^{(2),\Omega}$ and $Y=X_{\Delta}^{\Delta}\le M^{(2),\Delta}$. The group $R=M^{\Delta}\cong \mathrm{M}_{24}$ is $2$-transitive of degree $24$. Hence, we have $\mathrm{M}_{24}\cong R=M^\Delta\le Y\le M^{(2),\Delta}\cong \Sym(24)$. The only overgroups of $\mathrm{M}_{24}$ in $\Sym(24)$ are $\mathrm{M}_{24}$, $\Alt(24)$, and $\Sym(24)$. Thus, $Y\in\{\mathrm{M}_{24},\Alt(24),\Sym(24)\}$.
		
		Assume that $T^{[T:H]}$ is not $2$-closed. Then, by Lemma~\ref{Crucial}, $X=N\rtimes T$ with $N$ a nontrivial subdirect subgroup of $B_A$, for some non-trivial normal subgroup $A$ of $Y$. Hence, $A\in\{\mathrm{M}_{24},\Alt(24),\Sym(24)\}$.
		
		We now examine a part $\Delta_2\in\Sigma$ such that $M\cap M_2= L_2(23)$ (the groups $M\cap M_j$ are listed in \cite[Lemma 2.23.1]{Ivanov}).
		Suppose that the subdirect subgroup $\rho_{1,2}(N)\le A\times A$ is not a full diagonal subgroup of $A\times A$. Then since $A$ is almost simple, $\rho_{1,2}(N)$ must contain $\soc(A)\times\soc(A)$. Since $\soc(A)$ acts transitively on $\Delta$, it follows that $\rho_{1,2}(N)^{\Delta\times\Delta_2}$ is transitive. On the other hand, by Lemma~\ref{Crucial}(b), the group $\rho_{1,2}(N)^{\Delta\times\Delta_2}$ leaves invariant each $(M\cap M_2)^{\Delta\times\Delta_2}$-orbit.
		Since $M\cap M_2\cong L_2(23)$ cannot act transitively on a set of size $24^2$, we have a contradiction. Thus $\rho_{1,2}(N)\le A\times A$ is a full diagonal subgroup of $A\times A$. It then follows from Proposition \ref{RedProp}(b) that $N$ is a full diagonal subgroup of $B_A$. In particular, since $X\cong N\rtimes T$ with $N\cong A\in\{\mathrm{M}_{24},\Alt(24),\Sym(24)\}$ and $T\cong \mathrm{J}_4$, we have $X=N\times T$.
		
		We examine further the part $\Delta_2\in\Sigma$ with $M\cap M_2= L_2(23)$.
		We claim that
		\begin{align}\label{lab:L223}
			C_T(M\cap M_2)\text{ has even order.}
		\end{align}
		To prove this, consider the setwise stabiliser $T_{\{\Delta\cup \Delta_2\}}$ of $\Delta\cup \Delta_2$ in $T$. Note first that the subdegrees of $T^{[T:M]}$ are pairwise distinct (see \cite[Lemma 2.23.1]{Ivanov}), so each $T$-orbit in $\Sigma^{(2)}:=\{(\Delta',\Delta'')\text{ : }\Delta',\Delta''\in\Sigma,\Delta'\neq\Delta''\}$ is self-paired. It follows that $T_{\{\Delta\cup \Delta_2\}}$ contains an element, say $g$, which interchanges $\Delta$ and $\Delta_2$. Moreover, $T_{\{\Delta\cup \Delta_2\}}$ contains $M\cap M_2\cong L_2(23)$. Since $M\cap M_2\cap O_2(M)=1$, and the unique conjugacy class of subgroups of $\mathrm{M}_{24}$ isomorphic to $L_2(23)$ acts transitively on the cosets of $\mathrm{M}_{23}$ in $\mathrm{M}_{24}$, the group $(M\cap M_{2})^{\Delta}$ is transitive. It follows that $T_{\{\Delta\cup \Delta_2\}}^{\Delta\cup \Delta_2}$ is transitive, and that the kernel of the action of $T_{\{\Delta\cup \Delta_2\}}$ on the block system $\{\Delta,\Delta_2\}$ is the group $M\cap M_2\cong L_2(23)$.
		
		Now, $(M\cap M_2)^{\Delta\cup\Delta_2}=\rho_{1,2}(M\cap M_2)^{\Delta\cup\Delta_2}$, by the definition of $\rho_{1,2}$. Since $M\cap M_2<X$ and $X$ normalises $N$, we have that $\rho_{1,2}(M\cap M_2)^{\Delta\cup\Delta_2}$ normalises $\rho_{1,2}(N)^{\Delta\cup\Delta_2}$. As shown above, $\rho_{1,2}(N)^{\Delta\cup\Delta_2}$ is a full diagonal subgroup of $A\times A\le\Sym(24)\times\Sym(24)$, that is to say, for some $\theta\in\Aut(A)$ we have
		$\rho_{1,2}(N)^{\Delta\cup\Delta_2}=\{(a,a^{\theta})\text{ : }a\in A\}$.  Now $A\le\Sym(24)$ is almost simple and primitive, so it has trivial centraliser in $\Sym(24)$, and $N_{\Sym(24)}(A)=\Aut(A)\cong PGL_2(23), \Sym(24),\Sym(24)$ for $A= L_2(23), \Alt(24), \Sym(24)$, respectively. Hence
		the normaliser of $\rho_{1,2}(N)^{\Delta\cup\Delta_2}$ in $(\Sym(\Delta\cup\Delta_2))_{\Delta,\Delta_2}$ is $\{(b,b^{\theta})\text{ : }b\in \Aut(A)\}$. It follows that $\rho_{1,2}(M\cap M_2)^{\Delta\cup\Delta_2}
		=\{(d,d^{\theta})\text{ : }d\in D\}$, for some subgroup $D$ of $\Sym(24)$ with $D\cong L_2(23)$.
		
		Next, consider the natural epimorphism $\mu:X_{\{\Delta\cup \Delta_2\}}\rightarrow X_{\{\Delta\cup \Delta_2\}}^{\Delta\cup \Delta_2}\le\Sym(24)\wr 2$. Since the element $g$ (which interchanges $\Delta$ and $\Delta_2$) centralises $N$, its image $\mu(g)$ centralises $\mu(N)=\rho_{1,2}(N)^{\Delta\cup\Delta_2} =\{(a,a^{\theta})\text{ : }a\in A\}$. Thus, writing $\mu(g)=(y_1,y_2)\sigma$ (where $\sigma\in\Sym(2)\backslash \{1\}$ and $y_i\in \Sym(24)$), we have
		\begin{align*}
			(a,a^{\theta})=(a,a^{\theta})^{\mu(g)}=(a^{\theta y_2},a^{y_1})\text{ for all }a\in A.
		\end{align*}
		It follows that $a^{\theta y_2}=a=a^{\theta y_1^{-1}}$, for all $a \in A$. Thus, since $C_{\Sym(24)}(A)=1$, we have $y_1=y_2^{-1}=\theta$. Hence, for each $(d,d^{\theta})\in \rho_{1,2}(M\cap M_2)^{\Delta\cup\Delta_2}= \mu(M\cap M_2)$,
		\begin{align*}
			(d,d^{\theta})^{\mu(g)}=(d^{\theta y_2},d^{y_1})=(d,d^{\theta}),
		\end{align*}
		and therefore $\mu(g)$ centralises $\mu(M\cap M_2)$. Since $M\cap M_2$ is simple and $\mu(M\cap M_2)\neq 1$, the kernel of the restriction of $\mu$ to $M\cap M_2$ is trivial. Hence, $g$ centralises $M\cap M_2$, and since $g$ has even order, claim (\ref{lab:L223}) follows.
		
		This leads to a contradiction, since $M\cap M_{2}$ contains elements of order $23$, and the unique $T$-conjugacy class of such elements has centraliser of order $23$ (see the Web Atlas \cite{WebAtlas}). Thus, $T^{[T:H]}$ is $2$-closed.
	\end{proof}

	
	We now prove that $T=\mathrm{J}_4$ is totally $2$-closed.

	\begin{proposition}\label{J4prop}
		The Janko group $\mathrm{J}_4$ is totally $2$-closed.
	\end{proposition}
	
	\begin{proof}
		We use the notation fixed at the beginning of the section, so that $T=\mathrm{J}_4$, $D_1\cong 2^{11}:\mathrm{M}_{24}$ etc. Our task is to prove, using Strategy~\ref{Strategy}, that $T^\Omega=T^{(2),\Omega}$, where $\Omega=[T:H]$,     for each proper nontrivial subgroup $H$ of $T$.
		Since  $T=\Aut(T)$  and $T$ has no $2$-transitive representations, it follows from
		Lemma \ref{BaseLemma}(b)(i) that this holds for all maximal subgroups $H$. Thus we proceed to stage (2) of Step 1 in Strategy~\ref{Strategy}: constructing the set $\mathcal{S}_2$.
		By \cite[Table 2]{BOBW}, there are precisely three conjugacy classes of maximal subgroups $M$ of $T$ such that $T^{[T:M]}$ has base size at least $3$, with representatives $D_1\cong 2^{11}:\mathrm{M}_{24}$, $D_2\cong 2^{1+12}.3:\mathrm{M}_{22}.2$, and $D_3=2^{10}:\mathrm{L}_5(2)$ as defined at the beginning of the section. Therefore according to Strategy~\ref{Strategy}, the set $\mathcal{S}_2$ consists of all non-maximal subgroups $H$ such that the only maximal subgroups of $T$ containing $H$ are conjugate to $D_1$, $D_2$, or $D_3$.
		
		We now move on to Strategy \ref{Strategy}, Step 2, stage (1). We need to prove that $T^{[T:H]}$ is $2$-closed for all $H\in\mathcal{S}_2$ with $\ell_T(H)=2$. Let $H\in\mathcal{S}_2$, and assume that $T^{[T:H]}\neq T^{(2),[T:H]}$. Let $M$ be a maximal subgroup of $T$ containing $H$ so that, without loss of generality, $M=D_i$, for some $i=1,2,3$ and $H$ is maximal in $M$. We use the rest of the notation in Notation \ref{not2}, with $T=G$. So $\Omega=[T:H]$, $\Sigma=[T:M]$, $L=T^{\Sigma}$, $M=T_{\Delta}$ (with $\Delta\in\Sigma$), $H=T_{\omega}$ (with $\omega\in\Delta$), etc. Note that $L^{(2),\Sigma}=L$, by Step 1. We also set $X=T^{(2),\Omega}$, so $X>T$, and $Y=X_{\Delta}^{\Delta}\le M^{(2),\Delta}$.
		
		Suppose first that $i=1$ or $3$. Then $M=E_i\rtimes F_i$, where $E_i$ is elementary abelian, and $F_i$ is a simple group acting faithfully and irreducibly on $E_i$. Then, as mentioned in the proof of Proposition \ref{J4LemmaCor2}(1)(b), since $H$ is maximal in $M$,  either $H\cong F_i$ or $H=E_iI$, where $I$ is a maximal subgroup of $F_i$. It follows that either $M^{[M:H]}\cong E_iF_i$ is a primitive group of affine type, of degree $|E_i|$; or $M^{[M:H]}\cong F_i$ is a primitive almost simple group, of degree $|F_i:I|$. Since $T^{[T:H]}$ is not $2$-closed, it follows that the normal subgroup $A$ of $Y\le M^{(2),\Delta}$ from Proposition \ref{MainCor1}(2) is transitive. On the other hand, Proposition \ref{MainCor1}(2) and Proposition \ref{Base2Lemma} imply that the size of an $A$-orbit in $\Delta$ divides $d(T,M)$, with
		$d(T,M)$ as in Proposition \ref{Base2Lemma}(b), that is to say, $|M:H|=|\Delta|$ divides $d(T,M)$. Now, by Proposition \ref{Base2Lemma}(b), $d(T,M)$ divides $g(T,M)$, with   $g(T,M)$ as in Table \ref{tab:basesize}. In both cases $g(T,M)<|E_i|$, so  $M^{[M:H]}\cong F_i$ is nonabelian simple and primitive. If $i=1$ then $g(T,M)=24$ and it follows that $|M:H|=24$ and $I\cong \mathrm{M}_{23}$ maximal in $F_1\cong \mathrm{M}_{24}$; however, by Proposition \ref{J4Prelim}, $T^{[T:H]}$ is $2$-closed in this case, which is a contradiction.  Thus $i=3$, and by Proposition \ref{Base2Lemma}(b), $d(T,M)\leq 30$. This means that $|M:H|=|F_3:I|\leq 30$, which is not possible for any maximal subgroup $I$ of $F_3\cong \mathrm{L}_5(2)$.
		Thus $T^{[T:H]}$ is $2$-closed for all maximal subgroups $H$ of $M\in\{D_1,D_3\}$.
		
		Next, assume that $H$ is maximal in $M=D_2$. Then, as mentioned in the proof of Proposition \ref{J4LemmaCor}, either $H$ is $M$-conjugate to a supplement $F_2\cong 6.\mathrm{M}_{22}:2$ for $E_2=O_2(M)$ in $M$, or $H=E_2I$, for some $I<F_2$ maximal. Since $T^{[T:H]}$ is not $2$-closed, it follows that the normal subgroup $A$ of the primitive permutation group $Y\leq M^{(2),\Delta}$, from Proposition \ref{MainCor1}(2), is transitive. On the other hand, the size of an $A$-orbit divides $2$, by Proposition \ref{Base2Lemma}(b). Hence, we must have $|M:H|=|\Delta|=2$, and it follows that $H$ is the maximal subgroup $[M,M]\cong 2^{1+12}.3.\mathrm{M}_{22}$. By Proposition \ref{MainCor1}(1),
		for each $\Delta_2\in\Sigma\setminus\{\Delta\}$, the stabiliser $M_2=T_{\Delta_2}$ satisfies $M=(M\cap M_2)H$. We shall obtain a contradiction by proving that:
		\begin{align}\label{lab:2p}
			\text{for some $\Delta_2\in\Sigma\setminus\{\Delta\}$, the stabiliser $M\cap M_2=T_{\Delta,\Delta_2}$ is contained in $H$.}
		\end{align}
		We prove the claim using Magma and the $112$-dimensional representation of $T$ over $\mathbb{F}_2$. In particular, we  use the generating $112$-dimensional matrices $a$ and $b$ of $T$ given in the Web Atlas \cite{WebAtlas}, of orders $2$ and $4$, respectively. Using this notation, we have that $M$ is $T$-conjugate to $C_T(a)$, so we may assume that $M=C_T(a)$. Define
		$$
		t:=(ab^2)^4;c:=ab; y:={t^{c^2}}.
		$$
		Using the excellent computations done in \cite[Section 2]{BrayCameron}, we have that
		$$
		M=\langle [a,b]^5(ab^2)^6,ba(b^2)ab[a,ba(b^2)ab]^5
		ababab[a,ababab]^5\rangle.
		$$
		Then the action of $T$ on the set of $T$-cosets of $M$ is permutation isomorphic to the action of $T$ on the set of $T$-conjugates of $a$. Thus, the $2$-point stabilisers $M\cap M_2$ in $T^{[T:M]}$ are the groups $C_T(a)\cap C_T(a^g)$, for $g\in T$. Consider the $2$-point stabiliser $C_{M}(a^y)=C_T(a)\cap C_T(a^y)$. One can quickly check using \texttt{Magma}~\cite{MAGMA} that $(aa^y)^6$ is an element of $M$ of order $2$, and that $C_{M}((aa^y)^6)$ is contained in $[M,M]=H\cong 2^{1+12}.3.\mathrm{M}_{22}$. Hence, since $C_{M}(a^y)\le C_{M}((aa^y)^6)$, we have that $C_T(a)\cap C_T(a^y)=C_{M}(a^y)\le [M,M]=H$. Thus, (\ref{lab:2p}) is proved, yielding a contradiction. We conclude that $T^{[T:H]}$ is $2$-closed for all maximal subgroups $H$ of $M=D_2$, and hence that $T^{[T:H]}$ is $2$-closed for all $H\in\mathcal{S}_2$ with $\ell_T(H)=2$.

		This completes stage (1) of Step 2 of Strategy \ref{Strategy}. For stage (2) of Step 2, we construct the set $\mathcal{S}_3$. By Proposition \ref{J4LemmaCor2}, we remove from $\mathcal{S}_2$ all (conjugates of) proper subgroups of   $D_1$, $D_2$, and $D_3$ except those (conjugates of) subgroups $H$ of $D_2$ such that $\ell_T(H)\geq 3$, and if $H<K<T$, then $K$ is conjugate to one of the groups in
		$$
		\{D_2,\ [D_2,D_2],\ 2^{1+12}.3.\mathrm{L}_3(4).2\}.
		$$
		The set $\mathcal{S}_3$ consists precisely of these remaining subgroups $H$.
		
		We now commence Strategy \ref{Strategy}, Step 3, stage (1). We need to prove that $T^{[T:H]}$ is $2$-closed for all $H\in\mathcal{S}_3$ with $\ell_T(H)=3$. We may assume that such a subgroup $H$ is contained in $M:=D_2$. Then $H$ must be a maximal subgroup in either $[M,M]$, or a subgroup of shape $2^{1+12}.3.\mathrm{L}_3(4).2$. We listed conjugacy class representatives for these subgroups in the proof of Proposition \ref{J4LemmaCor2}(1)(c), namely in \eqref{J4-ell31} and \eqref{J4-ell32}. We also showed that every maximal subgroup of $[M,M]$, apart from those conjugate to $2^{1+12}.3.\Alt(7)$ or $2^{1+12}.3.\mathrm{L}_3(4)$, is contained in a maximal subgroup $K$ of $M$ with $T^{[T:K]}$ of base size $2$. Thus the subgroups $H$ which remain to be considered are the following:
		\begin{enumerate}[\upshape(i)]
			\item the maximal subgroups $H$ of $2^{1+12}.3.\mathrm{L}_3(4).2<M$; and
			\item the subgroup $H:=2^{1+12}.3.\Alt(7)<[M,M]<M$.
		\end{enumerate}
		Let $H$ be one of these subgroups, and let $K:=2^{1+12}.3.\mathrm{L}_3(4).2$ in case (i), and $K=[M,M]$ in case (ii), so that $H$ is maximal in $K$.
		In either case the quantities $g(T,K)=2$ and $d(T,K)\leq 2$ (see Table \ref{tab:basesize} and Proposition~\ref{Base2Lemma}(b)).
		Assume that $T^{(2),   [T:H]}\neq T$. Then, by Proposition \ref{MainCor1}(2) applied to $\Sigma=[T:K]$,  and Proposition \ref{Base2Lemma}(b), the usual argument shows that the only possibility is for $K, H$ as in case (i) with $H=[K,K]\cong 2^{1+12}.3.\mathrm{L}_3(4)$. So assume that we are in this case. Then by (\ref{lab:2p}), there exists an element $t$ of $T$ such that $M\cap M^t\le [M,M]$. Hence, $K\cap K^t\le K\cap [M,M]\cong 2^{1+12}.3.\mathrm{L}_3(4)=[K,K]$. It then follows from Proposition \ref{MainCor1}(1) that $T^{[T:H]}$ is $2$-closed, which is a contradiction. This completes stage (1) of Step 3. For stage (2) of Step 3,
		we simply note that the set $\mathcal{S}_4$ is empty, by Proposition \ref{J4LemmaCor2}(2). Thus, it follows from Strategy~\ref{Strategy} that $T=\mathrm{J}_4$ is totally $2$-closed.
	\end{proof}

	\section{The exceptional groups}\label{sec:ExGroups}
	
	In this section, we prove that none of the finite simple exceptional groups of Lie type listed in Proposition \ref{lem:simple1}(a) can be totally $2$-closed.  
	To do this, we assume throughout this section that
	\begin{align}\label{AssEx}
		& T\in\{{}^3D_4(q),{}^2F_4(q)', F_4(q),{}^2E_6(q), E_6(q),E_7(q),E_8(q)\},\nonumber\\
		&\text{with }
		q=p^f\text{, }p\text{ prime.} 
	\end{align}
	
	
	Our aim is to prove the following.
	
	\begin{proposition}\label{ExProp}
		Let $T$ be one of the groups at (\ref{AssEx}). Then $T$ is not totally $2$-closed.
	\end{proposition}
	
	Recall that the subdegrees of a transitive permutation group are the orbit lengths of a point stabiliser.
	Our first lemma collates the subdegrees of the minimal degree permutation representations of $T$, which were determined in \cite[Theorem 2]{VasF4} for  $F_4(q)$, in \cite[Theorems 1, 2 and 3]{VasE6} for  $E_6(q), E_7(q), E_8(q)$, and in \cite[Theorems 3, 4, 5 and 6]{VasT} for 
	${}^3D_4(q)$,  ${}^2E_6(q)$, ${}^2F_4(q)\, (q>2)$, and ${}^2F_4(2)'$, respectively. 
	
	\begin{lemma}\label{lem:ExDegLemma}
		Let $M$ be a point stabiliser for $T$ in a minimal degree (necessarily primitive) permutation representation. Then $M$ and the point stabilisers in $M^{[T:M]}$ are known, and $M$, the degree $|T:M|$, the rank, and the subdegrees are given in Table $\ref{tab:minex}$.
	\end{lemma}
	
	\begin{table}[]
		\centering
		\begin{tabular}{c|c|c|c|c}
			$T$   &  $M$  &  $|T:M|$ & rank & subdegrees\\
			\hline
			$ {}^3D_4(q)$   &  $P_2$  &  $(q^8+q^4+1)(q+1)$ & $4$ & $1,q(q^3+1),q^{9},q^5(q^3+1)$\\
			${}^2F_4(q)$,   &  $P_{2,3}$  &  $(q^6+1)(q^3+1)(q+1)$ & $5$ & $1,q(q^2+1),q^{10},{q^4}(q^2+1),$\\
			\ $q>2$     & & & & $q^{7}(q^2+1)$\\
			${}^2F_4(2)'$   &  $A_2(3):2$  &  $1600$ & $3$ & $1,351,1248$\\
			$F_4(q)$   &  $P_1\text{ or }P_4$  &  $\frac{(q^{12}-1)(q^4+1)}{q-1}$ & $5$ & $1,\frac{(q^{4}-1)(q^3+1)}{q-1},q^{15},
			q^5\frac{q^6-1}{q-1},$\\
			& & & & $q^8\frac{(q^{4}-1)(q^3+1)}{q-1}$
			\\
			${}^2E_6(q)$   &  $P_4$  &  $\frac{(q^{12}-1)(q^6-q^3+1)(q^4+1)}{q-1}$ & $5$ & $1,q(q^5+1)(q^3+1)(q+1),q^{21},$\\
			& & & & ${q^6}(q^5+1)(q^4+q^2+1)$,\\
			& & & & $q^{11}(q^5+1)(q^3+1)(q+1)$\\
			$E_6(q)$   &  $P_1\text{ or }P_6$  &  $\frac{(q^{9}-1)(q^8+q^4+1)}{q-1}$ & $3$ & $1,q\frac{(q^{8}-1)(q^3+1)}{q-1},
			q^8\frac{(q^{5}-1)(q^4+1)}{q-1}$\\
			$E_7(q)$   &  $P_1$  &  $\frac{(q^{14}-1)(q^9+1)(q^5+1)}{q-1}$ & $4$ & $1,q\frac{(q^{9}-1)(q^8+q^4+1)}{q-1},$\\
			& & & & $q^{10}\frac{(q^9-1)(q^8+q^4+1)}{q-1},q^{27}$\\
			$E_8(q)$   &  $P_1$  &  $\frac{(q^{30}-1)(q^{12}+1)(q^{10}+1)(q^8+1)}{q-1}$ & $5$ & $1,q\frac{(q^{14}-1)(q^9+1)(q^5+1)}{q-1},q^{57},$ \\
			& & & & $q^{29}\frac{(q^{14}-1)(q^9+1)(q^5+1)}{q-1}$\\
			& & & & $q^{12}\frac{(q^{14}-1)(q^{12}+q^6+1)(q^9+q^4+1)}{q-1}$\\
			\hline
		\end{tabular}
		\caption{Information on the minimal degree permutation representations for exceptional groups in \eqref{AssEx}}
		\label{tab:minex}
	\end{table}
	
	\begin{proposition}\label{TwistedCases}
		Suppose that either $T$ is twisted; or $T$ is untwisted with $f>1$, where $q=p^f$ as in \eqref{AssEx}. Let $M$ be a point stabilizer for $T$ in a minimal degree permutation representation, so $M$ is as in Table \ref{tab:minex}.
		Then $T$ is not totally $2$-closed.
	\end{proposition}
	
	\begin{proof}
		By Lemma \ref{lem:ExDegLemma}, the group $M$ is a maximal parabolic subgroup of $T$. More precisely, we write $M=P_I$, where $I$ is a subset of the set of node-labels of the Dynkin diagram for $T$, and $P_I$ is the parabolic subgroup obtained by deleting the nodes with labels in $I$. In particular, in each case the group $G:=TN_{\Aut(T)}(M)$ is strictly larger than $T$. Indeed, if $T$ is twisted, then $T$ has a non-trivial graph automorphism which normalises $M$, while if $T$ is untwisted and $f>1$, then $T$ has a non-trivial field automorphism which normalises $M$. See \cite[Sections 2.5 and 2.6]{GLS}. In particular, we use the definition of graph, field, and graph-field automorphisms of $T$ given by \cite[Definition 2.5.13]{GLS}.
		
		Now, since $M$ is maximal in $T$, we have $M=N_T(M)$. Hence, $|G:N_{\Aut(T)}(M)|=|T:M|$, and $T^{[T:M]}$ is permutationally isomorphic to the permutation group $T^{[G:N_{\Aut(T)}(M)]}$. By inspection of Table \ref{tab:minex}, we see that in each case, any two subdegrees of  $T^{[G:N_{\Aut(T)}(M)]}$ involve different $q$-parts, and hence the subdegrees are pairwise distinct. Thus $T$ and $G$ have the same orbits  on ordered pairs of points from $[G:N_{\Aut(T)}(M)]$, and it follows from Theorem~\ref{thm:W} that $G\le T^{(2),[G:N_{\Aut(T)}(M)]}$. Therefore $T$ is not totally $2$-closed.
	\end{proof}
	
	
	In proving Proposition \ref{ExProp}, it is convenient to deal with the cases $E_6(2)$, $E_7(2)$ and $E_8(2)$ separately. We do this in the next proposition, where we use the following labelling of the Dynkin diagrams for these groups. This is consistent with the labelling in the computer algebra system \texttt{Magma}~\cite{MAGMA}.
	\begin{align}\label{dynkin}\nonumber 
		&\dynkin[label,label macro/.code={#1},
		edge length=.75cm]E6\text{ }E_6
		&\dynkin[label,label macro/.code={#1},
		edge length=.75cm]E7\text{ }E_7\\ 
		&\dynkin[label,label macro/.code={#1},
		edge length=.75cm]E8\text{ }E_8
	\end{align}
	
	\begin{proposition}\label{EmProp}
		The simple exceptional groups $E_6(2)$, $E_7(2)$, and $E_8(2)$, are not totally $2$-closed.
	\end{proposition}
	
	\begin{proof} Here $T=E_m(2)$ for some $m\in\{6,7,8\}$. Let $(B,N)$ be a $(B,N)$-pair for $T$. Since $T$ is defined over a field of order $2$, the group $N$ is isomorphic to the Weyl group $W$ of $T$, with respect to this $(B,N)$-pair. Let $\Phi$ be the associated root system for $T$, with fundamental roots $\Pi:=\{a_1,\hdots,a_m\}$. As mentioned above, we label the roots and Dynkin diagram for $T$ as in \texttt{Magma}~\cite{MAGMA}, see Figure~\ref{dynkin}.  
		Let $M:=P_{a_3}$ be the maximal parabolic subgroup of $T$ obtained by deleting the node labelled by $3$ in the Dynkin diagram. Then $M=U\rtimes L$, where $U$ is the unipotent radical of $M$ and $L$ is a Levi subgroup. 
		
		For each root $a_i$, we denote by $\Phi_i$ the set of roots $r\in\Phi$ with the property that the $a_i$-coordinate of $r$ is non-zero; and we set $\Phi_{\neg\, i} = \Phi\setminus\Phi_i$. For any subset $\Lambda$ of $\Phi$, we write $\Lambda^+$ for the set of positive roots in $\Lambda$. Finally, we write $U_r$ for the root subgroup of the unipotent radical $U$ associated with $r$.
		
		Then $U$ is the $2$-group $U =\langle U_r\text{ : }r\in\Phi^+_3\rangle$, and we may take $L$ to be the standard Levi subgroup $L:=\langle U_r\text{ : }r\in\Phi_{\neg\, 3}\rangle$. It is clear from the Dynkin diagram, Figure~\ref{dynkin}, that
		\begin{align}\label{Lnote}
			L=\langle U_{\pm a_1}\rangle\times\langle U_{\pm a_2},U_{\pm a_4},U_{\pm a_5},\hdots,U_{\pm a_m}\rangle \cong A_1(2)\times A_{m-2}(2).
		\end{align}
		Let $\pi_1:L\rightarrow A_1(2)$ and $\pi_2:L\rightarrow A_{m-2}(2)$ denote the canonical projections. Set $H=U\rtimes [L,L]$, so that $|M:H|=2$.
		
		Now a straightforward computation with \texttt{Magma}~\cite{MAGMA} shows that, 
		for each $w\in W$,
		\begin{align}\label{eq:Phiroot}
			(\Phi^+_3\cup\Phi_{\neg\, 3})^w \cap \Phi_{\neg\, 3}
		\end{align}
		contains a root $r_w\in \Phi$ with nonzero $a_1$-coordinate (and zero $a_3$-coordinate).
		Recall that $N\cong W$, as we noted above.  Choose $n\in N$, and suppose that $n$ corresponds to $w\in W$. Let $r:=r_w$ be a root in the set in \eqref{eq:Phiroot} (with nonzero $a_1$-coordinate and zero $a_3$-coordinate). 
		Then it follows from \eqref{eq:Phiroot} that  $M^n \cap L$ contains the root subgroup $U_{r}$. Since $U_r$ has non-zero $a_1$-coordinate, $\pi_1(U_r)$ contains a Sylow $2$-subgroup of $A_1(2)\cong \Sym(3)$. Thus, $\pi_1(M^n\cap L)$ contains a Sylow $2$-subgroup of $A_1(2)$, and hence $(M\cap M^n)H=M$ for all $n \in N$.
		
		Now, by Bruhat Normal Form \cite[Theorem 2.3.5]{GLS}, we have $T=BNB$, so the element 
		$t=xny$ for some $x,y\in B<M$ and $n\in N$. Then, noting that  $y\in N_T(H)$ since $|M:H|=2$, we have 
		\[
		(M\cap M^t)H=(M\cap M^{ny})H=((M\cap M^n)H)^y = M^y=M.
		\] 
		It now follows from Proposition \ref{MainCor1}(1) applied with $G, \Omega, \Sigma, M, \Delta$ and $L=L^{(2),\Sigma}$ being $T, [T:H], [T:M], M, [M:H]$ and $T^{[T:M]}$, that at least one of $T^{[T:M]}$ or $T^{[T:H]}$ is not $2$-closed. In either case $T$ is not totally $2$-closed.
	\end{proof}
	
	We are now ready to prove Proposition \ref{ExProp}.
	\begin{proof}[Proof of Proposition \ref{ExProp}]
		By Proposition \ref{TwistedCases}, we may assume that $T$ is untwisted and that $q=p$ is prime. 
		Thus, by Proposition \ref{lem:simple1}(a) and Proposition \ref{EmProp}, we may assume that $T\in\{F_4(p),E_6(p),E_7(p),E_8(p)\}$, and that $p$ is an odd prime.
		
		Let $S$ be a maximally split torus in $T$, with $(B,N)$ a $(B,N)$-pair for $T$, with respect to $S$. Then $B$ has the form $B=U\rtimes S$, where $U$ is a Sylow $p$-subgroup of $T$.
		If $T^{[T:B]}$ is not $2$-closed there is nothing further to prove for this group, so we may assume that $T^{[T:B]}$ is $2$-closed. Since $p$ is odd, the group $S$ has even order, and hence has a subgroup $R$  of index $2$. Let $H:=UR$. Then $|B:H|=2$. 
		
		For the permutation group $T^{[T:B]}$, the subgroup $B$ is the stabiliser of the `point' $B$. Each $B$-orbit in $[T:B]\setminus\{B\}$ is of the form $\{ Bah\mid h\in B\}$ for some $a\in T\setminus B$. 
		Now, by Bruhat Normal Form \cite[Theorem 2.3.5]{GLS}, we have $T=BNB$, so $a=xny$ for some $x,y\in B$ and $n\in N$, and it follows that the $B$-orbit containing $Ba$ also contains $Bay^{-1}=Bn$, and the stabiliser in $B$ of the point $Bn$ is $B\cap B^n$.  Since $N=N_T(S)$, we have $S\le B\cap B^n$, and it follows that $B=(B\cap B^n)H$. This holds for all $n\in N$.
		
		Finally, by again using the fact that $T=BNB$, and arguing in exactly the same way as in the final paragraph of the proof of Proposition \ref{EmProp}, we can see that $(B\cap B^t)H=B$ for all $t\in T$, and hence, after applying Proposition \ref{MainCor1}(1), that $T^{[T:H]}$ is not $2$-closed. Hence, $T$ is not totally $2$-closed. 
	\end{proof}
	
	\subsection{Proof of Theorem~\ref{thm:J}}
	By Proposition~\ref{lem:simple1}, together with our results in Sections  \ref{SporadicSection}, \ref{J4Section}, and \ref{sec:ExGroups}, on sporadic simple groups, and Proposition~\ref{ExProp} in this section for the exceptional groups,  we know 
	that each of the groups $\mathrm{J}_1, \mathrm{J}_3, \mathrm{J}_4, \mathrm{Th},\mathrm{Ly},\mathbb{M}$ is totally $2$-closed, and that these are the only totally $2$-closed simple groups. Thus Theorem~\ref{thm:J} is proved.

	\section{Direct products}\label{DirectProductsSection}
	
	In this section, our aim is to prove Theorem \ref{thm:classification}. For this reason, we assume throughout the section that $G$ is a non-trivial finite totally $2$-closed group with $F(G)=1$. By Theorem \ref{thm:issimple}, the group $G$ must have the form $G=T_1\times \hdots\times T_r$ where, for each $i\le r$, $T_i$ is a finite totally $2$-closed nonabelian simple group, and $T_i$ is not a section of $\prod_{i\neq j}T_j$.
	Moreover, by Theorem~\ref{thm:J},  each of the $T_i$ lies in the set 
	$$
	\mathcal{S}:=\{\mathrm{J}_1, \mathrm{J}_3, \mathrm{J}_4, \mathrm{Th}, \mathrm{Ly}, \mathbb{M}\}.
	$$ 
	By inspection of the sections of these sporadic simple groups, we see that there exist distinct $S, T\in\mathcal{S}$ such that $S$ is isomorphic to a section of $T$ if and only if $S=\mathrm{Th}$ and $T=\mathbb{M}$, see Remark~\ref{rem:section}. Thus not both of the groups $\mathrm{Th}$ and $\mathbb{M}$ lie in $\{T_i\text{ : }1\le i\le r\}$, and hence $r\leq 5$ and one of the conditions (3)(i) or (3)(ii) of Theorem \ref{thm:classification} holds. This proves that each of the assertions (1)--(3) of  Theorem \ref{thm:classification} holds.

	Thus to complete the proof of Theorem \ref{thm:classification} we must show that, 
	if $G = \prod_{i=1}^rT_i$ with the $T_i$ pairwise no-isomorphic, and if one of the conditions  (3)(i) or (3)(ii) of Theorem \ref{thm:classification} holds, 
	then $G$ is totally $2$-closed. Since none of the groups in $\mathcal{S}$ admits a nontrivial factorisation (Corollary~\ref{lem:W2}), it follows from Proposition \ref{PreTransRed} that it is sufficient to prove that $\varphi(G)$ is $2$-closed for each permutation representation $\varphi:G\to \Sym(\Omega)$ (faithful and non-faithful) with $|\Omega|>1$ and $\varphi(G)$ transitive.
	With this in mind, we introduce the following hypothesis and notation, which we will assume, for the remainder of the section: 
	
	\begin{hypothesis}\label{hyp:dir}
		{\rm
			(a)\ The group $G=T_1\times \hdots\times T_r$  with the $T_i$ pairwise non-isomorphic, and either 
			\begin{enumerate}[\upshape(i)]
				\item $T_i\in\{\mathrm{J}_1,\mathrm{J}_3,\mathrm{J}_4,\mathrm{Th},\mathrm{Ly}\}$ for each $i\le r$; or
				\item $T_i\in \{\mathrm{J}_1,\mathrm{J}_3,\mathrm{J}_4,\mathrm{Ly},\mathbb{M}\}$ for each $i\le r$.
			\end{enumerate}
			(b) We extend the notation for direct products introduced in Hypothesis \ref{not1}. Namely, for a direct product $G_1\times\hdots\times G_s$ of arbitrary finite groups $G_i$, and each subgroup $L_i$ of $G_i$, we write 
			\[
			\widehat{L_i}=\{(x_1,\hdots,x_s)\text{ : }x_i\in L_i\text{, and } x_j=1\text{ for }j\neq i\}\le G_1\times\hdots\times G_s.
			\]  
			\noindent 
			(c) Seeking a contradiction, we assume that $G$ is \textbf{\underline{not}} totally $2$-closed, and that the number $r$ of simple direct summands is minimal with respect to this property. Note that $r>1$ by Theorem~\ref{thm:J}, and by 
			Proposition \ref{PreTransRed}, $G$ has a proper subgroup $H$ such that $G^{[G:H]}$ is not $2$-closed. Without loss of generality we may assume that the proper subgroup $H$ is maximal by inclusion such that $G^{[G:H]}$ is not $2$-closed. By \cite[Theorem~5.12 and Example 5.13]{Wielandt}, it follows that $H\ne 1$. Further, $H$ is core-free in $G$ (since otherwise $G/\Core_G(H)\cong G^{[G:H]}$  is not totally $2$-closed and is a direct product of less than $r$ of the simple groups in (i) or (ii), contradicting the minimality of $r$).

			\vspace{0.25cm}
			\noindent
			(d) Next, write $\mu_i:G\rightarrow T_i$ for the natural coordinate projection, for $i\leq r$. Note that $\mu_i(\widehat{T}_i\cap H)$ is normal in $\mu_i(H)$, but to avoid cumbersome notation, we will usually just say that $\widehat{T}_i\cap H$ is normal in $\mu_i(H)$.  
			
			\vspace{0.25cm}
			\noindent 
			(e) Our final piece of notation is defined with the minimality of $r$ in mind: we set $G_1:=T_1$, and $G_2:=T_2\times\hdots\times T_r$, so that both $G_1$ and $G_2$ are finite totally $2$-closed groups (by the minimality of $r$). 
			We write the coordinate projections here as $\pi_1:G\rightarrow G_1$ and $\pi_2:G\rightarrow G_2$ (so in particular $\pi_1=\mu_1$). We also set $M_i:=\pi_i(H)$ and $H_i=\widehat{G}_i\cap H$, for $i\in\{1,2\}$, where $\widehat{G}_i$ is as defined in (b) relative to the direct product $G_1\times G_2$. Note that $H$ is a subdirect subgroup of $M:=M_1\times M_2\le G$
			and $\pi_i(H_i)\unlhd M_i$ for each $i$. As in part (d), we often omit the $\pi_i$, and just say that $H_i\unlhd M_i$. Since $H$ is a subdirect subgroup of $M=M_1\times M_2$, it is important to note that $M_1/H_1\cong M_2/H_2$ (Goursat's Lemma, see \cite[Theorem 4.8(i)]{PS}).
		}
	\end{hypothesis}

	\begin{remark}\label{rem1}{\rm
			The final piece of notation in Hypothesis \ref{hyp:dir}  is introduced for convenience only: there will be nothing special about the labelling of the simple direct factors 
			of $G$, and we will sometimes vary the direct factor which is chosen to be $T_1$ in our analysis. We warn the reader however, that the labelling of the direct factors of $G$ changes  in this way, then the group $M$ defined in Hypothesis~\ref{hyp:dir}(e) may also change.
		}
	\end{remark}
	
	In order to prove that each group $G$ satisfying Hypothesis \ref{hyp:dir} is totally $2$-closed, we use the following strategy:
	
	\begin{Strategy}\label{Strategy2}
		Let $G, T_i, H$, etc be as in Hypothesis~\ref{hyp:dir}. We follow the four steps below. 
		
		\medskip\noindent
		\emph{Step $1$:}\quad We prove (Lemma~\ref{ProductAction}) that the subgroup $M=M_1\times M_2$ in Hypothesis \ref{hyp:dir}(e) is a proper overgroup of $H$, that  $M$ is core-free in $G$, and that $G^{[G:M]}$ is $2$-closed. This enables us to use the tools developed in Section \ref{ImprimitiveSection}, (see Notation \ref{not2}). 
		
		\medskip\noindent
		\emph{Step $2$:}\quad   The tools developed in Section \ref{ImprimitiveSection} also require information concerning the transitive permutation group $M^{[M:H]}$ with point stabiliser $H^{[M:H]}$. Note that $H$ is a subdirect subgroup of $M=M_1\times M_2$. Step 2 comprises two results (Lemmas~\ref{lem:subdirect} and~\ref{RegularLemma}) concerning subdirect subgroups of a direct product of finite groups. 
		
		\medskip\noindent
		\emph{Step $3$:}\quad   We prove (Lemma~\ref{lem:Step2}) that, for some $i\le r$, the permutation group $T_i^{[T_i:\mu_i(H)]}$ has base size greater than $2$, and so, by Proposition \ref{Base2Lemma}, this pair ($T_i,\mu_i(H)$) lies in Table \ref{tab:basesize}.

		\medskip\noindent
		\emph{Step $4$:}\quad  For each pair ($T_i,\mu_i(H)$) in Table \ref{tab:basesize}, we examine the structure of $X:=G^{(2), [G:H]}$ using Proposition~\ref{MainCor1}(2), seeking a contradiction. We use the same notation as that introduced in Notation~\ref{not2} with some small changes given in Notation~\ref{not:dirmain}. 
	\end{Strategy}
	
	We begin with Step 1.
	
	\begin{lemma}\label{ProductAction}
		Let $G, r, T_i, H, M$ be as in Hypothesis $\ref{hyp:dir}$. Then,
		\begin{enumerate}[\upshape(i)]
			\item $\mu_i(H)\neq T_i$ for any $i\le r$, $M$ is core-free in $G$, $G^{[G:M]}$ is $2$-closed, and $H<M<G$; and
			
			\item for each $i\le r$, and each composition factor $S$ of 
			$\mu_i(H)/(\widehat{T}_i\cap H)$, there exists $j\neq i$ such that $\mu_j(H)$  has a composition factor isomorphic to $S$.
		\end{enumerate}
	\end{lemma}
	
	\begin{proof}
		(i) Note that Hypothesis~\ref{not1} also holds for $G$ (with projection maps $G\to T_i$ now called $\mu_i$). Hence, since $H$ is core-free in $G$ by Hypothesis~\ref{hyp:dir}(c), it follows from Lemma \ref{lem:dp1} that 
		$\mu_i(H)\neq T_i$ for any $i\le r$. Then, since $M\le \mu_1(H)\times\hdots\times \mu_r(H)$, the group $M$  does not contain $\hat{T}_i$, for any $i\le r$, and so $M$ is core-free in $G$. Also (see Hypothesis~\ref{hyp:dir}(e)), $M= M_1\times M_2 < G_1\times G_2$, and by the minimality of $r$, $G_{i}^{[G_i:M_i]}$ is $2$-closed for  $i\in\{1,2\}$. Hence $G^{[G:M]}$ is $2$-closed, by \cite[Theorem 5.1]{CGJKKM}. In particular, since by assumption $G^{[G:H]}$ is not $2$-closed, we have $H<M<G$. This proves part (i).

		\vspace{0.25cm}
		\noindent (ii) In order to prove part (ii) we will prove the following claim: 
		
		\smallskip
		\emph{Claim: Let $S$ be a finite simple group, and let $X_1,\hdots,X_t$ be finite 
			groups ($t\geq1$) such that no $X_i$ has a section isomorphic to $S$. If $X$ is a subdirect subgroup of $X_1\times\hdots\times X_t$, then $X$ has no section isomorphic to $S$. }
		
		\smallskip
		First we see that part (ii) follows from the claim. Indeed, assume that the claim holds, and that $S$ is a composition factor of $\mu_i(H)/(\widehat{T}_i\cap H)$, for some $i\le r$. Without loss of generality assume that $i=1$, so $S$ is a composition factor of $\mu_1(H)/(\widehat{T}_1\cap H)= M_1/H_1$. Since $H$ is a subdirect subgroup of $M=M_1\times M_2$ (Hypothesis~\ref{hyp:dir}(e)), we have $M_1/H_1\cong M_2/H_2$, and hence $S$ is a composition factor of $M_2$. Then since $M_2=\pi_2(H)$ is a subdirect subgroup of $\mu_2(H)\times\hdots\times \mu_r(H)$, it follows from the claim that $\mu_j(H)$ has a composition factor isomorphic to $S$ for some $j\geq 2$.  
		
		\noindent
		\emph{Proof of the Claim:} The proof is by induction on $t$. The case $t=1$ is 
		trivially true, so assume that $t>1$ and that the claim holds for direct products 
		of less than $t$ groups. Suppose that $X$ is a subdirect subgroup of $X_1\times\hdots
		\times X_t$,  that no $X_i$ has a section isomorphic to $S$, but that $X$ has a 
		section isomorphic to $S$, say $S\cong Y/N$ with 
		$N\unlhd Y\leq X$. We will derive a contradiction. For each $i\leq r$, let $\nu_i:X\to X_i$ be the natural projection map and $Y_i=\nu_i(Y)$, so $Y$ is a subdirect subgroup of $Z:=Y_1\times\dots\times Y_t$ and no $Y_i$ has a section isomorphic to $S$.
		Further let $\nu_i':Z\rightarrow \prod_{j\neq i}Y_j$ be the natural projection map, for each $i$, and note that $\nu_1'(Y)$ is a subdirect subgroup of $Y_2\times\hdots\times Y_t$. Hence, if $\nu_1'(Y)$ has a section isomorphic to $S$, then by induction $Y_i$ has a section isomorphic to $S$, for some $i\geq 2$, which is a contradiction. Hence $\nu_1'(Y)$ has no section isomorphic to $S$. Since $\nu_1'(Y)/\nu_1'(N)$ is isomorphic to a quotient of the simple group $Y/N\cong S$, this implies that $\nu_1'(N)=\nu_1'(Y)$. Now the kernel of $\nu_1'|_Y$ is the group $K:=Y\cap \widehat{Y}_1$, and so $\nu_1'(N)=\nu_1'(Y)$ implies that $Y=NK$, whence $K/(N\cap K)\cong NK/N = Y/N\cong S$. Thus $\widehat{Y}_1\cong Y_1$ has a section isomorphic to $S$, which is a contradiction. This completes the inductive proof of the claim, and hence completes the proof of the lemma.
	\end{proof}
	
	We now move to Step 2, and prove  two lemmas concerning subdirect subgroups of direct products.
	
	\begin{lemma}\label{lem:subdirect}
		Let $X_1,\hdots,X_t$ be finite groups, and let $X$ be a subdirect subgroup of $X_1\times\hdots\times X_t$. Then the following assertions hold.
		\begin{enumerate}[\upshape(i)]
			\item If $t=2$, and $\widehat{X}_i\leq X$ for some $i\in\{1,2\}$, then $X=X_1\times X_2$.
			\item Let $p$ be a prime, and suppose that each group $X_i$ is a transitive permutation group of $p$-power degree. Then the Fitting subgroup $F(X)$ is a (possibly trivial) $p$-group.
		\end{enumerate}
	\end{lemma}
	\begin{proof}  
		(i) This follows immediately from \cite[Lemma 4.7]{PS}.
		
		\vspace{0.25cm} 
		(ii) For each $i\leq t$, assume that  $X_i$ is a transitive permutation group on a set $\Omega_i$ with $|\Omega_i|$ a power of $p$, and let $\nu_i:X\rightarrow X_i$ be the natural projection map. It is sufficient to show that $\nu_i(F(X))$ is a $p$-group for each $i$. Note that,  for each $i$, $\nu_i(F(X))$ is a nilpotent normal subgroup of $\nu_i(X)=X_i$ (since $X$ is subdirect), and hence $\nu_i(F(X))\leq  F(X_i)$. If $\nu_i(F(X))=1$ there is nothing to prove so assume that $\nu_i(F(X))\ne 1$ and let $q$ be a prime dividing $|\nu_i(F(X))|$. Then the unique (nontrivial) Sylow $q$-subgroup $Q$ of $\nu_i(F(X))$ is characteristic in $\nu_i(F(X))$ and hence normal in $X_i$. Since $X_i$ is transitive on 
		$\Omega_i$, all the orbits of its normal subgroup $Q$ have the same size, say $m$. Thus 
		$m$ divides $|\Omega_i|$ which is a power of $p$, and also $m$ divides $|Q|$ which is a power of $q$. It follows that $p= q$, and hence that $\nu_i(F(X))$ is a $p$-group.
		Thus part (ii) is proved. 
	\end{proof}

	\begin{lemma}\label{RegularLemma}
		Let $E$ be a finite group, let $\theta\in\Aut(E)$, and let 
		\[
		J:=\{(e,e^{\theta})\text{ : }e\in E\},
		\]
		a full diagonal subgroup of $X:=E\times E$. Then:
		\begin{enumerate}[\upshape(i)]
			\item $\core_X(J)=\{(z,z^{\theta})\text{ : }z\in Z(E)\}=Z(X)\cap J$.
			\item The group $Y:=X^{[X:J]}$ is a central product $E_1\circ E_2$, where $E_i\cong E$ for each $i$, and $E_1\cap E_2\cong Z(E)$. In particular $Y=X/(Z(X)\cap J)$.
			\item Writing $\ol{U}$ for the image of a subgroup $U\leq Y$ under the natural projection map $Y\to Y/(E_1\cap E_2)$, we have $\ol{Y}\cong \ol{E_1}\times\ol{E_2}$, and $\ol{J}$ is a subdirect subgroup such that 
			\[
			\ol{J}\cap \widehat{\ol{E_1}} =\ol{J}\cap \widehat{\ol{E_2}}=1.
			\]
		\end{enumerate}
	\end{lemma}
	
	\begin{proof}
		(i) and (ii): Since part (ii) follows immediately from (i), we only need to prove (i).
		Let $x\in \core_X(J)=\cap_{y\in X}J^y$. Then $x=(e,e^\theta)$ for some $e\in E$ (since $x\in J$), and also $x^y\in J$ for all $y=(e_1,e_2)\in X$. In particular, taking $y=(1,e_2)$, we have $x^y=(e,e^{\theta e_2})\in J$, so $e^{\theta e_2} = e^\theta$, and this holds for all $e_2\in E$. Hence $e^\theta\in Z(E)$, and since $\theta$ is an automorphism, it follows that $e\in Z(E)$. Thus  $\core_X(J)\subseteq \{(z,z^{\theta})\text{ : }z\in Z(E)\}$. The reverse inclusion clearly holds, so equality holds and hence part (i), and also part (ii), are proved.
		
		\vspace{0.25cm}
		\noindent (iii): The fact that $\ol{Y}\cong \ol{E_1}\times\ol{E_2}$ follows from part (ii), and by definition, $\ol{J}$ projects onto each direct factor $\ol{E_i}$. Noting that $Z(X)=Z(E)\times Z(E)$, it follows that $E_1\cap E_2 = Z(X)/(Z(X)\cap J)$, and hence $\ol{Y}=Y/(E_1\cap E_2)\cong X/Z(X)$. Under this isomorphism $\widehat{\ol{E_1}}$ corresponds to the subgroup of cosets $Z(X)x$ with $x$ of the form $x=(e,z)$   for some $e\in E, z\in Z(E)$. If such a coset $Z(X)x$ corresponds to an element of $\ol{J}$, then $x=x'z'$ with $x'\in J, z'\in Z(X)$, so $x=(e,z)=(e_1, e_1^\theta)(z_1, z_2)$ for some $e_1\in E, z_1, z_2\in Z(E)$. This implies firstly that $e_1^\theta=z z_2^{-1}\in Z(E)$ whence $e_1\in Z(E)$; and secondly that $e=e_1z_1\in Z(E)$. It follows that the coset $Z(X)x$ is trivial, and hence that $\ol{J}\cap \widehat{\ol{E_1}}=1$. The proof that 
		$\ol{J}\cap \widehat{\ol{E_2}}=1$ is similar.
	\end{proof}

	Now we proceed with Step 3, examining the base sizes of the simple direct factors $T_i$ in their coset actions on $[T_i,\mu_i(H)]$. 
	
	\begin{lemma}\label{lem:Step2}
		Let $G, r, T_i, \mu_i, H, M$ be as in Hypothesis $\ref{hyp:dir}$. 
		Then there exists $i\leq r$ such that the permutation  group $T_i^{[T_i:\mu_i(H)]}$ has base size greater than $2$, and in particular  $(T_i,\mu_i(H))$ is one of the pairs $(T, M)$ in Table $\ref{tab:basesize}$.
	\end{lemma}
	
	\begin{proof}
		By  Lemma \ref{ProductAction}(i),  $H$ is a proper subgroup of $M$, and hence $H$ is a proper subgroup of $N:=\mu_1(H)\times \hdots\times\mu_r(H)$ since $M\le N$ (Hypothesis~\ref{hyp:dir}(e)). Since $H$ is core-free in $G$ (Hypothesis~\ref{hyp:dir}(c)), it follows from Lemma \ref{lem:dp1} that $\mu_i(H)$ is a proper subgroup of $T_i$ for all $i\le r$. Thus, $N$ is core-free in $G$. If $T_i^{[T_i:\mu_i(H)]}$ has base size at most $2$ for all $i$, then $G^{[G:N]}$ has base size at most $2$, and it follows from Proposition \ref{MainCor1}(3) that $G^{[G:H]}$ is $2$-closed, contradicting Hypothesis \ref{hyp:dir}(c). Thus, for some $i\leq r$, $T_i^{[T_i:\mu_i(H)]}$ has base size greater than $2$, and hence, by Proposition~\ref{Base2Lemma}(a),  $(T_i,\mu_i(H))$ is one of the pairs $(T, M)$ in Table \ref{tab:basesize}.
	\end{proof}

	We now move on to Step 4 of Strategy~\ref{Strategy2}. We examine each possibility for ($T_i,\mu_i(H)$) from Table~\ref{tab:basesize}. First we deal with the unique pair in  Table \ref{tab:basesize} involving the Monster, namely $(\mathbb{M}, 2.\mathbb{B})$. 
	
	\begin{lemma}\label{NoPairs} For $T_i, \mu_i, H$ as in Hypothesis \ref{hyp:dir}, no pair $(T_i,\mu_i(H))$ is equal to $(\mathbb{M}, 2.\mathbb{B})$. 
	\end{lemma}
	\begin{proof}
		Suppose to the contrary that $(T_i,\mu_i(H))=(\mathbb{M}, 2.\mathbb{B})$, for some $i$. As discussed in Remark~\ref{rem1}, we may assume that $i=1$, so $M_1=2.\mathbb{B}$, and by Hypothesis~\ref{hyp:dir},  $T_j\in\{\mathrm{J}_1,\mathrm{J}_3,\mathrm{J}_4,\mathrm{Ly}\}$ for each $j\geq 2$. Now $H$ is a subdirect subgroup of $M=M_1\times M_2$, and by  Lemma \ref{ProductAction}(i), $H$ is a proper subgroup of $M$. Thus, by Lemma \ref{lem:subdirect}(i), $H$ does not contain $\widehat{M}_1$, so  $H_1=\widehat{T}_1\cap H= \widehat{M}_1\cap H$ is a proper normal subgroup of $M_1=2.\mathbb{B}$. It follows that $\mathbb{B}$ is a composition factor of $M_1/H_1$, but this contradicts Lemma \ref{ProductAction}(ii) since, for each $j\geq2$, $T_j\in\{\mathrm{J}_1,\mathrm{J}_3,\mathrm{J}_4,\mathrm{Ly}\}$ and none of these simple groups has a section isomorphic to $\mathbb{B}$.
	\end{proof}

	To deal with the other pairs ($T_i,\mu_i(H)$) from Table~\ref{tab:basesize},
	our main tool is Proposition \ref{MainCor1}(2), and we extend the notation used there slightly as follows.
	
	\begin{notation}\label{not:dirmain}
		\rm{
			We continue to adopt the notation and hypotheses set out in Hypothesis \ref{hyp:dir}, and in order to apply Proposition \ref{MainCor1}(2),  we recall and slightly extend the notation from Notation~\ref{not2} (notably parts  (2), (3), and (5) -- (10)):  
			
			\vspace{0.25cm}
			\noindent (a) $\Omega=[G:H]$, $\Sigma=[G:M]$ of size $s=|\Sigma|$, $\Delta=[M:H]$, $L=G^{\Sigma}$ (which equals $L^{(2),\Sigma}$ by Lemma~\ref{ProductAction}), $R=M^\Delta$, $Y=X_\Delta^\Delta$, and $G\le W=R\wr L$ acting on $\Omega$. For a subgroup $A\leq R$ we write $B_A$ for the corresponding base group $A^s$ of the base group $B=R^s$ of $W$. Note that $M$ is core-free in $G$ (by Lemma~\ref{ProductAction}(i)), so $L=G^\Sigma\cong G$.  
			
			\vspace{0.25cm}
			\noindent (b) Since $M=M_1\times M_2$ and $G=G_1\times G_2$ (Hypothesis~\ref{hyp:dir}(e)), the set $\Sigma=\Sigma_1\times \Sigma_2$ and $s=s_1s_2$, where $\Sigma_i=[G_i:M_i]$ of size $s_i$, for $i=1, 2$.  For this reason, and as we will see below, it is more natural to label the $G$-invariant partition $\Sigma$ of $\Omega$ in Notation \ref{not2}(2) as $\Sigma=\{\Delta_{(i,j)}\text{ : }1\le i\le s_1\text{; }1\le j\le s_2\}$, and we choose  $\Delta:=\Delta_{(1,1)}$, and $\omega\in\Delta$, (see Notation~\ref{not2}~(2) and~(3)).
			
			\vspace{0.25cm}
			\noindent (c) In Notation~\ref{not2}~(5) and (6),  we identify $\Sigma$ with the set $\{(i,j)\text{ : }1\le i\le s_1\text{; }1\le j\le s_2\}$, and we assume that $G\le W=R\wr L$ is acting on $\Omega=\Delta\times \{(i,j)\text{ : }1\le i\le s_1\text{; }1\le j\le s_2\}$; where we identify $\Delta_{(i,j)}$ with $\Delta\times\{(i,j)\}$. The elements $(x_{(1,1)},\hdots,x_{(s_1,s_2)})\in B_R$ and $\sigma=(\sigma_1,\sigma_2)\in L\cong G_1^{[G_1:M_1]}\times G_2^{[G_2:M_2]}$ (with $\sigma_i\in G_i^{[G_i:M_i]}$) act as follows on a point $(\delta,(i,j))\in\Omega$:  
			\begin{align*}
				(\delta,(i,j))^{(x_{(1,1)},\hdots,x_{(s_1,s_2)})}=(\delta^{x_{(i,j)}},(i,j)) 
				\text{ and }(\delta,(i,j))^{\sigma}= (\delta,(i^{\sigma_1},j^{\sigma_2})).
			\end{align*}
			
			\vspace{0.25cm}
			\noindent (d) In Notation~\ref{not2}~(7) -- (9), we write 
			\[
			R_{(i,j)}=\{(x_{(1,1)},\hdots,x_{(s_1,s_2)})\text{ : }x_{(k,\ell)}=1\text{ if }(k,\ell)\neq(i,j)\};
			\]
			for $A\leq \Sym(\Delta)$, we replace $W_i$ by $W_{(i,j)}=B_A\rtimes L_{(i,j)}$; the map $\rho_i$ by 
			\[
			\rho_{(i,j)}:W_{(i,j)}\to A \text{ given by }  \rho_{(i,j)}: x\to x_{(i,j)}
			\] 
			for $x=(x_{(1,1)},\hdots,x_{(s_1,s_2)})\in B_A$; the map $\rho_{i,j}$ by 
			\[
			\rho_{(i,j), (k,\ell)}:W_{(i,j)}\cap W_{(k,\ell)}\to A\times A \text{ given by }  \rho_{(i,j),(k,\ell)}: x\to (x_{(i,j)}, x_{(k,\ell)})
			\] 
			for distinct $(i,j), (k,\ell)$; and in applying this to notation to $W=R\wr L$, we write $M_{(i,j)}= W_{(i,j)}\cap G$, so $\rho_{(i,j)}(M_{(i,j)})=M_{(i,j)}^{\Delta_{(i,j)}}=R_{(i,j)}\cong R$, etc.  Moreover if, for $k\in\{1,2\}$,  
			$M^{(1)}_k$, $\hdots$, $M^{(s_k)}_k$ are the point stabilizers in the transitive permutation group $G_k^{[G_k:M_k]}$, then, relabelling if necessary, the stabiliser $M_{(i,j)}=M^{(i)}_1\times M^{(j)}_2$, for $1\le i\le s_1$, and $1\le j\le s_2$.}  
		
		\vspace{0.25cm}
		\noindent (e) By Proposition \ref{MainCor1}(2),  $X=N\rtimes G$ with $N$ a subdirect subgroup of a base group $B_A=A^s$ for some nontrivial $A\unlhd\, Y$, and the $A$-orbits in $\Delta$ have a constant length $a=a(G, M, H)$ such that $a$ divides $|L_{(i,j), (k,\ell)}|$, for all distinct   $\Delta_{(i,j)}, \Delta_{(k,\ell)}\in \Sigma$.

	\end{notation}

	Our strategy for proceeding is to prove that the parameter $a$ in Notation~\ref{not:dirmain}(e) divides a product of $r$ positive integers related to the simple groups $T_1,\dots,T_r$ (see Lemma~\ref{lem:factora}). We then prove (see Lemma \ref{lem:divisora}) that $a$ has the form $a=d_1d_2$, where the $d_i$ are orders of certain normal sections of $M_1/H_1$. Finally, in the proof of Theorem \ref{thm:classification}, we prove that in most cases, $d_1d_2$ cannot divide $a$, yielding a contradiction in those cases.

	\begin{lemma}\label{lem:factora}
		Assume that the notation and conditions from Notation~$\ref{not:dirmain}$ hold, noting in particular that $L=L^{(2),\Sigma}$ and $X\ne G$. In addition extend the notation $g(T,M)$ from Proposition $\ref{Base2Lemma}$ to all pairs $(T,M)$ with $T$ a finite simple group and $M< T$, by defining $g(T,M):=1$ if $(T,M)$ is not in Table $\ref{tab:basesize}$. Then the parameter $a=a(G,M,H)$ divides $\prod_{i=1}^rg(T_i,\mu_i(H))$.
	\end{lemma}
	
	\begin{proof}
		By Notation~$\ref{not:dirmain}$(e), $a$ divides $|L_{(i,j),(k,\ell)}|$, for all $i,k\le s_1$ and all $j,\ell\le s_2$ with $(i,j)\ne (k,\ell)$. Moreover since $G$ is faithful on $\Sigma$ (see Notation~$\ref{not:dirmain}$(a)) we have $|L_{(i,j),(k,\ell)}|=|M_{(i,j)}\cap M_{(k,l)}|$, and by Notation~$\ref{not:dirmain}$(d), $|M_{(i,j)}\cap M_{(k,l)}|= |M^{(i)}_1\cap M^{(k)}_1|\cdot|M^{(j)}_2\cap M^{(\ell)}_2|$. Thus, for fixed  $j$ and $\ell$,
		we see that $a$ divides $|M^{(i)}_1\cap M^{(k)}_1|\cdot|M^{(j)}_2\cap M^{(\ell)}_2|$ for all $i,k\le s_1$, and hence $a$ divides $d_1|M^{(j)}_2\cap M^{(\ell)}_2|$, where $d_1=\gcd\{|M^{(i)}_1\cap M^{(k)}_1| \,|\, i\ne k\}$. 
		
		By Hypothesis~\ref{hyp:dir}(e), $\mu_1(H)=M_1$, and by Lemma~\ref{ProductAction}(i), $\mu_1(H)\ne T_1$, and hence $M_1$ is a proper subgroup of $T_1$. If $(T_1, M_1)$ is not a pair in Table~\ref{tab:basesize}, then we have defined $g(T_1,M_1)=1$; moreover  by Proposition~\ref{Base2Lemma}(a), the permutation group $T_1^{[T_1,\mu_1(H)]}$ has base size at most $2$ and hence $M^{(i)}_1\cap M^{(k)}_1=1$ for some $i,k$. Thus $d_1=1=g(T_1,\mu_1(H))$ in this case. On the other hand, if  $(T_1, M_1)$ is a pair in Table~\ref{tab:basesize}, then $d_1$ is the quantity $d(T_1, \mu_1(H))$ in Proposition~\ref{Base2Lemma}(b), and by that result, $d_1$ divides the quantity $g(T_1,\mu_1(H))$. Thus in either case $a$ divides 
		$g(T_1,\mu_1(H))\cdot  |M^{(j)}_2\cap M^{(\ell)}_2|$, for all $j, \ell\leq s_2$.
		
		We complete the proof by induction on $r$. If $r=2$, then $\mu_2(H)=M_2$ and the argument of the previous paragraph shows that  $a$ divides $g(T_1,\mu_1(H))\cdot g(T_2,M_2)$, proving the lemma in this case. Assume now that $r\geq 3$, and that the result holds for direct products of fewer than $r$ simple groups. Set $G_1':=T_2$, $G_2':=T_3\times\hdots\times T_r$, and $G':=G_1'\times G_2'$. Also, let $\pi_i':G'\rightarrow G_i'$ be the coordinate projections and $M_i':=\pi_i'(M_2)$, for $i=1, 2$. Note that, since $\pi_2(H)=M_2$, it follows from the definition of $\pi_2'$ that $M_1'=\mu_2(H)$ and this is a proper subgroup of $G_1'=T_2$ by Lemma~\ref{ProductAction}(a). 
		Set $M':=M_1'\times M_2'$. Then $M^{(j)}_2\cap M^{(\ell)}_2$ is a subgroup of $G'=G_1'\times G_2'$ contained in $M'$. 
		Now, the greatest common divisor of the orders $|M^{(j)}_2\cap M^{(\ell)}_2|$, for all $j,\ell\le s_2$ with $j\neq \ell$, divides the greatest common divisor of the orders $|(M')^{x}\cap (M')^{y}|$, for all $x, y\in G'$. 
		The argument in the first two paragraphs above shows that $|M^{(j)}_2\cap M^{(\ell)}_2|$ divides $g(T_2,M_1')\cdot |(M_2')^{x}\cap (M_2')^y|$ for every fixed pair of elements 
		$x,y\in G_2'$. 
		Since $M_1'=\mu_2(H)<T_2$, we have that $a$ divides $g(T_1,\mu_1(H))\cdot g(T_2,\mu_2(H))\cdot |(M_2')^{x}\cap (M_2')^y|$, for all $x, y\in G_2'$, and by induction we conclude that $a$ divides $g(T_1,\mu_1(H))\cdot g(T_2,\mu_2(H))\cdots g(T_r,\mu_r(H))$, completing the proof.  
	\end{proof}

	Next we prove Lemma~\ref{lem:divisora} which provides a factorisation of the parameter $a=a(G, M, H)$. 
	
	\begin{lemma}\label{lem:divisora}
		Assume that the notation and conditions from Notation~$\ref{not:dirmain}$ hold, and recall that $X=G^{(2),\Omega}$. Then 
		
		\begin{enumerate}[\upshape(i)]
			\item there exist normal subgroups $V$ of the centre $Z(M_1/H_1)$, and  $U$ 
			of $(M_1/H_1)/Z(M_1/H_1)$, such that $a=a(G,M,H)=|U|\cdot|V|$; and
			\item if $a$ is a power of a prime $p$, then $F(X)$ is a $p$-group.
		\end{enumerate}
	\end{lemma}
	
	Recalling from Hypothesis~\ref{hyp:dir}(e) that $M_1/H_1\cong M_2/H_2$, we see that in the decomposition  $a=|U|\cdot|V|$ in part (i), each of $U, V$ is (isomorphic to) a normal section of $M_1/H_1\cong M_2/H_2$.

	\begin{proof}
		\quad (i) By Notation~\ref{not:dirmain}(e), $A\unlhd Y=X_\Delta^{\Delta}$, so the set 
		of $A$-orbits in $\Delta$ is a $Y$-invariant partition of $\Delta$ with parts of size $a>1$, and since $M^\Delta=R\leq Y$ (Notation~\ref{not:dirmain}(a)) this partition is also $M^\Delta$-invariant. By Notation~\ref{not:dirmain}(a), $H$ is the stabiliser in $M$ of a point $\omega\in\Delta$, and so the setwise stabiliser $K$ in $M$ of the $A$-orbit $\omega^A$ satisfies $H<K\leq M$ and $|K:H|=a$. 
		
		Recall the definitions of the $M_i, H_i$ in Hypothesis~\ref{hyp:dir}(e) and in particular that $M_1/H_1\cong M_2/H_2\cong E$, say. Now $S:=H_1\times H_2\unlhd M=M_1\times M_2$ and $S<H<M$, so the permutation groups induced on $\Delta$ by $M$ and by $\overline{M}:= M/S\cong 
		(M_1/H_1)\times (M_2/H_2)\cong E\times E$ are the same. In the latter group, $\overline{H}=H/S$ is the stabiliser of $\omega$ and $\overline{K}=K/S$ is the setwise stabiliser of $\omega^A$. In particular $a=|\overline{K}:\overline{H}|$. Since $H$ is a subdirect subgroup of $M=M_1\times M_2$ it follows that $\overline{H}$  is a subdirect subgroup of 
		$\overline{M}= E\times E$, and moreover from the definition of $S$ it follows that 
		$\overline{H}$ meets each direct factor of $\overline{M}= E\times E$ trivially, so $\overline{H}$ is a full diagonal subgroup of $E\times E$ and we have $\overline{H}=\{(z,z^\theta)| z\in E\}$ for some $\theta\in\Aut(E)$ and $\overline{H}\cap(E\times 1)=1$. 
		
		For $i=1,2$, let $\nu_i:\overline{M}\to E$ denote the natural projection map onto the $i^{th}$ direct factor. Since $\overline{H}$ is a subdirect subgroup, 
		$\nu_i(\overline{H})=E$ and hence also $\nu_i(\overline{K})=E$, for each $i$. Let 
		$\widehat{K_1}:= \overline{K}\cap \Ker(\nu_2)$, and note that $\Ker(\nu_2)=E\times 1$. Then $\widehat{K_1}\unlhd \overline{K}$ 
		and $\overline{K}/\widehat{K_1}\cong \nu_2(\overline{K})= E$.  Also $\widehat{K_1}\cap \overline{H}=\overline{H}\cap 
		\Ker(\nu_2)=1$. It follows that $\overline{K}=\widehat{K_1}\rtimes \overline{H}$, 
		and hence $|\widehat{K_1}|=|\overline{K}:\overline{H}|=|K:H|=a$.  Now $E\cong M_1/H_1$ has 
		normal subgroups $\widehat{K_1}$ and $Z(E)$, and hence also $\widehat{K_1}\cap Z(E)
		\unlhd E$. Thus $a=d_1d_2$, where 
		$d_1= |\widehat{K_1}/(\widehat{K_1}\cap Z(E))|$ and $d_2=|\widehat{K_1}\cap Z(E)|$. 
		Part (i) follows on noting that $V:= \widehat{K_1}\cap Z(E)$ is a normal subgroup of 
		$Z(E)$, and $U:=(\widehat{K_1}Z(E))/Z(E)\cong \widehat{K_1}/(\widehat{K_1}\cap Z(E))$ is a normal subgroup of $E/Z(E)$.
		
		\vspace{0.25cm}
		\noindent (ii) Suppose now that $a$ is a power of a prime $p$. Since $X=N\rtimes G$ with $N$ a subdirect subgroup of the base group $B_A=A^s$ (Notation~\ref{not:dirmain}(e)), and since $G$ is a direct product of nonabelian simple groups, it follows that $F(X)=F(N)\leq F(A)^s$. 
		Now $A$ is a permutation group on $\Delta$ with all orbits of length $a$ (a power of $p$), and hence $A$ is a subdirect subgroup of $X_1\times\dots\times X_t$ with the $X_i$ satisfying the conditions of Lemma~\ref{lem:subdirect}(ii). Hence $F(A)$ is a $p$-group, and therefore also $F(X)$ is a $p$-group. 
	\end{proof}


	\begin{remark}\label{rem:explain}{\rm 
			We make a preliminary remark about how we may exploit Lemma \ref{lem:divisora}(ii) in the proof of Theorem \ref{thm:classification}: suppose that, for a fixed pair $(T_1,M_1)$, we have used Lemmas \ref{lem:factora} and \ref{lem:divisora} to show that $a(G,M,H)$ is a power of a prime $p$, and hence, by Lemma \ref{lem:divisora}(ii) that $F(X)$ is a $p$-group. Suppose that we then consider a pair $(T_j,\mu_j(H))$ where $2\le j\le r$. By changing the labelling, we can replace $(T_1,M_1)$ by $(T_j,\mu_j(H))$, but in doing this we change the overgroup $M$ of $H$, say to $\wt{M}$. It may be possible to use Lemmas \ref{lem:factora} and \ref{lem:divisora} to show that $a(G,\wt{M},H)$ is a power of a prime $s$, with $s\neq p$, whence by Lemma \ref{lem:divisora}(ii), $F(X)$ is an $s$-group, implying that $F(X)=1$. 
			This strategy proves to be very useful. We emphasise that, when we change the labelling, even though the subgroup $M$ may change, the group $X=G^{(2),\Omega}$ and its subgroup $H$ do not change. }
	\end{remark}

	We are now ready to prove Theorem \ref{thm:classification}.
	
	\begin{proof}[Proof of Theorem $\ref{thm:classification}$]
		Given our discussion at the beginning of Section~\ref{DirectProductsSection}, 
		to complete the proof of Theorem \ref{thm:classification}, it is sufficient to prove that each of the groups $G=\prod_{i=1}^rT_i$ satisfying (i) or (ii) of Theorem $\ref{thm:classification}$ is totally $2$-closed.
		
		In pursuit of a contradiction, we assume that one of these groups $G$ is not totally $2$-closed. We then fix $G, H, \Omega, \Sigma, M, R, L, \Delta, s_1,s_2, \rho_{(i,j),(k,\ell)}$, etc., and also $X:=G^{(2), \Omega}$ and $Y=X^{\Delta}_{\Delta}$,  as in
		Notation~$\ref{not:dirmain}$.
		Recall that  $L=L^{(2),\Sigma}$, and also that $X=N\rtimes G$ with $N$ a subdirect subgroup of a base group $B_A=A^s$ for some nontrivial $A\unlhd Y$, so all $A$-orbits in $\Delta$ have the same length $a=a(G,M,H)$. As in the statement of Lemma \ref{lem:factora}, we extend the notation $g(T,M)$ from Lemma \ref{Base2Lemma} to all pairs $(T,M)$ with $T$ a finite simple group and $M< T$, by defining $g(T,M):=1$ if $(T,M)$ is not in Table \ref{tab:basesize}, and  of course $g(T,M)$ is equal to the appropriate entry in the third column of Table \ref{tab:basesize}  if $(T,M)$ does occur. By Lemma~\ref{lem:factora}, $a=a(G,M,H)$ divides 
		$\prod_{i=1}^rg(T_i,\mu_i(H))$.
		
		Relabelling the $T_i$ if necessary, we may without loss of generality assume, by Lemma \ref{lem:Step2}, that the pair $(T_1,\mu_1(H))$ lies in Table \ref{tab:basesize}, and by Lemma \ref{NoPairs}, that $(T_i,\mu_i(H))\neq (\mathbb{M}, 2.\mathbb{B})$ for any $i$. We consider each of the other pairs in Table \ref{tab:basesize}, and show that none of them has the conditions required for $(T_1,\mu_1(H))$. Note that $H$ is a subdirect subgroup of $M=M_1\times M_2$ and that $H_i\unlhd M_i$ for each $i$ (Hypothesis~\ref{hyp:dir}(e)); moreover, by Lemmas \ref{ProductAction}(i) and \ref{lem:subdirect}(i), $M_1/H_1$ is non-trivial.  After some general comments in the next paragraph, we subdivide our arguments according to the structure of $M_1/H_1$, recalling that $M_1=\mu_1(H)>H_1$.
		
		For each  $i\leq r$, since $(T_i,\mu_i(H))$ is not $(\mathbb{M}, 2.\mathbb{B})$, it follows from our comments in the second paragraph of the proof, and from Table \ref{tab:basesize}, that $g(T_i, \mu_i(H))$ has the form $2^{k}3^{\ell}5^m7^n$. Hence by Lemma \ref{lem:factora}, $a$ is a $\{2,3,5,7\}$-number, and then by Lemma \ref{lem:divisora}(i),  $a=d_1d_2$ where $d_1, d_2$ is the order of a normal subgroup of $(M_1/H_1)/Z(M_1/H_1)$ or $Z(M_1/H_1)$ respectively. In particular, since $a>1$, at least one of $(M_1/H_1)/Z(M_1/H_1)$ or $Z(M_1/H_1)$ has a nontrivial normal $\{2,3,5,7\}$-subgroup, say $S$, with $|S|$ dividing $a$.  Furthermore, if $|S|$ has order divisible by $7$ then $7$ divides $a$ and hence, by Lemma~\ref{lem:factora}, $7$ divides $g(T_i, \mu_i(H))$ for some $i$, and this in turn implies, by Table \ref{tab:basesize} that the pair $(T_i,\mu_i(H))=(\mathrm{J}_4,2^{10}.\mathrm{L}_5(2))$.
		
		\vspace{0.25cm}
		\noindent\textit{Claim $1$. $M_1/H_1$ is not almost simple.}\quad Suppose to the contrary that $M_1/H_1$ is almost simple. Inspecting the almost simple quotients $M_1/H_1$ of the groups in the second column of Table \ref{tab:basesize}, we see that the only examples with socle a  $\{2,3,5,7\}$-group have socle $\mathrm{L}_3(4)$, and that $M_1/H_1= \mathrm{L}_3(4)$ or $\mathrm{L}_3(4).2$. In particular $M_1/H_1$ has trivial centre, and so by Lemma~\ref{lem:divisora}(i),  $a = \delta|\mathrm{L}_3(4)|$ with $\delta=1$ or $2$. 
		Since $|\mathrm{L}_3(4)|$ has order divisible by $7$, it follows from the previous paragraph that $(T_i,\mu_i(H))=(\mathrm{J}_4,2^{10}.\mathrm{L}_5(2))$ for some $i\geq 2$. Interchanging $T_1$ and $T_i$ as discussed in Remark~\ref{rem1}, we obtain a new group $\wt{M}=\wt{M}_1\times \wt{M}_2$ with $\wt{M}_1=2^{10}.\mathrm{L}_5(2)$. Lemma~\ref{lem:divisora}(i) now implies that $a = \delta|\mathrm{L}_3(4)|$ is equal to a product of the orders of some composition factors of $\wt{M}_1$, which is a contradiction. This proves Claim 1.

		\vspace{0.25cm}
		\noindent\textit{Claim $2$. $|M_1/H_1|>2$.}\quad Suppose to the contrary that $M_1/H_1\cong C_2$, and 
		recall that $M_2/H_2\cong M_1/H_1$. Thus (see Hypothesis~\ref{hyp:dir} and Lemma~\ref{ProductAction}(i)) $H_1\times H_2\leq H<M=M_1\times M_2$ with 
		$M/(H_1\times H_2)\cong C_2^2$, and $\pi_1(H)=M_1$, $\pi_2(H)=M_2$. It follows that $H$ is the unique index $2$ subgroup of $M$ containing  $H_1\times H_2$ and projecting onto each $M_i$, and in particular $|\Delta|=|M:H|=2$. Moreover, for $i=1, 2$, $G_i$ acts faithfully and imprimitively on $[G_i:H_i]$ preserving a nontrivial partition corresponding to $[G_i:M_i]$ with blocks of size $2=|M_i/H_i|$. Thus the conditions of Notation~\ref{not2} hold for $G_i$. Moreover $G_i$ is totally $2$-closed by Hypothesis~\ref{hyp:dir}(e), and in particular $G_i^{[G_i:H_i]}$ is $2$-closed, and also $G_i^{[G_i:M_i]}$ is $2$-closed (since $M_i$ is core-free in $G_i$ by Lemma~\ref{ProductAction}(i)). Thus the conditions of Proposition~\ref{MainCor1}(1) hold for $G_i$,   and it follows that there exists $k_i$ with $2\le k_i\le s_i$ such that $(M_i\cap M_i^{(k_i)}) H_i\ne M_i$, and hence $M_i\cap M_i^{(k_i)}\le H_i$ (since $|M_i:H_i|=2$). It follows that $M_{(k_1,k_2)}=(M_1\cap M_1^{(k_1)})\times (M_2\cap M_2^{(k_2)})\le H_1\times H_2\le H$, and this implies that $M_{(k_1,k_2)}H=H\neq M$. Since $|\Delta|=2$, this contradicts Proposition \ref{MainCor1}(1) applied to $G$, and Claim 2 is proved. 
		
		
		\vspace{0.25cm}
		In the rest of the proof we deal in turn with each of the remaining possibilities $(T,M)$ in Table~\ref{tab:basesize} for $(T_1, M_1)$. Note that, by Claims 1 and 2, and in the light of Remark~\ref{rem:explain},  if  $(T_i, \mu_i(H))$ occurs in  Table~\ref{tab:basesize}, then $|\mu_i(H)/(\widehat{T}_i\cap H)|>2$ and $\mu_i(H)/(\widehat{T}_i\cap H)$ is not almost simple; and in particular $T_i\ne \mathrm{J}_1$ or $\mathrm{J}_3$, and we have already seen that $T_i\ne\mathbb{M}$.

		\vspace{0.25cm}
		\noindent\textit{Case: $(T_1,M_1)$ with $T_1=\mathrm{Ly}$ and $M_1\in\{G_2(5), 3.\mathrm{McL}, 3.\mathrm{McL}:2\}$.}\quad By the comments above and Table~\ref{tab:basesize},  it follows that $M_1/H_1$ is $3.\mathrm{McL}$ or $3.\mathrm{McL}:2$. Since  $a$ is a $\{2,3,5,7\}$-number and is a product of the orders of some composition factors of $M_1/H_1$ it follows that $a=3$ so $A\le \Sym(3)$. Hence,  by Lemma \ref{lem:divisora}(ii), $F(X)$ is a $3$-group.  Moreover, since $N\unlhd X$ and $N\leq B_A=A^s$ it follows that $F(X)\ne 1$.
		
		Next we note, by Lemma \ref{ProductAction}(ii), that $\mu_k(H)$ has a composition factor isomorphic to $\mathrm{McL}$, for some $k\geq 2$. However the only finite simple group in $\{\mathrm{J}_1,\mathrm{J}_3,\mathrm{J}_4,\mathrm{Th},\mathrm{Ly},\mathbb{M}\}$, apart from $\mathrm{Ly}$, with $\mathrm{McL}$ as a section is $\mathbb{M}$ (see \cite{Atlas}), so $T_k=\mathbb{M}$. Moreover,  inspecting the maximal subgroups of $\mathbb{M}$ shows that the only one with $\mathrm{McL}$ as a section is $2^{1+24}.\mathrm{Co}_1$, and hence  $\mu_k(H)$ is contained in a maximal subgroup $2^{1+24}.\mathrm{Co}_1$ of $\mathbb{M}$. Further, inspecting the subgroups of $\mathrm{Co}_1$ we find that there are exactly two conjugacy classes of subgroups with $\mathrm{McL}$ as a composition factor, and both are second maximal subgroups of $\mathrm{Co}_1$: the classes are represented by the subgroups $\mathrm{McL}<\mathrm{Co}_2<\mathrm{Co}_1$ and $\mathrm{McL}:2<\mathrm{Co}_3<\mathrm{Co}_1$. Thus $\mu_k(H)$ has shape $I.S$, where $I$ is a $2$-group, and $S\in\{\mathrm{McL},\mathrm{McL}:2\}$. 
		If we now relabel the $T_i$ as discussed in Remark \ref{rem:explain} so that $T_1=\mathbb{M}$ and $M_1=I.S$, then by Lemma \ref{lem:divisora}, $a=3$ is a product of the orders of some composition factors of (the new) $M_1/H_1$, which is not the case. This contradiction completes  our analysis of all cases with $T_1=\mathrm{Ly}$.
		
		\vspace{0.25cm}
		Thus, in the light of Remark~\ref{rem:explain},  if  $(T_i, \mu_i(H))$ occurs in  Table~\ref{tab:basesize}, then $T_i\ne \mathbb{M},\mathrm{Ly}, \mathrm{J}_3, \mathrm{J}_1$, and hence $T_i$ is $\mathrm{J}_4$ or $\mathrm{Th}$. Also $\mu_i(H)/(\widehat{T}_i\cap H)$ is not almost simple, and has order at least $3$.

		\vspace{0.25cm}
		\noindent\textit{Case: $(T_1,M_1)=(\mathrm{J}_4,2^{11}:\mathrm{M}_{24})$ or $(\mathrm{J}_4,2^{10}:\mathrm{L}_5(2))$ or $(\mathrm{Th},2^{5}.\mathrm{L}_5(2))$.}\quad Suppose first that $T_1=\mathrm{J}_4$, so $M_1=O_2(M_1).S$ with $S=\mathrm{M}_{24}$ or $\mathrm{L}_5(2)$. Then $g(T_1,M_1)$ divides $2^4.3^2.5.7$ by Table~\ref{tab:basesize}.
		Now, for each $i\geq2$, if $(T_i, \mu_i(H))$ does not occur in  Table~\ref{tab:basesize}, then $g(T_i,\mu_i(H))=1$, while if it does occur, then, as we noted above, $T_i=\mathrm{Th}$ (since $T_1$ is the only occurrence of $\mathrm{J}_4$), and so by  Table~\ref{tab:basesize}, $g(T_i,\mu_i(H))$ divides $6$. Thus  
		by Lemma \ref{lem:factora} the integer $a$ divides $2^5.3^3.5.7$. However, since $M_1=O_2(M_1).S$ acts irreducibly on the elementary abelian group $O_2(M_1)$, it follows from Lemma \ref{lem:divisora}(i) that $a$ is divisible by $|O_2(M_1)|$ or $|S|$, which is a contradiction. Thus $(T_i,\mu_i(H))$ is not $(\mathrm{J}_4,2^{11}:\mathrm{M}_{24})$ or $(\mathrm{J}_4,2^{10}:\mathrm{L}_5(2))$ for any $i$.
		
		In the light of the results above, an analogous argument with  $T_1=\mathrm{Th}$ and $M_1=O_2(M_1).\mathrm{L}_5(2)=2^5.\mathrm{L}_5(2)$ gives $g(T_1,M_1)=1$ from  Table~\ref{tab:basesize}, and for each $i\geq2$, if $g(T_i,\mu_i(H))>1$ then, by  Table~\ref{tab:basesize}, $g(T_i,\mu_i(H))=2$ and $T_i=\mathrm{J}_4$ (and there is at most one such $i$). Hence by Lemma \ref{lem:factora}, the integer $a$ divides $2$, so $a=2$. However, since $M_1$ acts irreducibly on $O_2(M_1)$, it follows from Lemma \ref{lem:divisora}(i) that $a$ is divisible by $|\mathrm{L}_5(2)|$ or $2^5$, which is a contradiction. Thus, for each $i\geq 1$, $(T_i,\mu_i(H))$ is not one of the three pairs in this case.

		\vspace{0.25cm}
		\noindent\textit{Case: $(T_1,M_1)=(\mathrm{Th},{}^3D_4(2))$ or $(\mathrm{Th},{}^3D_4(2).3)$.}\quad By the previous case and Claim 1, the only possibility is $M_1={}^3D_4(2).3$ and $H_1={}^3D_4(2)$.  Thus (see Hypothesis~\ref{hyp:dir} and Lemma~\ref{ProductAction}(i)),  $M_1/H_1\cong M_2/H_2\cong C_3$ and $H_1\times H_2\leq H<M=M_1\times M_2$ with 
		$M/(H_1\times H_2)\cong C_3^2$, $\pi_1(H)=M_1$, and $\pi_2(H)=M_2$. Hence $|\Delta|=|M:H|=3$, $H\unlhd M$, and also $a=3$ (see Proposition~\ref{MainCor1}(2)). 
		
		Suppose that, for some $i\geq2$,  $(T_i, \mu_i(H))$ occurs in  Table~\ref{tab:basesize}, so as observed above, $T_i=\mathrm{J}_4$ (since $T_1$ is the only occurrence of $\mathrm{Th}$). Then by the discussion in Remark~\ref{rem:explain} and Lemma~\ref{lem:divisora}(i), $a=3$ is a product of the orders of some composition factors of $\mu_i(H)/(\widehat{T}_i\cap H)$, and it follows from Table~\ref{tab:basesize} that $\mu_i(H)=2^{1+12}.3.R$ for some $R\in\{ \mathrm{M}_{22}, \mathrm{M}_{22}.2,\mathrm{L}_3(4), \mathrm{L}_3(4).2\}$.   Moreover in this case $g(T_i,\mu_i(H))=2$ by Table~\ref{tab:basesize}. Now at most one of the $T_i$ is equal to $\mathrm{J}_4$, and for all other $i\geq 2$, $T_i^{[T_i:\mu_i(H)]}$ has base size $2$.  
		
		By \cite[Lemma 2.20.1]{Ivanov95}, there exists $t_1\in T_1$ such that $M_1\cap M_1^{t_1}\cong\Sym(3)$. For each $i\geq 2$, if $T_i^{[T_i:\mu_i(H)]}$ has base size $2$, then choose $t_i\in T_i$ such that $\mu_i(H)\cap\mu_i(H)^{t_i}=1$; while if $T_i=\mathrm{J}_4$ and 
		$\mu_i(H)=2^{1+12}.3.R$ with $R\in\{ \mathrm{M}_{22}, \mathrm{M}_{22}.2,\mathrm{L}_3(4), \mathrm{L}_3(4).2\}$, then we may choose $t_i\in T_i$ such that $\mu_i(H)\cap\mu_i(H)^{t_i}$ has order dividing $16$ (see \cite[last line of Table 2]{BrayCameron}). Let $g_1=t_1\in G_1$, and $g_2=(t_2,\hdots,t_r)\in G_2$, and set $g:=(g_1,g_2)\in G$.  
		Then since $M=M_1\times M_2$, we see that $M\cap M^g=(M_1\cap M_1^{g_1})\times (M_2\cap M_2^{g_2})$. From the definition of the $t_i$ we see that $M_2\cap M_2^{g_2}$ is a $2$-group, and $M_1\cap M_1^{g_1}\cong \Sym(3)$. It follows that $M\cap M^g$ has no quotient of order $3$. Thus, $(M\cap M^g)H\neq M$, since $H\unlhd M$ and $M/H\cong C_3$. This however contradicts Proposition \ref{MainCor1}(4). 
		
		\vspace{0.25cm}
		Thus we now know that,  if  $(T_i, \mu_i(H))$ occurs in  Table~\ref{tab:basesize}, then $T_i= \mathrm{J}_4$ and $\mu_i(H)=2^{1+12}.3.R$ for some $R\in\{ \mathrm{M}_{22}, \mathrm{M}_{22}.2,\mathrm{L}_3(4), \mathrm{L}_3(4).2\}$. In particular these are the only possibilities for $(T_1, M_1)$, and we complete the proof by showing that none of these four pairs is possible.

		\vspace{0.25cm}
		\noindent\textit{Case: $(T_1,M_1)=(\mathrm{J}_4,2^{1+12}.3.R)$ with $R\in\{ \mathrm{M}_{22}, \mathrm{M}_{22}.2,\mathrm{L}_3(4), \mathrm{L}_3(4).2\}$.}\quad  In each of these cases, $g(T_1,M_1)=2$, by Table~\ref{tab:basesize}. Moreover, for each $i\geq2$,  $T_i^{[T_i:\mu_i(H)]}$ has base size $2$ and $g(T_i,\mu_i(H))=1$ (since now the only case where $(T_j, \mu_j(H))$ occurs in  Table~\ref{tab:basesize} is for $j=1$). Thus $a=2$ by Lemma \ref{lem:factora}. 
		
		We claim that $H_1=1$.  
		Note that, by Claim 2, $|M_1/H_1|>2$. Moreover, by Lemma~\ref{lem:divisora}(i), $a=|U|\cdot |V|$ for some $U\unlhd (M_1/H_1)/Z(M_1/H_1)$ and $V\leq Z(M_1/H_1)$. Since $a=2$, inspecting the possibilities for $M_1/H_1$ from the groups $M_1$ listed,  we see that the only possibility for $H_1$ in each case is $H_1=1$ (with $U=1$ and $V\cong Z(M_1)$), proving the claim.    
		
		Since $H\leq M_1\times M_2=M$, the subgroup $H$ centralises $Z(M_1)$. Consider the subgroup  $K:=HZ(M_1)$ of $M$. The condition $H_1=1$ implies that $K$ properly contains $H$. Moreover since, by Hypothesis~\ref{hyp:dir}(c), $H$ is maximal by inclusion such that $G^{[G:H]}$ is not $2$-closed, it follows that  $G^{[G:K]}$ is $2$-closed.
		Since $M_1$ is core-free in $T_1$ there exists $t_1\in T_1$ such that $M_1\cap M_1^{t_1}$ does not contain $Z(M_1)$. Also, for each $i\ge 2$,  $T_i^{[T_i:\mu_i(H)]}$ has base size $2$, and so there exists  $t_i\in T_i$ such that $\mu_i(H)\cap\mu_i(H)^{t_i}=1$. Set $g_1:=t_1$, $g_2:=(t_2,\hdots,t_r)$, and $g:=(g_1,g_2)\in G$. Then $\mu_i(K\cap K^{g})=1$ for all $i\geq 2$, and hence $K\cap K^g \le K\cap \widehat{T}_1=Z(M_1)\cong 2$. On the other hand, $\pi_1(K\cap K^g)\le \pi_1(M\cap M^g)=\pi_1(M_1\cap M_1^{t_1})=M_1\cap M_1^{t_1}$, and $M_1\cap M_1^{t_1}$ does not contain $Z(M_1)$. Thus $K\cap K^g=1$, so $G^{[G:K]}$ has base size $2$. It now follows from Proposition \ref{MainCor1}(3) that  $G^{[G:H]}$ is $2$-closed. This final contradiction completes the proof of Theorem~\ref{thm:classification}.
	\end{proof}

\section{Concluding remarks: the general case}\label{sec:concluding}
As a result of the work in this paper, the task of classifying all finite totally $2$-closed groups $G$ can now be reduced to the case where $G$ is insoluble and $F(G)>1$ (see Question 1 in Section \ref{s:graph}). We conclude the paper with some remarks about this case, and a reduction theorem.

We note first that a crucial step in our classification was the essential reduction to classifying those finite groups (with trivial Fitting subgroup) which are $2$-closed in all of their transitive permutation representations (see Proposition \ref{PreTransRed}). 
Although we weren't able to reduce our study of arbitrary finite totally $2$-closed groups to transitive actions, the
following theorem (and subsequent corollary) significantly reduces the types of faithful actions that need to be examined. In order to state it, we need some additional terminology: Permutation representations of a group $G$ on sets $\Gamma$ and $\Gamma'$ are said to be \emph{permutationally equivalent} if there exists a bijection $f:\Gamma\to \Gamma'$ such that, for all $\gamma\in\Gamma$ and all $g\in G$, the image $(\gamma^g)f=(\gamma f)^g$. Note that $f$ induces a bijection $\Gamma\times\Gamma\to\Gamma'\times\Gamma'$, namely $f:(\alpha,\beta)\to (\alpha f,\beta f)$, and that this induced map sends $G$-orbits in $\Gamma\times\Gamma$ to $G$-orbits in $\Gamma'\times\Gamma'$.

Our general reduction theorem can now be stated as follows.
\begin{theorem}\label{thm:permeq}
Let $G\leq \Sym(\Omega)$, and suppose that $\Gamma, \Gamma'$ are distinct $G$-orbits on which the $G$-actions are permutationally equivalent relative to a bijection $f:\Gamma\to\Gamma'$. Let $\Omega':=\Omega\setminus\Gamma$, so $\Omega=\Omega'\cup\Gamma$. Then
\begin{enumerate}
\item[\upshape(a)] both $G$ and  $G^{(2),\Omega}$ act faithfully on $\Omega'$, and
\item[\upshape(b)] the restriction  $G^{(2),\Omega}|_{\Omega'} = G^{(2),\Omega'}$. 
\item[\upshape(c)] If $G$ is finite, then  $G=G^{(2),\Omega}$ if and only if $G^{\Omega'}
=G^{(2),\Omega'}$.
\end{enumerate}
\end{theorem}

\begin{proof}
(a) Note first that since $G^{(2),\Omega}$ has the same orbits as $G$ in $\Omega$, the set $\Omega'$ is $G^{(2),\Omega}$-invariant. Also, $G\leq G^{(2),\Omega}$. To prove (a), it is therefore sufficient to show that $G^{(2),\Omega}$ is faithful on $\Omega'$. So assume that $x\in G^{(2),\Omega}$ such that $x$ fixes $\Omega'$ pointwise, and let $\gamma\in\Gamma$. Then $\gamma f\in \Gamma'\subseteq \Omega'$, so $(\gamma f)^x=\gamma f$. 
By Theorem~\ref{thm:W}, there exists $g\in G$ such that $\gamma^x=\gamma^g$ and $(\gamma f)^x=(\gamma f)^g$. Thus $\gamma f=(\gamma f)^g$, and by the defining property of $f$, $(\gamma f)^g=(\gamma^g)f$. This implies that $\gamma f=(\gamma^g)f = (\gamma^x)f$, and since $f$ is a bijection we deduce that $\gamma = \gamma^x$. It follows that $x$ fixes $\Gamma$ pointwise, and hence that $x=1$. This proves part (a).

(b) Let $g\in G^{(2),\Omega}$. Then $g$ leaves invariant each $G$-orbit in $\Omega\times \Omega$. In particular, $g^{\Omega'}$ leaves invariant each $G$-orbit in $\Omega'\times \Omega'$ and hence (see \cite[Remark 4.2]{W}) $g^{\Omega'}\in G^{(2),\Omega'}$. Thus 
 $G^{(2),\Omega}|_{\Omega'} \leq  G^{(2),\Omega'}$.
 To prove the reverse inclusion let $x\in G^{(2),\Omega'}$. We extend the definition of $x$ to a permutation $y\in\Sym(\Omega)$ by defining $\gamma^y=\gamma^x$ if $\gamma\in\Omega'$, and  $\gamma^y=((\gamma f)^x)f^{-1}$ if $\gamma\in\Gamma$. Since $f: \Gamma\to\Gamma'$ is a bijection, $y$ is a permutation of $\Omega$. 
 
 We claim that $y\in  G^{(2),\Omega}$. Note that this claim implies that each element of $G^{(2),\Omega'}$ lies in
 $G^{(2),\Omega}|_{\Omega'}$, and hence a proof of the claim completes the proof of part (b). By definition, $y$ leaves invariant all the $G$-orbits in $\Omega'\times \Omega'$, so we need to consider $G$-orbits $\Delta$ contained in $\Gamma_1\times \Gamma_2$ where the $\Gamma_i$ are $G$-orbits in $\Omega$ and at least one is equal to $\Gamma$. 
 
Suppose first that $\Gamma_1=\Gamma_2=\Gamma$ and define $\Delta'=\{(\alpha f, \beta f) \mid (\alpha,\beta)\in\Delta\}$.   To see that $\Delta'$ is $G$-invariant let $g\in G$ and 
$(\alpha f, \beta f)\in\Delta'$. Then, using the defining property of $f$,
\[
(\alpha f, \beta f)^g = ((\alpha f)^g, (\beta f)^g) = (\alpha^g f, \beta^g f)
\]
which lies in $\Delta'$ since $\Delta$ is $G$-invariant. A similar argument shows that $\Delta'$ is a $G$-orbit in $\Gamma'\times\Gamma'\subseteq \Omega'\times\Omega'$. 
Thus $y$ leaves $\Delta'$ invariant (since $y^{\Omega'}=x\in G^{(2),\Omega'}$). 
Now take an arbitrary element $(\alpha,\beta)\in \Delta$. Then by the definition of $y$, $(\alpha,\beta)^y=((\alpha f)^xf^{-1}, (\beta f)^xf^{-1})$ and we have to prove that this lies in $\Delta$. 
By the definition of $\Delta'$, $(\alpha f,\beta f)\in\Delta'$, and since $y$ leaves $\Delta'$ invariant and $y^{\Omega'}=x$, $\Delta'$ contains $(\alpha f,\beta f)^y=((\alpha f)^x,(\beta f)^x)$. 
Applying Theorem~\ref{thm:W} to $x\in G^{(2),\Omega'}$, we see that there exists $h\in G^{\Omega'}$ such
that  $(\alpha f)^x = (\alpha f)^h$ and $(\beta f)^x=(\beta f)^h$, and by part (a), there exists a unique element $g\in G$ such that $g^{\Omega'}=h$, and so 
$(\alpha f)^x = (\alpha f)^g$ and $(\beta f)^x=(\beta f)^g$.
Thus we have
\[
(\alpha f)^xf^{-1} = (\alpha f)^g f^{-1} = (\alpha^g)f f^{-1}=\alpha^g,
\]
and similarly $(\beta f)^xf^{-1}=\beta^g$. Thus  
\[
(\alpha,\beta)^y=((\alpha f)^xf^{-1}, (\beta f)^xf^{-1})=(\alpha^g,\beta^g)=(\alpha,\beta)^g 
\]
which lies in $\Delta$ since $\Delta$ is $G$-invariant. It follows that $y$ leaves invariant each $G$-orbit $\Delta\subseteq \Gamma\times\Gamma$.

To deal with the remaining $G$-orbits we assume that $\Gamma_1\ne \Gamma_2=\Gamma$ (since the case 
 $\Gamma_1=\Gamma\ne \Gamma_2$ is similar). This time we define $\Delta'=\{(\alpha, \beta f) \mid (\alpha,\beta)\in\Delta\}$. A similar argument shows that $\Delta'$ is a $G$-orbit in $\Omega'\times\Omega'$ and hence is left invariant by $y^{\Omega'}=x$. For  an arbitrary pair $(\alpha,\beta)\in \Delta$, we have $(\alpha,\beta)^y=(\alpha^x, (\beta f)^xf^{-1})$ and we have to prove that this lies in $\Delta$. 
 By  Theorem~\ref{thm:W} applied to $x\in G^{(2),\Omega'}$, there exists $g\in G$ such
that  $\alpha^x = \alpha^g$ and $(\beta f)^x=(\beta f)^g$, and by the property of $f$, 
$(\beta f)^g=(\beta^g)f$. Thus
\[
(\alpha,\beta)^y=(\alpha^x, (\beta f)^xf^{-1})=(\alpha^g, (\beta^g)ff^{-1}) = (\alpha^g, \beta^g)=(\alpha, \beta
)^g\in \Delta,
\]
since $\Delta$ is $G$-invariant. Thus $y$ leaves $\Delta$ invariant in this case also. It follows that $y$ leaves invariant each $G$-orbit in $\Omega\times\Omega$ and hence $y\in  G^{(2),\Omega}$. Part (b) follows.

(c) Now assume that $G$ is finite. Suppose that $G=G^{(2),\Omega}$. By part (a), 
$G\cong G^{\Omega'}$ and  $G^{(2),\Omega}\cong G^{(2),\Omega}|_{\Omega'}$, and by part (b) the latter group is  $G^{(2),\Omega'}$. Thus, $G^{\Omega'}\cong G^{(2),\Omega'}$, and since $G$ is finite it follows that $G^{\Omega'}=G^{(2),\Omega'}$.
Conversely suppose that $G^{\Omega'}=G^{(2),\Omega'}$. By part (b), $G^{\Omega'}=G^{(2),\Omega}|_{\Omega'}$, and by part (a),  we have $G\cong G^{(2),\Omega}$, and since $G$ is finite equality follows: $G= G^{(2),\Omega}$.
\end{proof}

\begin{corollary}\label{cor:permeq}
Let $G$ be a finite group. Then $G$ is totally $2$-closed if and only if $G= G^{(2),\Omega}$, for each faithful $G$-action on a set $\Omega$ such that, for distinct $G$-orbits $\Gamma, \Gamma'$ in $\Omega$, the $G$-actions on $\Gamma, \Gamma'$ are not permutationally equivalent.
\end{corollary}

\begin{proof}
By definition, $G$ is  totally $2$-closed if and only if $G= G^{(2),\Omega}$, for each faithful $G$-action on a set $\Omega$. If there exist distinct $G$-orbits $\Gamma, \Gamma'$ in $\Omega$ such that the $G$-actions on $\Gamma, \Gamma'$ are permutationally equivalent, then by Theorem~\ref{thm:permeq}, $G$ acts faithfully on $\Omega':=\Omega\setminus\Gamma$ and the $G^{(2),\Omega}$ is isomorphic to $G^{(2),\Omega'}$. Since $G$ is finite, it follows that $G^\Omega$ is $2$-closed if and only if $G^{\Omega'}$ is $2$-closed. Thus to check that $G$ is totally $2$-closed we only need to check $2$-closure for faithful actions in which the $G$-actions on distinct orbits are not permutationally equivalent. 
\end{proof}

The significance of Corollary~\ref{cor:permeq} is that in order to decide whether or not a given finite group $G$ is totally $2$-closed, we no longer have to concern ourselves with studying all faithful permutation representations of $G$ (of which there are infinitely many). We can now determine the total $2$-closure or non total $2$-closure of $G$ by deciding whether or not $G$ is $2$-closed in all of its faithful permutation representations in which no two distinct $G$-orbits admit permutationally equivalent actions.  Since the number of permutational equivalence classes of transitive $G$-actions is equal to the number $c(G)$ of $G$-conjugacy classes of subgroups, there are finitely many (at most $2^{c(G)}-1$) such representations. 

We conclude the paper by making a few remarks about how to check the total $2$-closure of a given ``small" finite insoluble group $G$ with non-trivial Fitting subgroup. One of the quickest ways to prove that such a group $G$ is \textbf{not} totally $2$-closed is to use Lemma \ref{lem:AAT31}: if $G$ has a factorization $G=HK$ with $\core_G(H)\cap \core_G(K)=1$ and $G\neq \core_G(H)\times\core_G(K)$, then $G$ is not totally $2$-closed. As an example, if $G=3^.\Alt(6)$ is a triple cover of $\Alt(6)$, then $G$ has subgroups $H$ and $K$ with $H\cong \Alt(5)$ and $K\cong \mathrm{GL}_2(4)$ satisfying the above properties. (We remark that Lemma \ref{lem:AAT31} is deduced from (the much stronger) Theorem \ref{thm:W2}, so Theorem \ref{thm:W2} can also be a highly useful tool in proving that a group cannot be totally $2$-closed.)

If no such factorization exists (or more generally, if Theorem \ref{thm:W2} is not available), then the question of whether or not $G$ is totally $2$-closed becomes more difficult. Although one can use direct computation and Corollary \ref{cor:permeq} when $|G|$ is small enough, the number of faithful permutation representations coming from the corollary, and the degrees involved, get large quite quickly. Thus, checking the $2$-closure of the relevant permutation representations becomes computationally heavy, even for relatively small $|G|$. For instance, if $G$ is a finite insoluble group with non-trivial Fitting subgroup, and if either $|G|<1440$, or $|G|=1440$ and $G$ is not isomorphic to the group $\texttt{SmallGroup}(1440,4594)$ (which is a double cover $2^.\mathrm{PGL}_2(9)$ of $\mathrm{PGL}_2(9)$) in the GAP small groups library \cite{GAP}, then one can directly compute that $G$ has a factorization $G=HK$ with $\core_G(H)\cap\core_G(K)=1$, but $G\neq \core_G(H)\times\core_G(K)$. 
It then follows from Lemma \ref{lem:AAT31} that no such group is totally $2$-closed. The group $G:=\texttt{SmallGroup}(1440,4594)$ has no factorization $G=HK$ with $\core_G(H)\cap\core_G(K)=1$, however, and so, although one can check that $G$ is $2$-closed in all of its faithful transitive permutation representations, the question of whether or not $G$ 
is totally $2$-closed remains open. (There are $30$ conjugacy classes of subgroups of $G$ in this case, with subgroups in $26$ of them containing $Z(G)\cong 2$, and subgroups in the other four classes being core-free. Thus, if $G$ is totally $2$-closed, then proving this fact using Corollary \ref{cor:permeq} would require one to check the total $2$-closure of $2^{30}-2^{26}=1006632960$ permutation representations, of increasingly large degree.)

\end{document}